\documentclass[11pt]{article}

\usepackage{graphicx}
\graphicspath{{figs/}{tikz/}{placeholder-figs/}}
\usepackage{xcolor}
\usepackage{enumerate}
\usepackage{color}
\usepackage{tcolorbox}
\definecolor{darkblue}{rgb}{0.0,0.0,0.2}
\definecolor{gray}{gray}{0.5}
\definecolor{lightgray}{gray}{0.9}
\definecolor{lightred}{rgb}{1,0.6,0.6}
\definecolor{darkgreen}{rgb}{0,0.5,0}
\definecolor{DARKGREEN}{rgb}{0,0.5,0} 
\definecolor{BLUE}{rgb}{0,0,0.5} 
\definecolor{WHITE}{rgb}{1,1,1} 
\definecolor{BLACK}{rgb}{0,0,0} 
\usepackage[colorlinks=true,breaklinks=true,bookmarks=true,urlcolor=darkblue,
citecolor=darkblue,linkcolor=darkblue,bookmarksopen=false,draft=false]{hyperref}

\usepackage{pgfplots}
\pgfplotsset{compat=1.17}
\usepackage{tikz} 
\usetikzlibrary{matrix, shapes, fit, arrows.meta, decorations.pathreplacing, calc, positioning} 

\usepackage{pifont}

\newcommand{\Comments}{0}
\newcommand{\ifgeq}[3]{\ifnum#1>#2 {#3}\else\ifnum#1=#2 {#3}\fi\fi}
\usepackage[colorinlistoftodos,textsize=tiny]{todonotes}

\ifnum\Comments>0               
\usepackage[left=0.7in,right=2.1in,top=1in,bottom=1in]{geometry}
\else
\usepackage[top=1in,bottom=1in,left=1.1in,right=1.1in]{geometry}
\fi

\usepackage[font={small}]{caption}
\usepackage[numbers,sort,compress]{natbib}
\newcommand{\citetsafe}[2]{\citet[#1]{#2}}
\usepackage[shortlabels]{enumitem}

\usepackage{etoolbox,xparse,xspace,etex}
\usepackage{array,multirow,booktabs} 
\usepackage{amsmath,amssymb}
\usepackage{centernot}
\usepackage{mathtools}
\newcommand\numberthis{\addtocounter{equation}{1}\tag{\theequation}}
\newcommand{\forceindent}{\leavevmode{\parindent=2em\indent}}

\usepackage[ruled,vlined,resetcount]{algorithm2e}
\renewcommand{\algorithmcfname}{Protocol}
\newcommand{\protocol}{\algorithmcfname\xspace}
\SetNlSty{}{}{}

\usepackage{amsthm}

\newtheorem{theorem}{Theorem}[section]
\newtheorem{proposition}[theorem]{Proposition}
\newtheorem{nontheorem}[theorem]{Non-Theorem}

\newtheorem{lemma}[theorem]{Lemma}
\newtheorem{corollary}[theorem]{Corollary}

\newtheorem{condition}[theorem]{Condition}

\newtheorem{definition}[theorem]{Definition}
\theoremstyle{definition}
\newtheorem{remark}[theorem]{Remark}
\newtheorem{example}[theorem]{Example}

\let\oldS\S
\renewcommand{\S}{\oldS\!}

\let\emptyset\varnothing

\newcommand{\A}{\mathcal{A}}
\newcommand{\B}{\mathcal{B}}
\newcommand{\F}{\mathcal{F}}
\newcommand{\G}{\mathcal{G}}

\newcommand{\N}{\mathbb{N}}
\renewcommand{\P}{\mathcal{P}}
\newcommand{\Sc}{\mathcal{S}}

\newcommand{\U}{\mathcal{U}}
\newcommand{\W}{\mathcal{W}}
\newcommand{\X}{\mathcal{X}}
\newcommand{\Y}{\mathcal{Y}}
\newcommand{\Z}{\mathcal{Z}}

\newcommand{\Ht}{\ensuremath{\mathsf{H}}\xspace}
\newcommand{\Tt}{\ensuremath{\mathsf{T}}\xspace}

\newcommand{\dcl}{\mathrm{dcl}}
\newcommand{\cdcl}[1]{\overline{\dcl(#1)}}
\newcommand{\mdcl}{\mathrm{mdcl}}
\newcommand{\bb}{\mathrm{bb}}
\newcommand{\ba}{\mathrm{ba}}

\newcommand{\conv}{\mathrm{conv}}

\usepackage{dsfont}
\newcommand{\ones}{\mathds{1}}

\newcommand{\reals}{\mathbb{R}}

\newcommand{\extreals}{\overline{\mathbb{R}}}
\newcommand{\infreals}{{\reals\cup\{\infty\}}}

\newcommand{\argmin}{\mathop{\mathrm{argmin}}}
\newcommand{\argmax}{\mathop{\mathrm{argmax}}}

\newcommand{\argsup}{\mathop{\mathrm{argsup}}}

\newcommand{\inprod}[2]{\left\langle #1, #2 \right\rangle}

\newcommand{\Eg}{\textsf{\textbf{E}}}
\newcommand{\Egu}{\overline{\Eg}}
\newcommand{\Egl}{\underline{\Eg}}
\newcommand{\Pg}{\textsf{\textbf{P}}}
\newcommand{\Pgu}{\overline{\Pg}}
\newcommand{\Pgl}{\underline{\Pg}}

\newcommand{\Ep}{\mathbb{E}}
\newcommand{\E}{\Ep}
\newcommand{\Epu}{\overline{\Ep}}
\newcommand{\Epl}{\underline{\Ep}}
\newcommand{\Pp}{\mathbb{P}}

\newcommand{\Epucons}{\Epu^0}
\newcommand{\Epuseq}{\Epu^*}
\newcommand{\Epseq}{\Ep^*}

\newcommand{\Evar}{\mathcal{E}}

\newcommand{\emptystring}{{\boldsymbol{\varepsilon}}}

\newcommand{\ext}{\mathrm{ext}}
\newcommand{\res}{\mathrm{res}}

\newcommand{\Reg}{\mathrm{Reg}}
\newcommand{\Alg}{\mathrm{Alg}}
\newcommand{\Rel}{\mathrm{Rel}}
\newcommand{\Eop}{\mathop{\E}}

\newcommand{\ALLN}{A_\text{\scalebox{0.6}{LLN}}}

\title{Minimax Duality in Game-Theoretic Probability\thanks{Work in progress: mistakes are likely, and substantial revisions may follow.  Feedback is welcomed!}}
\author{Rafael Frongillo\thanks{Department of Computer Science, University of Colorado Boulder, USA. Email: raf@colorado.edu.}}

\begin{document}

\maketitle

\begin{abstract}
  Game-theoretic probability uses the structure of gambles to define a concept like probability, but which is more flexible and robust.
  We show that results in game-theoretic probability can be thought of as minimax theorems for specific zero-sum games between two players, Gambler and World.
  The traditional measure-theoretic versions arise when World must play first.
  This perspective suggests the possibility of a more general minimax theorem from which a wide array of game-theoretic results would follow.
  After developing a new framing of game-theoretic probability via gamble spaces, we prove such a theorem for finite time.
  Applying this minimax theorem to games derived from existing measure-theoretic statements, we prove several existing and novel game-theoretic statements.
  This general minimax theorem can be thought of as a composite Ville's theorem, as we discuss along with future directions.
\end{abstract}

\newpage
\tableofcontents
\newpage

\section{Introduction}
\label{sec:introduction}

The origins of the field of probability are often traced to a conversation between Pascal and Fermat in 1654~\citep{hald2005history}, on how to divide the prize in a contest that ends prematurely (Fig.~\ref{fig:pascal-fermat}).
Pascal proposed a ``game-theoretic'' approach to calculate the division, reasoning about the stakes of a gamble on the eventual winner in terms of the stakes for gambles on each individual match.
In response, Fermat proposed a ``measure-theoretic'' approach, wherein one enumerates the set of possible joint outcomes of all matches, and then calculates the prize division combinatorially.

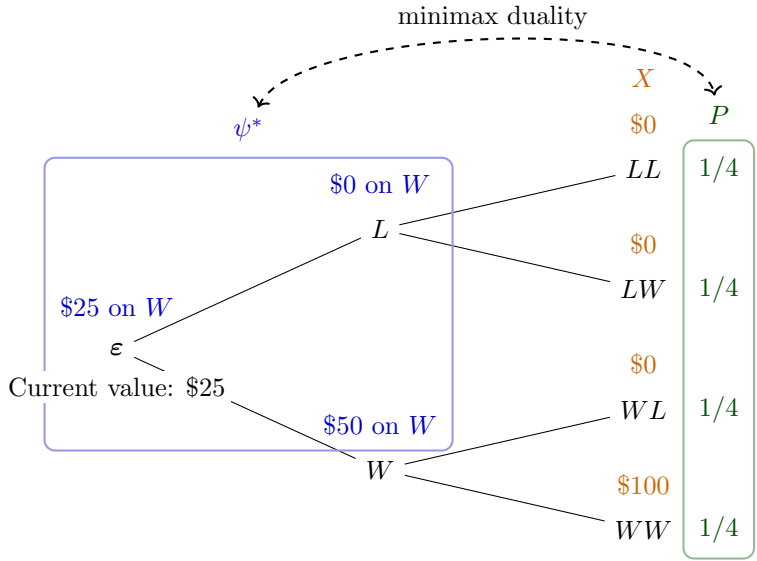
\begin{figure}[p]
  \centering
  \colorlet{gamblercolor}{blue!80!black}
  \colorlet{naturecolor}{green!30!black}
  \colorlet{Xcolor}{orange!80!black}

  \begin{tikzpicture}[yscale=0.8, grow=right, level distance=3.5cm,
    level 1/.style={sibling distance=4cm},
    level 2/.style={sibling distance=2cm},
    level 3/.style={sibling distance=1cm},
    every node/.style={font=\small}]
    
    \node (root) {$\emptystring$}
    child {node (T) {$W$}
      child {node (TT) {$WW$}}
      child {node (TH) {$WL$}}
    }
    child {node (H) {$L$}
      child {node (HT) {$LW$}}
      child {node (HH) {$LL$}}
    };
    
    \node[color=gamblercolor] (eps-bet) [above= 2pt of root] {\$25 on $W$};
    \node[color=gamblercolor] (H-bet)   [above= 2pt of H]    {\$0 on $W$};
    \node[color=gamblercolor] (T-bet)   [above= 2pt of T]    {\$50 on $W$};

    \foreach \name/\val in {TT/\$100, TH/\$0, HT/\$0, HH/\$0} {
      \node (\name-val) [color=Xcolor, above=2pt of \name] {\val};
      \node (\name-prob) [color=naturecolor, right=1.4cm of \name.center, anchor=east] {1/4};
    }

    \node[color=Xcolor] (X) [above= 2pt of HH-val] {$X$};

    \node[draw=gamblercolor!40, thick, rounded corners, inner sep=2pt,
    fit=(eps-bet)(H-bet)(T-bet)] (gamblerbubble) {};
    \node[draw=naturecolor!40, thick, rounded corners, inner sep=2pt,
    fit=(TT-prob)(TH-prob)(HT-prob)(HH-prob)] (naturebubble) {};

    \node (psi)  [color=gamblercolor, above=2pt of gamblerbubble.north] {$\psi^*$};
    \node (Pnode)[color=naturecolor,  above=2pt of naturebubble.north]  {$P$};

    \draw[<->, thick, dashed]
    (psi) to[out=70, in=100, looseness=.6]
    node[midway, above]{minimax duality}
    (Pnode);

    \node (curr-val) [below= 2pt of root, fill=white, inner sep=2pt] {Current value: \$25};
  \end{tikzpicture}
  
  \caption{Example derivations from the famous conversation between Pascal and Fermat.
    \\
    \forceindent
    Depicted is a best-of-3 contest between Alexa and Calen, with a prize of \$100, where Alexa already lost the first match.
    Thus the only way for Alexa to win the prize is to win both remaining games ($WW$).
    Pascal reasoned that were Alexa to win the next game, she would now be on even footing with Calen, and it would be natural for them to earn an even split of the prize.
    Thus the amount owed to Alexa in state $W$ should be \$50.
    As Alexa cannot win the prize after losing the next match (and is thus owed \$0 in state $L$), similar reasoning suggests that the amount owed Alexa now is \$25.
    Fermat instead proposed dividing the money according to the possible ways for Alexa to win the prize, in this case 1 out of 4, also concluding that Alexa is owed \$25 now.
    \\
    \forceindent
    These two derivations can be seen as the two optimal strategies in a zero-sum game---not the one between Alexa and Calen, but between two external players, Gambler and World, which respectively bet on and choose the match outcomes.
    Gambler chooses a gambling strategy $\psi$, which places a bet (under even odds) on the next match based on the history so far.
    World chooses a probability measure $P$ on the sequence of outcomes $\{W,L\}^2$ from the initial situation.
    Let $X:\{W,L\}^2\to\reals$ encode the value of the final outcome to Alexa, namely $X(WW) = 100$, $X(WL) = X(LW) = X(LL) = 0$.
    The payoff to World in the game, $u(P,\psi)$, is the expected difference between $X$ and the winnings of Gambler.
    Specifically, identifying $\{W,L\}$ with $\{-1,1\}$ for convenience, we may write $u(P,\psi) = \E_P[X(Y_1Y_2) - \psi(\emptystring)Y_1 - \psi(Y_1)Y_2]$, where $Y_1,Y_2 \in \{-1,1\}$ represent the outcomes.
    \\
    \forceindent
    When Gambler must play first, the optimal strategy $\psi^*$, depicted in blue, bets \$25 on $W$ in the initial situation, and either refrains from betting if Alexa loses or bets another \$50 on $W$ if she wins.
    This strategy is exactly what Pascal derived by backward induction.
    (More precisely, it is the amount staked in each state, which in this case is equal to the value owed at that state.)
    On the other hand, when World must play first, the optimal strategy is the uniform distribution $P$ on $\{W,L\}^2$, which is exactly Fermat's approach.
    Both optimal strategies are unique, and as minimax duality holds in this game, both give the same payoff of \$25 to World.
    The conclusion in both cases is again that Alexa's standing $X$ is ``worth'' \$25.
    \\
    \forceindent
    As a final point, imagine that the contest was best-of-4, and Alexa and Calen had each already won a game.
    If we split the prize upon a tie, we have $X = (\$0,\$50,\$50,\$100)$ top to bottom.
    Interestingly, Fermat's weights do not change: we still have $P = (1/4,1/4,1/4,1/4)$.
    But Pascal's derivation would change, starting with \$50 but placing no bet, and regardless of the first match outcome, placing a \$25 bet on $W$ in the next round.
    In this sense, the game-theoretic approach is a refinement of the measure-theoretic (\S~\ref{sec:intro-composite-ville}).
  }
  \label{fig:pascal-fermat}
\end{figure}

Clearly, the latter, measure-theoretic foundation of probability, has become the dominant formalism, and philosophical perspective, in the field.
Yet as \citet{shafer2019game,shafer2001probability} elegantly demonstrate, following Ville and many others, one can derive an equally rich theory of probability entirely from game-theoretic principles.
This game-theoretic perspective has several advantages of robustness and versatility (\S~\ref{sec:case-for-gtp}).
For example, as is appreciated in the literature in online machine learning, game-theoretic statements hold without any stochastic assumption on data-generating process, nor even an assumption that such a ``process'' exists.
Aspects of this robustness and versatility also underpin the emerging field of game-theoretic statistics and e-values, an appealing alternative to traditional hypothesis testing and p-values.

These two theories of probability often agree, just as in the initial discussion between Pascal and Fermat.
Yet the precise general relationship between the two theories has remained largely unexplored.
Establishing a stronger connection between game-theoretic and measure-theoretic probability would deepen our understanding of probability.
Moreover, showing broad conditions under which the two align would allow one to ``lift'' measure-theoretic statements to the stronger, worst-case versions of game-theoretic probability.
Such a connection would also clarify the extent to which testing and inference techniques in game-theoretic statistics, based on measure-theoretic supermartingales and variations, are truly ``game-theoretic'', i.e., strategies in a game.
Finally, a better understanding of this duality would clarify the relationship between game-theoretic probability and other ``robust'' testing and inference approaches, such as robust optimization and robust representations of financial risk measures.

Building on \citet{shafer2001probability,shafer2019game}, this article forges a new bridge between measure-theoretic and game-theoretic probability through the lens of minimax duality.
Our central observation is that \emph{results in game-theoretic probability can be thought of as minimax theorems}: the statement that neither player has an advantage to playing second in certain zero-sum game.
More broadly, the game-theoretic and measure-theoretic expectations agree precisely when minimax duality holds.
One can see this minimax duality already in the conversation between Pascal and Fermat (Fig.~\ref{fig:pascal-fermat}).
This perspective allows us to significantly strengthen the connection between the two theories of probability (\S~\ref{sec:contributions}), though several fundamental open questions remain (\S~\ref{sec:discussion}).

To develop the theory needed to establish these connections, we will see a new framing of game-theoretic probability based on \emph{gamble spaces}.
This formalism highlights connections to measure-theoretic probability, in a way which clarifies the relationship between various core definitions.
After developing this formalism (\S~\ref{sec:definitions},~\ref{sec:prices-probability}), we move to minimax theorems (\S~\ref{sec:minimax-theorems}) and their application (\S~\ref{sec:translating-mtp-gtp}).
For now, let us see a motivating example of how minimax duality arises.

\subsection{Motivating example}
\label{sec:motivating-example}

Consider a simple result from probability theory, the law of large numbers (LLN) for bounded martingale difference sequences:
if $Y_1,Y_2,\ldots$ is a sequence of random variables taking values in $[-1,1]$ with joint law $P$ such that $\E_P[Y_t \mid Y_{1..t-1}] = 0$ almost surely, and $\ALLN = \{y \in [-1,1]^\infty : \lim_{n\to\infty} \frac 1 n \sum_{t=1}^n y_t = 0\}$ is the set of realizations satisfying the law of large numbers, then $P(Y_{1..\infty} \in \ALLN) = 1$.

In a typical game-theoretic version of this statement, e.g.\ \citet[Proposition 1.2]{shafer2019game}, we suppose one is allowed to \emph{gamble} on the values $y_t$ of the $Y_t$.
In each round $t$, one can place a wager $\beta_t\in\reals$ and receive $\beta_t y_t$ when it is revealed that $Y_t = y_t$.
(See \protocol~\ref{alg:bounded-lln}.)
We can formalize the strategy in a function $\psi:[-1,1]^* \to \reals$ which specifies the bet $\beta_{t} = \psi(y_{1..t-1})$ given a sequence $y_{1..t-1} \in [-1,1]^{t-1}$ of partial outcomes.
The total earnings on the full outcome sequence $y\in[-1,1]^\infty$ are given by $Z^\psi(y) := \liminf_{t\to\infty} \sum_{i=1}^t \psi(y_{1..i-1}) y_i$, where the limit infimum takes a conservative view.
The game-theoretic result is then the following: there exists a gambling strategy $\psi^*$ that never risks bankruptcy from an initial capital of \$1, and becomes infinity rich when the law of large numbers fails.
Formally, there is some strategy $\psi^*$ such that $1+Z^{\psi^*}(y) \geq 0$ for all $y\in[-1,1]^\infty$, a condition equivalent to no-bankruptcy in this setting,
\footnote{No bankruptcy clearly implies $1+Z^{\psi^*}(y) \geq 0$ by the definition of $\liminf$.  For the converse, if $1 + \sum_{i=1}^t \psi(y_{1..i-1}) y_i < 0$ for some $y_{1..t}$, then taking $y_i = 0$ for all $i > t$ would violate the condition.
  While the two conditions coincide here, the limit infimum is generally the correct notion;
  see \S~\ref{sec:multiplicative} and Remark~\ref{rem:no-bankruptcy}.}
and $Z^{\psi^*}(y) = \infty$ for all $y\notin \ALLN$.

\begin{algorithm}[t]
  \caption{Sequential bets on a bounded outcome}
  \For{$t = 1, 2, \ldots $}{
    Gambler chooses $\beta_t \in \reals$ \\
    World chooses $y_t \in [-1,1]$ \\
    Gambler receives $\beta_t y_t$ \\
  }
  \label{alg:bounded-lln}
\end{algorithm}

The central observation of this article is that such a game-theoretic statement, like the above for the law of large numbers, can be thought of as a minimax theorem in a particular zero-sum game between Gambler, choosing the strategy $\psi$, and World,
\footnote{\citet{shafer2001probability,shafer2019game} often use the name Reality for the player that chooses the outcomes.
  The name Nature is perhaps better suited to describe one level further removed from the actual outcome, such as the laws governing the outcome.
  We adopt World as a compromise: in this version of the game, World is choosing the outcomes, but in the minimax dual game, World will choose a probability measure over outcomes.}
choosing the outcomes $y_t$.
In the game for the law of large numbers above, Gambler chooses a strategy $\psi$ with $\inf Z^\psi > -\infty$ (no bankruptcy with respect to some finite initial capital), World chooses a probability measure $P$ over $[-1,1]^\infty$ equipped with the Borel $\sigma$-algebra,
and the payoff to World is
\begin{align}
  u(P,\psi) := \E_P\left[X - Z^\psi\right]~,
\end{align}
where we set $X = \ones_{(\ALLN)^c}$.
The payoff to Gambler is $-u(P,\psi)$.
\footnote{Technically, for $u$ to be defined, $\psi$ must be measurable and the relevant expectations defined.  In some sense we can side-step both of these issues (\S~\ref{sec:prices-probability}).  This extension is important as typical game-theoretic statements do not require measurability, since they consider the case where Gambler must play first and World may choose $\omega\in\Omega$ directly (see below).}
Phrased in this way, one can see $u(P,\psi)$ as the \emph{replication cost} of the variable $X = \ones_{(\ALLN)^c}$ with respect to the probability measure $P$ and strategy $\psi$: the infimum over $\alpha \in \reals$ such that $\alpha + Z^\psi \geq X$ (see Fig.~\ref{fig:replication}).
The quantity $\inf_\psi \sup_P u(P,\psi)$ is then the best worst-case replication cost.

Let us see why the game-theoretic law of large numbers amounts to a minimax theorem in this game.
First consider the version of this game where World must play first, which corresponds to the typical situation in probability theory when the measure $P$ is known.
In this case, if $\E_P[Y_t \mid Y_{1..t-1}] = 0$ fails with positive probability, then Gambler can choose $\psi$ to capitalize on this bias, forcing $u(P,\psi) \to -\infty$.
\footnote{Specifically, for sufficiently small $\epsilon>0$, and any $\alpha > 0$, Gambler can set $\psi(y_{1..t-1}) = \alpha/\E_P[Y_t \mid Y_{1..t-1} = y_{1..t-1}]$ on the first round when the absolute value of the denominator is at least $\epsilon$.  Then Gambler risks a bounded amount ($\alpha/\epsilon$) and earns $\alpha$ in expectation on a round where the condition is violated.  By taking $\alpha\to\infty$, Gambler earns an infinite expected profit if the condition is violated with positive probability.  See Proposition~\ref{prop:consistent-implies-seq-consistent} and Theorem~\ref{thm:ep0-equals-ep}.}
Otherwise, $P$ is a martingale measure and we have $\E_P Z^\psi \leq 0$ by Fatou's lemma.
Thus, the best payoff that World can guarantee is
\begin{align}
  \sup_P \inf_\psi u(P,\psi)
  &= \sup_{P : \E_P[Y_t \mid Y_{1..t-1}] = 0\; P\text{-a.s.}} \inf_\psi u(P,\psi)
  \\
  &= \sup_{P : \E_P[Y_t \mid Y_{1..t-1}] = 0\; P\text{-a.s.}} \inf_\psi \E_P\left[\ones_{(\ALLN)^c} - Z^\psi\right]
  \\
  &= \sup_{P : \E_P[Y_t \mid Y_{1..t-1}] = 0\; P\text{-a.s.}} \E_P[\ones_{(\ALLN)^c}]
  \\
  &= 0~,
\end{align}
where we apply the (measure-theoretic) law of large numbers in the final equality.

Now consider the case where Gambler must play first.
Here we may restrict World to point measures without loss of generality, as World cannot gain by randomizing.
By the game-theoretic result discussed above, there exists a gambling strategy $\psi^*$ such that $Z^{\psi^*} \geq -1$ and $Z^{\psi^*}(y) = \infty$ when $y\notin \ALLN$.
Thus,
\begin{align}
  \inf_\psi \sup_P u(P,\psi)
  &= \inf_\psi \sup_y \ones\{y\notin \ALLN\} - Z^\psi(y)
    \label{eq:intro-gambler-force-consistent}
  \\
  &\leq \inf_{\lambda \geq 0} \sup_y \ones\{y\notin \ALLN\} - Z^{\lambda\psi^*}(y)
  \\
  &= \inf_{\lambda \geq 0} \sup_y \ones\{y\notin \ALLN\} - \lambda Z^{\psi^*}(y)
  \\
  &\leq \inf_{\lambda \geq 0} \sup_y
    \begin{cases}
      \lambda & y\in \ALLN\\
      -\infty & y\notin \ALLN
    \end{cases}
  \\
  &=0~.
\end{align}

Recall that we always have $\sup_P \inf_\psi u(P,\psi) \leq \inf_\psi \sup_P u(P,\psi)$, i.e., it is weakly better to play second in a zero-sum game.
\footnote{In this particular setting, we could also observe $\inf_\psi \sup_P u(P,\psi) \geq 0$ directly via the point measure $P$ on the all-zero outcome sequence.}
The existence of the gambling strategy $\psi^*$ thus implies \emph{minimax duality}:
\begin{align}
  \label{eq:minimax-duality-intro}
  \sup_P \inf_\psi u(P,\psi) = \inf_\psi \sup_P u(P,\psi)~.
\end{align}
In other words, one could regard the game-theoretic result, i.e., the derivation of $\psi^*$, as a constructive minimax theorem for this particular game.
In fact, if one were able to show minimax duality for this game, perhaps using a general noncontructive minimax theorem, it would imply the existence of such a gambling strategy, at least in the limit: a sequence of gambling strategies starting from \$1 which become unboundedly rich when $\ALLN$ fails.
(See Corollary~\ref{cor:prob-one-iff}.)

\subsection{Consistency and composite Ville}
\label{sec:intro-composite-ville}

Ville's celebrated theorem~\citep{ville1939etude}, and variations, roughly states that for a fixed probability measure $P$, and event $A$ with $P(A)=0$, there is a nonnegative $P$-supermartingale starting at 1 which becomes infinite on $A$.
More generally, for any $A$, there are nonnegative $P$-supermartingales starting at 1 that become arbitrarily close to $1/P(A)$ on $A$.

In our example above, we may consider any martingale measure $P$, and the event $A = (\ALLN)^c$ where the LLN fails.
The capital process from the strategy $\psi^*$, given by $X_t = 1 + \sum_{i=1}^t \psi^*(Y_{1..i-1}) Y_i$, is exactly a witness to Ville's Theorem: a nonnegative $P$-supermartingale starting at 1 that becomes infinitely rich on $A$.
In fact, we can say something more, since the same $\{X_t\}_t$ is a $P$-supermartingale for all martingale measures $P$ \emph{simultaneously}, and likewise we have $P(A) = 0$ for all such $P$.
One could therefore consider $\psi^*$ a \emph{composite} strategy for Gambler, which establishes a composite version of the LLN.

More broadly, let us define $\Delta_0$ to be the set of probability measures $P$ ``consistent'' with the gambles, meaning the set of $P$ making them nonprofitable conditioned on the outcomes $y_{1..t}$ so far.
\footnote{We will refer to this condition in \S~\ref{sec:prices-probability} and beyond as \emph{sequential consistency}, as opposed to (global) consistency, which says that any net gamble $Z^\psi$ must have $\E_P Z^\psi \leq 0$.}
In \S~\ref{sec:motivating-example}, $\Delta_0$ is the set of martingale measures.
From similar arguments to those in \S~\ref{sec:motivating-example}, then, one can see that the following chain of inequalities will always hold in settings like this one:
\begin{equation}
  \label{eq:intro-price-inequalities}
  \begin{tikzpicture}[baseline]
    \matrix [matrix of math nodes,row sep=0.1cm,column sep=0.1cm] (m) {
      \text{(i)} && \text{(ii)} && \text{(iii)}
      \\
      {\displaystyle \sup_{P\in\Delta_0} P(A)} & \leq
      &\displaystyle \sup_P \inf_\psi u(P,\psi)
      &\leq
      &\displaystyle \inf_\psi \sup_P u(P,\psi)~.
      \\[-6pt]
      \rotatebox{-90}{$\to$}&&\rotatebox{-90}{$\to$}&&\rotatebox{-90}{$\to$}
      \\[-4pt]
      \Epuseq X && \Epu X && \Egu X
      \\ };
  \end{tikzpicture}
\end{equation}
where $A = (\ALLN)^c$ in \S~\ref{sec:motivating-example}.
One can easily see the composite measure-theoretic LLN from this chain of inequalities, as the strategy $\psi^*$ witnesses the right-hand side of eq.~\eqref{eq:intro-price-inequalities} being $0$, giving $P((\ALLN)^c) = 0$ for all $P\in\Delta_0$ from the left-hand side.
When gambles are scalable as they are in this example, the first inequality is always an equality, which is why the second inequality, minimax duality, is the key understanding when measure-theoretic and game-theoretic probability agree.
These three quantities are central to our theory, and correspond to the ``prices'' $\Epuseq X$, $\Epu X$, $\Egu X$ we will define in \S~\ref{sec:definitions} and \S~\ref{sec:prices-probability}.
See Theorem~\ref{thm:chain-of-price-inequalities} and Corollary~\ref{cor:sequential-price-inequalities} for the chain of inequalities, and Theorem~\ref{thm:ep0-equals-ep} and Corollary~\ref{cor:ep0star-equals-ep} for the first equality.

Putting these ideas together, we can see that a general minimax theorem, which would establish (ii)\;$=$\;(iii) in eq.~\eqref{eq:intro-price-inequalities} for a wide variety of games $u(P,\psi)$ and variables $X$, would imply a composite version of Ville's Theorem.
To see this implication, let us start with a desired set $\P$ of probability measures and set $A$.
We set up a game with scalable gambles such that the consistent probability measures are exactly $\Delta_0 = \P$, and take $X = \ones_A$.
By scalability and minimax duality, we have (i)\;$=$\;(iii), which means that the maximum likelihood of $A$, given by $p = \sup_{P\in\P} P(A)$, is equal to $\inf_\psi \sup_P u(P,\psi)$, the best worst-case replication cost of $X$.
As a result, for any $p' > p$, there is a strategy $\psi$ with $\sup_P u(P,\psi) < p'$.
Equivalently, $\tfrac 1 {p'} \psi$ can take an initial capital of \$1 and exceed \$$\tfrac 1 {p'}$ on $A$, which is arbitrarily close to $\tfrac 1 p$.

A key condition in this setup is the consistency $\Delta_0 = \P$, which heavily restricts the sets $\P$ one can consider in a theorem of this type.
In particular, given any $P,P' \in \P$ with filtration $\F_t = \sigma(Y_1,Y_2,\ldots,Y_t)$ generated by the per-round outcomes, at time $t$ we must be able to switch from following $P$ to following $P'$, a condition similar to stability under pasting~\citep{bartl2020conditional} or fork-convexity~\citep{ramdas2022testing}.
Typically, this condition means $\P$ must contain a certain rich class of martingale measures.
As we will see in \S~\ref{sec:translating-mtp-gtp}, an implication is that, in order to lift measure-theoretic statements to game-theoretic ones, the measure-theoretic statement must allow for martingales, or processes with similarly permissive conditional structure.

As a final remark, let us consider the relationship between optimal strategies for both players.
Again taking the game in \S~\ref{sec:motivating-example} for concreteness, note that for any event $A$ (indeed any $X$), the optimal strategy $P$ for World when playing first is always an element of $\Delta_0$, by eq.~\eqref{eq:intro-gambler-force-consistent}.
Since $P \in \Delta_0$ neutralizes Gambler's winnings, and the only remaining utility is in $P(A)$, the optimal strategies are precisely $\argsup_{P\in\Delta_0} P(A)$.
When $A$ is a probability zero event for all $P\in\Delta_0$, like the complement of the LLN or LIL (law of the iterated logarithm), every $P\in\Delta_0$ is optimal for World.
Yet even among probability zero events $A$, it is clear that the optimal strategy $\psi^*$ for Gambler depends heavily on $A$.
In that sense, as observed by \citet[\S~1]{ruf2023composite}, one can consider game-theoretic probability to be a \emph{constructive refinement} of measure-theoretic probability.

\subsection{The case for game-theoretic probability}
\label{sec:case-for-gtp}

What are the benefits, if any, to working with game-theoretic probability, as opposed to the standard measure-theoretic framework?
Here we list five, expanded on from \citet{vovk2009merging,vovk2017measurability,shafer2019game}.

\begin{enumerate}
\item Stronger guarantees: worst-case, ``pathwise'' statements.
  As appreciated in finance and online machine learning, statements in game-theoretic probability hold for all possible outcome sequences.
  As such, no assumption is needed on how the outcome sequence is generated, stochastic or otherwise.
  The guarantees are stronger in the sense that they still readily imply their stochastic counterparts, a point we clarify further in \S~\ref{sec:prices-probability} and \S~\ref{sec:gtp-to-mtp-always}.
  
\item Philosophical appeal: relevance to the real world.
  Beyond their mathematical strength, game-theoretic statements operate in a model of the world that more readily accommodates the data we currently apply probabilitistic methods to.
  As Kolmogorov famously lamented about the use of probability to analyze literature~\citep{vitanyi2013tolstoy,kolmogorov1965three}, we rarely perform inference on or test hypotheses from data which are truly generated by a stochastic process.
  Arguably, the more common situation is that the existence of such a stochastic process would be proposterous.
  
\item Deeper, more constructive understanding of probability.
  As demonstrated in Fig.~\ref{fig:pascal-fermat} and the final paragraph of \S~\ref{sec:intro-composite-ville}, game-theoretic probability can be considered a constructive refinement of measure-theoretic probability.
  To understand an event in game-theoretic probability requires more than determining its probability: it requires a strategy to test it, to bet against it.
  As put by \citet{ruf2023composite}, ``Different measure-zero events---for example, sequences violating the strong law of large numbers (SLLN) and those violating the LIL---obviously have the same probability, but they result in different betting strategies.''
  Thus, even if, as we conjecture in \S~\ref{sec:future-work}, the two theories align in the vast majority of useful cases, in that they assign the same ``price'' to each variable, it is still fruitful to study replication strategies that give those prices.
  
\item Clarity of assumptions.
  One commonly touted benefit of game-theoretic probability is that one does not need to specify the expected value (``price'') of every variable $X$.
  Yet in many cases, this is not a fair criticism of measure-theoretic probability: many measure-theoretic statements also refrain from specifying $\E X$ for every $X$, and thus hold for a set of probability measures $\P$ as in \S~\ref{sec:intro-composite-ville}.
  For example, the bounded measure-theoretic (martingale) LLN only assumes the underlying $P$ satisfies $\E_P[Y_{t+1} \mid Y_{1..t}] = 0$ $P$-a.s., and thus the result holds for the set $\P$ of such $P$.
  Perhaps a more accurate criticism, then, is that measure-theoretic results do not make this set $\P$ explicit.
  Indeed, it can be challenging to take statements like the bounded LLN and identify the set $\P$ for which the statement holds.
  (And for other statements, such a $\P$ need not exist, as in Non-Theorem~\ref{thm:potential-gtp-clt}.)

\item Ease of sequential constructions.
  As discussed in~\citet[\S~9.1]{shafer2019game}, the game-theoretic approach to sequential settings like discrete-time processes starts from the local and moves to the global.
  That is, a protocol specifies the gambles available in each round, and one deduces global properties of the game from there.
  This approach lends itself to considerable flexibility and ease relative to the measure-theoretic approach, where now measurability presents a nontrivial technical barrier.
  This barrier manifests in this article when trying to establish the tower property of the measure-theoretic upper expectation in \S~\ref{sec:tower-properties-sequential-minimax}.
  Despite the fact that the game-theoretic and measure-theoretic tower properties have essentially the same underlying logic, while the proof is short and intuitive for the game-theoretic tower property (Proposition~\ref{prop:tower-property-egu}), the measure-theoretic proof (Lemma~\ref{lem:usa-facts}) is much more involved.
  The essential difficulity is that one needs to find a suitable class of functions that is stable under iterated upper expectations, and Borel measurable functions are not such a class (Example~\ref{ex:seq-meas-fail}).
  Yet the game-theoretic framework is unincumbered by such considerations, at least in discrete time~\citep{vovk2017measurability}, all while being more flexible and giving stronger results.
\end{enumerate}

\subsection{Contributions}
\label{sec:contributions}

The LLN example in \S~\ref{sec:motivating-example} begs a question:
could one prove a \emph{general} minimax theorem for many such probability games of interest?
While there is a vast literature on minimax theorems~\citep{simons1995minimax}, these theorems do not obviously apply to eq.~\eqref{eq:minimax-duality-intro} in all cases of interest (see \S~\ref{sec:minimax-theorems}).
Yet the proliferation of game-theoretic results that match their measure-theoretic counterparts suggests that such a general theorem should be possible.

In this article, we give such a general minimax theorem for finite time horizons, which we use to recover many existing game-theoretic results as well as establish new ones.
(The infinite-horizon case, including a theorem covering the LLN example above, remains open (\S~\ref{sec:future-work}.))
This minimax theorem in turn requires a suite of new results to connect game-theoretic probability to measure-theoretic probability.
In total, the technical contributions of this article are as follows.
\begin{enumerate}
\item A reframing of game-theoretic probability by defining and building upon \emph{gamble spaces}, as an analog of probability spaces (Definition~\ref{def:gamble-space}).
  Compared to \citet{shafer2001probability,shafer2019game}, gamble spaces satisfy fewer axioms (\S~\ref{sec:axioms-fa}).
  This simple reframing provides a foundation upon which several general results can stand, and clarifies the relationship to related disciplines like online machine learning (\S~\ref{sec:online-learning}) and financial risk measures (\S~\ref{sec:financial-risk-measures}).
  
\item The observation that when trying to replicate a measurable variable $X$, Gambler may without loss of generality restrict to measurable strategies (Proposition~\ref{prop:dcl-wlog}).
\item The series of price inequalities analogous to eq.~\eqref{eq:intro-price-inequalities}, which considers the cases when World plays first, and is additionally consistent with the gambles (Theorem~\ref{thm:chain-of-price-inequalities}, Corollary~\ref{cor:sequential-price-inequalities}).
  A corollary is that measure-theoretic expectations always lower bound the game-theoretic version.
\item The observation discussed in \S~\ref{sec:intro-composite-ville} that a minimax theorem implies a composite version of Ville's Theorem (\S~\ref{sec:seq-price-equal}).
\item A new minimax theorem for finite time horizons (Theorem~\ref{thm:finite-sequential-minimax}), which gives rise to a general way to convert measure-theoretic statements into game-theoretic ones (Corollary~\ref{cor:gtp-to-mtp-sequential}).
  This result extends a backward induction argument from online machine learning~\citep{abernethy2009stochastic}, which bears some resemblance to Pascal's derivation above.
\item Several new game-theoretic results, to illustrate the minimax theorem (\S~\ref{sec:translating-mtp-gtp}).
  As discussed in \S~\ref{sec:case-for-gtp}, these results show the existence of gambling strategies to replicate the relevant quantities, but do so nonconstructively.
\item A new connection to finitely additive measures, namely that minimax duality \emph{always} holds for convex gamble spaces when relaxing countable additivity (\S~\ref{sec:axioms-fa}).
  In a strong sense, then, game-theoretic probability in its full generality is a finitely additive theory of probability.
\end{enumerate}

We leave several directions for future work; in particular, we conjecture that an even more general minimax theorem can be established for countably infinite time (\S~\ref{sec:future-work}).

\subsection{Relationship to the literature}
\label{sec:related-work}

We briefly review the connection and relevance to several disciplines.

\paragraph{Game-theoretic probability}

The literature on game-theoretic probability is well summarized by \citet{shafer2001probability,shafer2019game}.
As discussed in \S~\ref{sec:definitions}, many of the basic definitions and results in this article are identical to or straightforward extensions of those found in \citet{shafer2019game}, while others are novel in their generality.
Adding to this literature, our main minimax result (Theorem~\ref{thm:finite-sequential-minimax}) is the first to apply to nontrivial composite settings where the per-round outcome set $\Y$ can have infinite cardinality.
As we discuss in \S~\ref{sec:existing-minimax}, \citet{shafer2019game} also give two minimax theorems, one (their Theorem 9.3) a game-theoretic extension of Ville's Theorem, which is not composite, and the other (their Theorem 9.7) which requires $\Y$ to be finite.
Both of these results hold for infinite time horizons, however, whereas ours is a finite-time result.

\paragraph{Game-theoretic statistics}

While we do not focus on testing protocols per se, the framework and results presented are intimately connected to game-theoretic statistics~\citep{ramdas2023game}.
Of particular relevance are Ville-like results for composite settings, such as \citet{ramdas2022testing,ruf2023composite}.
We also make use of a recent characterization of e-variables due to \citet{larsson2025variables} for our finite-time minimax result.

\paragraph{Online machine learning}

The price inequalities in \S~\ref{sec:prices-probability} were inspired by a series of works in the online learning literature applying minimax duality to upper and lower bound the regret of various algorithms and settings~\citep{koolen2014efficient,abernethy2012minimax,abernethy2009stochastic}.
Particularly relevant are \citet{abernethy2009stochastic} and \citet{abernethy2012minimax}, as both explicitly apply minimax duality to relate the ``prices'' defined in \S~\ref{sec:price-definitions} in a more restricted setting.
The backward induction argument in our main minimax result, Theorem~\ref{thm:finite-sequential-minimax}, is a generalization of the argument of \citet[Theorem 1]{abernethy2009stochastic}.
The present framework was particularly inspired by the works of \citet{rakhlin2012relax,foster2018online,rakhlin2017equivalence}.

\paragraph{Finance}

There are deep connections between game-theoretic probability and two subfields of mathematical finance.
The first and perhaps least surprising is the literature on pathwise hedging equalities and inequalities~\citep{beiglboeck2016pathwise,beiglboeck2014martingale,beiglboeck2015pathwise,nutz2013constructing,bouchard2013arbitrage}.
Perhaps closest to our setting are \citet{nutz2013constructing} and \citet{bouchard2013arbitrage}.
The latter introduces a very similar setup and minimax result, in the special case where the gambles available in each round are linear, i.e., stock portfolio returns of the form $Z:y\mapsto\inprod{\beta}{y}$ where $\beta,y\in\reals^d$.

The second is the literature on financial risk measures~\citep{follmer2016stochastic, delbaen2002coherent, kratschmer2005robust, kratschmer2006sigma}.
As also noted in \citet{shafer2019game}, game-theoretic upper expectations satisfying their axioms (see \S~\ref{sec:axioms-fa}) are closely related to coherent financial risk measures.
We observe that removing essentially all of these axioms still preserves the two defining properties of financial risk measures, translation and monotonicity.
The literature on time-consistent risk measures is particularly relevant, as sequential gamble spaces are essentially always time-consistent.
We leverage this connection and the elegant work of \citet{bartl2020conditional} to prove our sequential minimax result.

\subsection{Acknowledgements}

I am deeply indebted to Aaditya Ramdas and Peter Gr\"unwald for numerous conversations, invitations, and specific feedback and ideas which appear throughout the article.
I also am grateful for the conversations with
Tobias Fissler,
Wouter Koolen,
Martin Larsson,
Manuel Lladser,
Ryan Martin,
Nishant Mehta,
Johannes Ruf,
Glenn Shafer,
Karthik Sridharan,
Zachary van Oosten,
Volodya Vovk,
Bo Waggoner,
Ruodu Wang,
and
Johanna Ziegel.
Finally, I thank
the other participants at the 2024 \emph{Game-theoretic statistical inference} workshop at Mathematisches Forschungsinstitut Oberwolfach,
and students at the University of Colorado Boulder in \emph{Online Machine Learning, Forecasting, and e-Values} (Fall 2023) and \emph{Game-theoretic Probability, Statistics, and Machine Learning} (Fall 2025),
for discussions, feedback, references, insights, and encouragement.

\section{Game-theoretic probability via gamble spaces}
\label{sec:definitions}

To state general results about game-theoretic probability, it will be useful to work within a unifying framework that we dub \emph{gamble spaces}.
In this framework, there are only two players, World and Gambler.

Gamble spaces essentially represent the ``offers'' view of \citet[\S~6]{shafer2019game} but without the axioms.
One advantage of the unifying framework is the ability to prove general statements about certain types of gamble spaces, rather than reiterating similar statements for each protocol.
Despite this shift in framing, much of the notation, terminology, and results in this section already appear in, or draw inspiration from, the books of \citet{shafer2001probability,shafer2019game}.

\subsection{Basic definitions}
\label{sec:basic-definitions}

Throughout, the action space of World is $\Omega$, an arbitrary set representing an \emph{outcome} space.
Gambler has access to some set $\Z$ of \textsl{gambles} $Z:\Omega\to\extreals$, where throughout we let $\extreals := \reals\cup\{\infty,-\infty\}$ be the extended reals.
Upon outcome $\omega\in\Omega$, a gamble $Z\in\Z$ has a payoff $Z(\omega)$ to Gambler.
Elements of $\Z$ can be thought of as contracts, algorithms to buy/sell stocks, insurance policies, etc.
Finally, we will call any function $X:\Omega\to\extreals$ a \emph{variable}; all gambles are variables.
For variables $X,Y$, we write $X \geq Y$ to mean a pointwise inequality, i.e., $X(\omega)\geq Y(\omega)$ for all $\omega\in\Omega$.
When $\Omega$ is clear from context, we write $\sup X := \sup_{\omega\in\Omega} X(\omega)$ and similarly for $\inf X$.

\begin{definition}[Gamble space]
  \label{def:gamble-space}
  A \emph{gamble space} is a pair $(\Omega,\Z)$, where $\Omega$ is a set and $\Z \subseteq (\Omega\to\extreals)$.
\end{definition}

Given some variable $X$ and set of gambles $\Z$, one defines the upper and lower game-theoretic probability of $X$ with respect to $\Z$ as follows.
The definition allows for expressions of the form $\infty - \infty$ and $(-\infty) - (-\infty)$; to be pessimistic, we take both of these to equal $\infty$.
See \S~\ref{sec:infinite-cases}.

\begin{figure}[t]
  \hspace*{1cm}
  \includegraphics{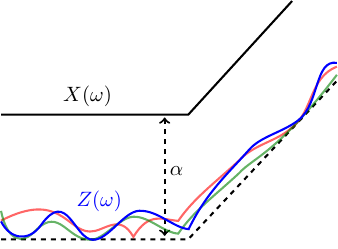}
  \hfill
  \includegraphics{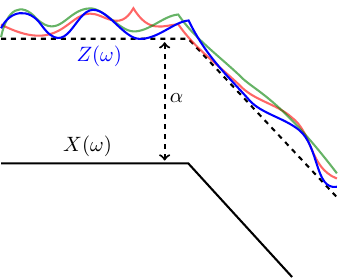}
  \hspace*{1cm}
  \caption{The replication cost to sell or buy $X$, respectively, relative to existing gambles $\Z$.  This cost is exactly $\alpha$ if the depicted gamble $Z$ achieves the infimum in eq.~\eqref{eq:prop-upper-ex-rep-cost}.}
  \label{fig:replication}
\end{figure}

\begin{definition}[Game-theoretic upper expectation]
  \label{def:egu}
  
  Let $(\Omega,\Z)$ be a gamble space, and $X:\Omega\to\extreals$ a variable.
  Then the \emph{upper game-theoretic expectation of $X$}, with respect to $(\Omega,\Z)$, is given by the following equivalent definitions (see Proposition~\ref{prop:egu-equiv-defs}).
  \begin{align}
    \Egu X
    &:= \inf_{Z\in\Z} \sup_{\omega\in\Omega} X(\omega) - Z(\omega)
      \label{eq:formal-egu}
    \\
    &\phantom{:}= \inf\{\alpha\in\reals \mid \exists Z\in\Z \text{ s.t. } Z + \alpha > X\}~.
      \label{eq:prop-upper-ex-rep-cost}
  \end{align}
  We define the \emph{lower game-theoretic expectation of $X$}, with respect to $(\Omega,\Z)$ by $\Egl X := - \Egu(-X)$.
  When $\Egu X = \Egl X$, we write $\Eg X := \Egu X = \Egl X$.
\end{definition}
When $X$ is real-valued, the strict inequality in eq.~\eqref{eq:prop-upper-ex-rep-cost} can be replaced with a weak one; see Proposition~\ref{prop:egu-equiv-defs} and Remark~\ref{rem:egu-rep-cost-strict}.
When working with multiple gamble spaces, we sometimes write $\Egu_\Z X$ to explicitly note the dependence on $\Z$.

As discussed in \S~\ref{sec:introduction}, one can think of eq.~\eqref{eq:formal-egu} as a zero-sum game between Gambler and World, where Gambler plays first.
From eq.~\eqref{eq:prop-upper-ex-rep-cost}, one can also think of $\Egu X$ as a \emph{replication cost} with respect to $\Z$ (Fig.~\ref{fig:replication}).
This replication cost formulation can be useful when thinking of $\alpha$ as the initial capital replication strategy: successfully replicating $X$ starting with $\alpha$ would show $\Egu X \leq \alpha$.

We define the upper and lower game-theoretic probability of $A \subseteq \Omega$ as follows,
\begin{align}
  \Pgu A &:= \Egu \ones_A~,
  \\
  \Pgl A &:= \Egl \ones_A~,
\end{align}
where $\ones_A:\Omega\to\{0,1\}$ is the indicator variable with $\ones_A(\omega) = 1$ if and only if $\omega \in A$.
From eq.~\eqref{eq:prop-upper-ex-rep-cost} we also have
\begin{align}\label{eq:pgu-rep-cost}
  \Pgu A &= \inf\{\alpha\in\reals \mid \exists Z\in\Z \text{ s.t. } Z + \alpha > \ones_A\}~.
\end{align}
If equal, we write $\Pg A := \Pgu A = \Pgl A$.
When convenient we will write $\Pgl A = 1$ as ``$A$ g.t.a.s.'' (game-theoretically almost surely).

\begin{proposition}\label{prop:egu-equiv-defs}
  The two definitions of $\Egu$ in Definition~\ref{def:egu} are equivalent.
  If $X(\omega)=Z(\omega)=c \implies c\in\reals$ for all $Z\in\Z$, then we further have
  \begin{align}
    \Egu X
    &= \inf\{\alpha\in\reals \mid \exists Z\in\Z \text{ s.t. } Z + \alpha \geq X\}~.
      \label{eq:egu-rep-cost-weak}
  \end{align}
  In particular, eq.~\eqref{eq:egu-rep-cost-weak} holds if $X\in\reals$ or $\Z\subseteq(\Omega\to\reals)$.
\end{proposition}
\begin{proof}
  We will use the following claim about extented real numbers.
  \begin{quote}
    For $\alpha\in\reals$ and $x,z\in\extreals$, we have $z + \alpha > x \iff \alpha > x - z$.\\
    If $x \neq z$ or $x,z\in\reals$, we have $z + \alpha \geq x \iff \alpha \geq x - z$.
  \end{quote}
  If $x,z\in\reals$ the claim is trivial.
  It remains only to check cases with at least one infinite value.
  If $x = z = \infty$ or $x = z = -\infty$, both strict inequalities are false by our convention $\infty - \infty = (-\infty) - (-\infty) = \infty$.
  Otherwise, $x$ and $z$ obey the usual rules of arithmetic and both strict and weak equivalences hold.

  By applying the strict part of the claim pointwise, we have $Z + \alpha > X \iff \alpha > X - Z$.
  Thus
  \begin{align*}
    \inf_{Z\in\Z} \sup(X-Z)
    &= \inf_{Z\in\Z} \inf\{\alpha \in \reals \mid \alpha > X - Z\}
    \\
    &= \inf\{\alpha \in \reals \mid \exists Z\in\Z \text{ s.t. } \alpha > X - Z\}
    \\
    &= \inf\{\alpha \in \reals \mid \exists Z\in\Z \text{ s.t. } Z + \alpha > X\}~.
  \end{align*}
  By the second part of the claim, when $X(\omega)=Z(\omega)=c \implies c\in\reals$, the same logic shows the same statement with weak inequality.
\end{proof}

\begin{remark}\label{rem:egu-rep-cost-strict}
  As the proof of Proposition~\ref{prop:egu-equiv-defs} shows, the strict inequality in eq.~\eqref{eq:prop-upper-ex-rep-cost} is crucial for the equivalence between the two definitions when $X$ or $Z$ take on infinite values.
  If we replaced it with a weak inequality, then for $X=\infty$ and $\Z=\{\infty\}$ we would have $\Egu X = \infty$ under eq.~\eqref{eq:formal-egu} but $\Egu X = -\infty$ under eq.~\eqref{eq:prop-upper-ex-rep-cost}.
  For much of the applications and examples, however, $X$ will be real-valued or even bounded (as in the case of $\Pgu$), and we may safely use the weak inequality version.
  
\end{remark}

\begin{remark}\label{rem:upper-expectation-philosophy}
  To think of $\Egu$ as an ``upper expectation'' and $\Pgu$ as an ``upper probability'', one might wish them to satisfy certain properties, such as $\Egu X \geq \Egl X$, $\Egu X \leq \sup X$, $\Egu[X + Y] \geq \Egu X + \Egu Y$, $\Pgu A \in [0,1]$, $\Pgu A + \Pgu A^c \geq 1$, etc.
  In general, as various counterexamples in \S~\ref{sec:conditions-upper-lower} show, all of these properties can fail without further assumptions.
  (For a simple example, when $\Z = \{\omega \mapsto 1\}$ we have $\Egu 0 = -1 < 1 = \Egl 0$.)
  In this article, we will most commonly assume that gambles are arbitrage-free and positive linear (a convex cone; Definition~\ref{def:gamble-space-conditions}), from which these and many other familiar properties follow.
  \citet{shafer2019game} reserve the term ``upper expectation'' solely for gamble spaces satisfying these conditions and an additional continuity condition (Axiom E5 in \S~\ref{sec:axioms-fa}), from which stronger properties follow like monotone convergence.

  Yet, despite straining the philosophy, we may fruitfully continue studying $\Egu$ and $\Pgu$ without these additional conditions.
  It turns out that several important properties do not require any conditions whatsoever, such as the chain of price inequalities (Theorem~\ref{thm:chain-of-price-inequalities}) and tower property (Proposition~\ref{prop:tower-property-egu}), and others require only scalability, such as the Ville-like characterization of $\Pgu A$ as the smallest $\alpha$ where a gambling strategy can start at 1 and reach any level below $1/\alpha$ on $A$ (\S~\ref{sec:multiplicative}).
  In some settings like online machine learning, we will not even have scalability, making $\Pgu$ less interpretable but also less interesting: rather than trying to replicate  $X=\ones_A$ for some event $A$, we will be more interested in replicating quantities like the performance of the best action in hindsight (\S~\ref{sec:online-learning}).
  While not the focus of the present work, capturing online machine learning was a key motivation for the generality of our definitions.
\end{remark}

\subsection{Multiplicative gambles and the capital process}
\label{sec:multiplicative}

An important case of interest is when the variable $X$ to be replicated is nonnegative, such as an indicator $\ones_A$.
In this case, it is natural to write gambles as \emph{multiplicative} rather than additive, and think of them as investments of capital.
This view is especially fruitful in sequential settings, when one can think of Gambler as reinvesting her capital repeatedly over many rounds:
the resulting \emph{capital process} can be written as the product of multiplicative gambles.
This process is even simpler when Gambler is allowed to scale gambles up or down, as in \S~\ref{sec:motivating-example}.
Since this view is so fundamental to many framings of this theory, for example in game-theoretic statistics, let us develop the high-level connection, even before we have seen several core definitions (e.g.\ sequential gambles, Definition~\ref{def:sequential-gamble-space}).

As a first observation, consider any gamble $Z$ and any initial capital $\alpha > 0$.
Then we could think of Gambler playing $Z$ as a one-round capital process, starting with $\alpha$ and ending with $\alpha + Z$.
We can equivalently think of this process as multiplicative: letting $E = 1 + (1/\alpha)Z$, we start with $\alpha$ and multiply our capital by $E$, ending with $\alpha E = \alpha + Z$.
(The notation $E$ is suggestive of e-variables; see below.)
The converse is straightforward: we simply define $Z = \alpha(E - 1)$.
This conversion makes perfect sense for any gamble space and any $\alpha > 0$.

In our game, Gambler is trying to replicate the variable $X$, and must choose $\alpha$ and $Z$ so that $\alpha + Z \geq X$.
Thus, the above conversion will only be satisfactory for Gambler when $\Egu X \geq 0$, meaning one can approach $\Egu X$ from above with a sequence of strictly positive initial capitals $\alpha > 0$.
When gambles are sufficiently arbitrage-free in the sense that the only way to prevent a loss is to refrain from betting ($Z\geq 0 \implies Z = 0$),
\footnote{This notion is slightly stronger than the notion of arbitrage-free in Definition~\ref{def:gamble-space-conditions}.} the condition $\Egu X \geq 0$ is implied by $X \geq 0$ (Propositions~\ref{prop:basic-facts-no-assumptions},~\ref{prop:basic-facts-zero}), hence the focus on nonnegative variables.

Let us now assume $X\geq 0$ and consider sequential setting where in each round gambles are arbitrage-free in the sense above.
In each round $t = 1,2,\ldots$, Gambler now selects some $Z_t$, and World selects an outcome for that round.
Starting from capital $C_0 \geq 0$, Gambler's capital process is $C_t = C_0 + \sum_{i\leq t} Z_{i}$, culminating in $C_T \geq X$.
That is, Gambler succesfully replicates $X$ starting with $C_0 \geq 0$.
Following \citet{ville1939etude} and \citet{shafer2001probability,shafer2019game}, we refer to $\{C_t\}_t$ as a \emph{game-theoretic supermartigale} (Definition~\ref{def:supermartingale}), and show that (when measurable) it is always a measure-theoretic supermartingale as well (Proposition~\ref{prop:supermartingale-to-measure}).

Under our assumptions, one can always convert the additive process $\{C_t\}_t$ to a multiplicative one.
If $C_t > 0$, then we may write $C_{t+1} = C_t E_{t+1}$ where $E_{t+1} = 1 + (1/C_t)Z_{t+1}$ as above.
If $C_t \leq 0$, then it must be the case that $C_t = 0$ and there exists some $Z_{t+1} \geq 0$.
Otherwise, for every choice of $Z_{t+1}$, World can force $C_{t+1} < 0$, and since gambles are arbitrage-free, World can keep the subsequent capital process bounded below 0, contradicting $C_T \geq X \geq 0$.
Thus, by our assumption, we must have $Z_{t+1} = 0$ and thus $C_{t+1} = 0$ as well, and can safely define $E_{t+1} = 0$.
In all cases, then, we may write $C_t = C_0 \prod_{i\leq t} E_i$.
One can interpret this multiplicative process as \emph{reinvesting} the current capital in a new gamble $E_t$.

This multiplicative representation is especially nice when the gambles $\Z$ in each round are \emph{scalable}, meaning $Z\in\Z, c\geq 0 \implies c Z \in \Z$ (Definition~\ref{def:gamble-space-conditions}).
In this case, if $C_t > 0$ and $Z_t$ is a gamble keeping $C_t + Z_t \geq 0$, we have $Z'_t := (1/C_t)Z_t \in \Z$ and $E_t = 1 + (1/C_t)Z_t = 1 + Z'_t$.
Thus, we have a bijection between $Z_t\in\Z$ such that $C_t + Z_t \geq 0$, and $E_t \in \Evar(\Z) := \{1 + Z \mid Z\in\Z, Z\geq -1\}$.
This definition more directly aligns with the usual definition of e-variables: a ``fair'' multiplicative gamble $E$ taking values in $[0,\infty]$ with $\Egu E \leq 1$.
\footnote{Technically we have $\Egu E = \infty$ when $E$ takes on value $\infty$, but we can approximate such an $E$ arbitrarily closely with $[0,\infty)$-valued e-variables.}
The key advantage of the multiplicative framing is that the no-bankcruptcy constraint $C_t + Z_{t+1} \geq 0$, which depends on $C_t$ and thus on the entire outcome sequence thus far, is replaced by $E_{t+1} \in \Evar(\Z)$, a constraint which is invariant across rounds.
In other words, when the gambles $\Z$ are available in every round, the multiplicative gambles available in each round are simply $\Evar(\Z)$.
Without scalability, the multiplicative gambles on round $t$ would generally depend on $C_t$.

A natural process that arises from this viewpoint is the log capital sequence $\{L_t\}_t$ given by $L_t = \log C_t \in \extreals$, which captures the exponential growth rate of capital.
Given scalable additive gambles $\Z$, and multiplicative gambles $\Evar(\Z)$, we may naturally define $\Z^{\log} = \{\log E \mid E\in\Evar(\Z)\}$, which turns the study of exponential growth back into additive gambles: $L_t = \log C_0 + \sum_{i\leq t} \log E_i$, where now each $\log E_i \in \Z^{\log}$; see Example~\ref{ex:lln-gambles-multiplicative} and \S~\ref{sec:online-learning}.
This exponential growth viewpoint was introduced by \citet{kelly1956new}, now commonly called ``Kelly betting'', and is the dominant measure of power in game-theoretic statistics~\citep{ramdas2023game,grunwald2024safe}.

\begin{example}[Multiplicative LLN strategy]
  \label{ex:lln-gambles-multiplicative}
  To make this discussion concrete, consider the setting in \S~\ref{sec:motivating-example} with $\Y = [-1,1]$ and $\Z = \{y \mapsto \beta y \mid \beta \in \reals\}$.
  Here the ``additive'' gambles $Z_t$ are parameterized by the choice $\beta_t$.
  As $\Z$ is scalable, the corresponding set of multiplicative gambles is simply $\Evar(\Z) := \{y\mapsto 1 + \alpha y \mid \alpha \in [-1,1]\}$.
  For a nonnegative capital process $\{C_t\}_t$, we can define the multiplicative ``reinvestment'' $\alpha_t := \beta_t / C_t$, so that $E_t = 1 + (1/C_t) Z_t = 1 + \alpha_t y_t$ and again $C_t = \prod_{i\leq t} E_i$.
  (We set $\alpha_t = 0$ when $C_t = 0$.)
  As $C_{t+1} \geq 0$, we must have $|\beta_t|\leq C_t$, and thus $\alpha_t \in [-1,1]$ and $E_t \in \Evar(\Z)$.
  Strategies, especially log optimal strategies, can be more natural to state multiplicatively.
  For example, one choice of $\psi^*$ in \S~\ref{sec:motivating-example} is the Krichevsky--Trofimov estimator which simply sets $\alpha_t = \tfrac{1}{t} \sum_{i=1}^{t-1} y_i$~\citep{krichevsky1981sequential,orabona2016coin}.
  The log capital $L_t = \sum_{i\leq t} \log(1+\alpha_ty_t)$, itself the capital process on gamble space $\Z^{\log} = \{y\mapsto \log(1+\alpha y) \mid \alpha\in[-1,1]\}$, appears
again in \S~\ref{sec:online-learning}, where we will see how to derive such strategies via online machine learning.
\end{example}

As in \S~\ref{sec:motivating-example} and Example~\ref{ex:lln-gambles-multiplicative}, a particularly common source of nonnegative variables $X$ are indicators $\ones_A$.
When gambles are scalable, the replication representation of $\Pgu A$ (using Proposition~\ref{prop:egu-equiv-defs}, as $\ones_A\in\reals^\Omega$) simplifies to an important form,
\begin{align}\label{eq:upper-prob-scalable}
  \Pgu A
  &= \inf\,\{\alpha > 0 : \exists Z\in\Z \text{ s.t. } \alpha+Z \geq  \ones_A\}
  \\
  &= \inf\left\{\alpha > 0 : \exists Z\in\Z \text{ s.t. } 1+Z \geq \frac 1 \alpha \ones_A\right\}
  \\
  &= \inf\left\{\alpha > 0 : \exists E\in\Evar(\Z) \text{ s.t. } E \geq \frac 1 \alpha \text{ on } A\right\}~.
\end{align}
In other words, Gambler may gamble without risking bankruptcy, and scale her initial capital by (arbitrarily close to) $1/\Pgu A$ when $A$ occurs.
When $\Pgu A = 0$, Gambler can make her capital grow arbitrarily large.
The expressions above establish the converse of these statements as well (see \S~\ref{sec:prob-one}).
The existence of $\psi^*$ above thus exhibits $\Pgu (\ALLN)^c = 0$, i.e., $\ALLN$ holds g.t.a.s.

\subsection{Examples}
\label{sec:examples}

As discussed in \S~\ref{sec:intro-composite-ville}, given a measurable gamble space $(\Omega,\Z)$, meaning each $Z\in\Z$ is measurable, we may define the set of \emph{consistent} probability measures by $\Delta_0(\Z) = \{P\in\Delta(\Omega) \mid \E_P Z \leq 0 \;\forall Z\in\Z\}$.
When we define $\Delta_0$ in \S~\ref{sec:consistency-seq-def}, we will lift the restriction that $\Z$ be measurable.

\begin{example}[Fair coin]\label{ex:coin-formal}
  Let us see how to represent a ``fair coin'', via the gamble space $(\Omega,\Z)$ given by outcomes $\Omega = \{-1,1\}$, and gambles $\Z = \{Z_\beta: \omega\mapsto \beta\omega \mid \beta\in\reals\}$.
  Here the event $A_\Ht = \{1\}$ represents heads, and $A_\Tt = \{-1\}$ tails.
  A gamble $Z_\beta\in\Z$ pays $\beta$ upon heads and $-\beta$ upon tails.

  To check our intuition that these are fair gambles, we can compute the set of consistent probability measures $\Delta_0(\Z)$.
  Since $\omega\mapsto\omega, \omega\mapsto-\omega \in \Z$, we have $\Delta_0(\Z) = \{P\}$ where $P(A_\Ht) = P(A_\Tt) = 1/2$, as desired.
  
  Let us compute the upper game-theoretic probability of heads:
  \begin{align}\label{eq:coin-formal-upper}
    \Pgu A_\Ht = \Egu \ones_{\{1\}} = \inf_{\beta\in\reals} \sup_{\omega\in\{-1,1\}} \ones\{\omega = 1\} - \beta\omega = \inf_{\beta\in\reals} \max(1-\beta,\beta)~.
  \end{align}
  Clearly $\beta = 1/2$ minimizes the right-hand side, giving $\Pgu A_\Ht = 1/2$.
  From the perspective of replication, the choice $\beta = 1/2$ corresponds to a gambling strategy that starts with \$$1/2$ and bets \$$1/2$ on heads, yielding a net \$1 when $\omega=1$ and \$0 if $\omega=-1$, as desired.

  Turning now to $\Pgl A_\Ht$, again we compute:
  \begin{align}
    \Pgl A_\Ht = \Egl \ones_{\{1\}} = \sup_{\beta\in\reals} \inf_{\omega\in\{-1,1\}} \ones\{\omega = 1\} + \beta\omega = \sup_{\beta\in\reals} \min(1+\beta,-\beta)~,
  \end{align}
  which is achieved by $\beta = 1/2$, giving $\Pgl A_\Ht = 1/2$ and thus $\Pg A_\Ht = 1/2$.
  In summary then, we indeed have a fair game-theoretic coin.

  One can add a bias $\epsilon \in [-1/2,1/2]$ to the coin by tilting the betting odds.
  Specifically, for $\Z_\epsilon = \{ Z_\beta: \omega\mapsto \beta(\omega-2\epsilon) \mid \beta\in\reals\}$, we will have $\Pg A_\Ht = 1/2 + \epsilon$.
  Similarly, we have $\Delta_0(\Z_\epsilon) = \{P_\epsilon\}$ where $P_\epsilon(A_\Ht) = 1/2 + \epsilon$.
\end{example}

The fact that the game-theoretic probabilities matched the measure-theoretic ones above is an instance of minimax duality.
We can verify this duality directly using e.g.\ Sion's minimax theorem (Theorem~\ref{thm:sion}),
\begin{align}\label{eq:coin-minimax}
  \Pgu A_\Ht
  &= \inf_{\beta\in\reals} \sup_{\omega\in\{-1,1\}} \ones\{\omega = 1\} - \beta(\omega - 2\epsilon)
  \\
  &= \inf_{\beta\in\reals} \sup_{P\in\Delta(\{-1,1\})} \E_P \left[\ones\{\omega = 1\} - \beta(\omega-2\epsilon)\right]
  \\
  &= \sup_{P\in\Delta(\{-1,1\})} \inf_{\beta\in\reals} \E_P \left[\ones\{\omega = 1\} - \beta(\omega-2\epsilon)\right]
  \\
  &= \sup_{p\in[0,1]} \inf_{\beta\in\reals} p - \beta((p-(1-p))-2\epsilon)
  \\
  &= \sup_{p\in[0,1]} \inf_{\beta\in\reals} p - \beta(2p-1-2\epsilon)
  \\
  &= 1/2 + \epsilon~.
\end{align}
In the final equality, since Gambler is free to choose any $\beta\in\reals$, the optimal $p$ must therefore set $2p-1-2\epsilon=0$, giving $p^*=1/2 + \epsilon$.
Indeed, this final observation is equivalent to $\Delta_0(\Z_\epsilon) = P_\epsilon$ as defined above.
As we will see, World must choose a consistent measure $P$ whenever the gamble space satisfies a scaling property (Theorem~\ref{thm:ep0-equals-ep}).
\footnote{It could be that $\Z$ fails to have any consistent probability measures.
  For example, if $\epsilon>1/2$ above then $Z_\beta > 0$ when $\beta<0$.
  Indeed, one has $\Pgu A_\Ht = -\infty$ in this case, and in fact $\Egu X = -\infty$ for all real-valued $X$.
  }

\begin{figure}[t]
  \centering
  \begin{tikzpicture}
    \begin{axis}[
      axis lines = middle,
      xlabel = $\omega$,
      ylabel = $y$,
      xmin = -1.1, xmax = 1.1,
      ymin = -1.1, ymax = 1.1,
      domain = -1:1,
      samples = 100,
      grid = both,
      grid style = {dashed, gray!30},
      legend pos = north west,
      legend style={font=\small},
      tick label style={font=\small}
      ]

      \addplot [blue, thick] {x^3} node[pos=0.9, right] {$X'$};
      \addlegendentry{$\omega^3$}

      \addplot [red, thick] {0.75*x + 0.25}; 
      \addlegendentry{$\frac{3}{4}\omega + \frac{1}{4}$}

      \addplot [orange, thick] {0.75*x - 0.25}; 
      \addlegendentry{$\frac{3}{4}\omega - \frac{1}{4}$}

      \node[circle, fill=red, inner sep=1.5pt, label={right:{\raisebox{-15pt}{$\Egu X'$}}}] at (axis cs:0, 0.25) {};
      \node[circle, fill=orange, inner sep=1.5pt, label={left:{\raisebox{5pt}{$\Egl X'$}}}] at (axis cs:0, -0.25) {};

    \end{axis}
  \end{tikzpicture}
  \label{fig:outcome-interval}
  \caption{A visualization of Example~\ref{ex:single-interval} showing a nontrivial price gap $\Egl X' = -\tfrac 1 4 < \tfrac 1 4 = \Egu X'$.  We can conveniently plot both the upper and lower expectations by reflecting the latter; rather that plotting the smallest affine function dominating $-X'$, we plot the largest affine lower bound.}
\end{figure}
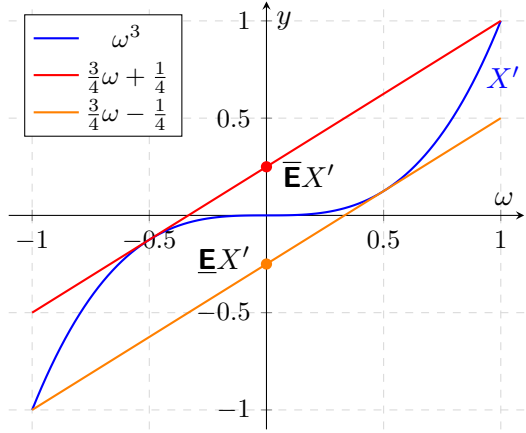

\begin{example}[Outcome interval]\label{ex:single-interval}
  The gambles in each round of the LLN game in \S~\ref{sec:motivating-example} took the form of a gamble space $(\Omega,\Z)$ given by $\Omega = [-1,1]$ and $\Z = \{ \omega\mapsto \beta\omega \mid \beta\in\reals\}$.
  More generally, we can consider $\Omega \subseteq \reals$ and $\Z = \{ \omega\mapsto \beta(\omega-c) \mid \beta\in\reals\}$ for some fixed $c\in\reals$.
  Intuitively, these gambles mean the outcome has ``mean'' equal to $c$.
  We can again check this intuition by computing the consistent probability measures: $\Delta_0(\Z) = \{P\in\Delta(\Omega) \mid \E_P X = c\}$ where  $X:\omega\mapsto\omega$ is the identity variable.
  Let us check that $\Eg X = c$ as well.
  As Gambler can replicate $X$ with initial capital $c$, and $-X$ with initial capital $-c$, by choosing $\beta = 1$ and $\beta = -1$, respectively.
  More formally $X - Z_{1} = c$ and $-X - Z_{-1} = -c$, so we have $\Egu X \leq c$ and $\Egu(-X) \leq -c$, and thus $\Egl X = -\Egu(-X) \geq c$.
  At this point we can directly verify the remaining inequalities $\Egu X \geq c$ and $\Egl X \leq c$, or appeal to the fact that $\Egu \geq \Egl$ in this case (Remark~\ref{rem:upper-lower-inequality}).

  While the variable $X : \omega \mapsto \omega$ has a game-theoretic expectation, most variables in this gamble space do not.
  For example, set $c=0$ for simplicity, and consider $X'(\omega) = \omega^3$.
  Here one can calculate the optimal $\beta$ as $3/4$, so that $(3/4)\omega + 1/4 \geq X$.
  We conclude $\Egu X' = 1/4$.
  By symmetry, the smallest affine function dominating $-X'$ on $[-1,1]$ is $-(3/4)\omega + 1/4 \geq X'$, giving $\Egu (-X') = 1/4$ and thus $\Egl X' = -1/4$.
  
  We can think of upper expectations as prices, the cost of replicating $X$ under a particular assumption about available gambles (or the power of World; see \S~\ref{sec:prices-probability}).
  Through this lens, the variable $X$ is fully priced, whereas $X'$ has a nontrivial ``bid-ask spread'' of $[-1/4,1/4]$.
  One can see that only variables which are affine functions will be fully priced in this gamble space; the rest will have nontrivial spreads.
  
\end{example}

\begin{example}[Variance]\label{ex:single-variance}
  Similarly, we may take $(\Omega=\reals,\Z)$ with $\Z = \{ \omega\mapsto \beta(\omega-c) + \alpha((\omega-c)^2-v) \mid \alpha,\beta\in\reals\}$ for some fixed $c,v\in\reals$.
  Letting $X:\omega\mapsto\omega$, we now have both $\Eg X = c$ and $\Eg (X-c)^2 = v$, by the same logic as Example~\ref{ex:single-interval}.
  If instead we had restricted $\alpha \geq 0$ in $\Z$, then we would only have $\Egu (X-c)^2 \leq v$; informally, there would be no way to ``short'' the quadratic variation of $X$, and profit when e.g.\ $X=c$.

  More generally, if we have $f\in\Z$ for any $f:\Omega\to\reals$, then $\Egu f \leq \sup (f-f) = 0$.
  (The restriction $f:\Omega\to\reals$ is important, as $f-f \neq 0$ if $f$ takes on infinite values.)
  In the $\Z$ above, we have the following elements of $\Z$: $\omega\mapsto\omega-c$, $\omega\mapsto-(\omega-c)$,
  $\omega\mapsto((\omega-c)^2-v)$,
  $\omega\mapsto-((\omega-c)^2-v)$,
  which give the conclusions above.
\end{example}

Example~\ref{ex:single-variance} raises an interesting question: if we only need these four gambles to control the game-theoretic mean and variance of $X$, why do we allow Gambler to scale and combine these gambles with coefficients $\alpha,\beta\in\reals$?
Indeed, this ``positive linearity'' of gambles is ubiquitous in the literature.
The reason has to do with the behavior of $\Egu$ on variables other than these four gambles, and the often-implicit goal that $\Egu$ align with some analog of a measure-theoretic expectation.
For example, the game-theoretic proof of Markov's inequality $\Pgu[X \geq \alpha] \leq \Egu X / \alpha$ for $X\geq 0$ requires $\Z$ to be scalable (Definition~\ref{def:gamble-space-conditions}).
The same is true for sufficient conditions and characterizations of events holding g.t.a.s.\ (game-theoretically almost surely); see Lemma~\ref{lem:prob-one} and Proposition~\ref{prop:prob-one-iff}.
In \S~\ref{sec:prices-probability}, we will see two deeper reasons why positive linearity is important for $\Egu$ to align with $\Ep$: (i) for minimax duality to hold, and (ii) for Gambler to be able to punish World for choosing a measure inconsistent with the gambles.

\begin{example}[Proper scoring rule]
  
  Consider a different binary outcome example, again with $\Omega = \{-1,1\}$, but this time with gambles
  $\Z = \{\omega\mapsto S(p,\omega)-S(p^*,\omega) \mid p\in[0,1]\}$
  where $S(p,\omega)$ is a strictly proper scoring rule~\citep{savage1971elicitation,gneiting2007strictly}.
  ($S$ is strictly proper if $\{p^*\} = \argmax_{p\in[0,1]} p^* S(p,1) + (1-p^*) S(p,-1)$.)
  Two prominent examples are the quadratic score $S(p,\omega) = p(\omega+1) - p^2$ and the log scoring rule (negative log loss), given by $S(p,1) = \log p$ and $S(p,-1) = \log(1-p)$.
  In this setting, Gambler may choose any prediction $p\in[0,1]$ and receives the excess score of $p$ relative to that of the baseline prediction $p^*$.
  By definition of strictly proper, we have $\Delta_0(\Z) = \{P^*\}$ where $P^*(\{1\}) = p^*$.
  Intuitively, for $p^*=1/2$, Gambler would only take a bet if she thought the coin was not fair, just as in Example~\ref{ex:coin-formal}.

  For log score and $p^*=1/2$, we can simplify these gamble via the parameterization $\beta = 2p-1$, giving:
  \begin{align*}
    S(p,\omega) - S(\tfrac 1 2,\omega)
    &=
      \begin{cases}
        \log 2p & \omega=1
        \\
        \log 2(1-p) & \omega=-1
      \end{cases}
    = \log(1 + \beta\omega)~.      
  \end{align*}
  We can thus rewrite this gamble space as $\Z = \{\omega\mapsto \log(1+\beta \omega) \mid \beta \in [-1,1]\}$.
  In this form, it is clear how one could generalize to $\Omega = [-1,1]$, via the same form for $\Z$.
  This gamble space is now exactly $\Z^{\log}$ from Example~\ref{ex:lln-gambles-multiplicative}, which captures the growth rate of capital in the setting of \S~\ref{sec:motivating-example}.
  We will see $\Z^{\log}$ again in \S~\ref{sec:online-learning}.
\end{example}

\subsection{Conditions and basic facts}
\label{sec:formalism-basic-facts}

We begin with two facts that require no assumptions on the gamble space.

\begin{proposition}\label{prop:basic-facts-no-assumptions}
  Let $(\Omega,\Z)$ be a gamble space and $X,Y:\Omega\to\extreals$ variables.
  Then we have
  \begin{enumerate}
  \item\label{fact:translation} $\Egu (X + c) = \Egu X + c$ for all $c\in\reals$ (translation)
  \item\label{fact:monotonicity} $X \leq Y \implies \Egu X \leq \Egu Y$ (monotonicity).
  \end{enumerate}
\end{proposition}
\begin{proof}
  We have by definition
  \begin{align*}
    \Egu (X+c)
    &= \inf_{Z\in\Z} \sup_{\omega\in\Omega} (X(\omega)+c) - Z(\omega)
    \\
    &= \left(\inf_{Z\in\Z} \sup_{\omega\in\Omega} X(\omega) - Z(\omega)\right) + c
      = \Egu X + c~.
  \end{align*}
  
  For the second statement, for all $Z\in\Z$ and $\omega\in\Omega$ we have $Y(\omega) - Z(\omega) \geq X(\omega) - Z(\omega)$.
  Hence for all $Z\in\Z$ we have $\sup_{\omega\in\Omega} Y(\omega) - Z(\omega) \geq \sup_{\omega\in\Omega} X(\omega) - Z(\omega)$.
  Finally,
  \begin{align*}
    \Egu Y
    = \inf_{Z\in\Z} \sup_{\omega\in\Omega} Y(\omega) - Z(\omega)
    \geq \inf_{Z\in\Z} \sup_{\omega\in\Omega} X(\omega) - Z(\omega) = \Egu X~.
  \end{align*}
\end{proof}

One may recognize translation and monotonicity as the base assumptions on financial risk measures.
Indeed, game-theoretic upper expectations are exactly financial risk measures, modulo a minus sign and the allowance of infinite-valued variables; see \S~\ref{sec:financial-risk-measures}.

\begin{definition}[Conditions on gamble spaces]
  \label{def:gamble-space-conditions}
  Let $(\Omega,\Z)$ be a gamble space.
  We define the following conditions on $\Z$.
  \begin{enumerate}
  \item Contains zero: $0\in\Z$.
  \item Arbitrage-free: $Z\in\Z \implies$ $\inf Z \leq 0$.
  \item Normalized: arbitrage-free and contains zero.
  \item Scalable: $Z\in\Z, \alpha\geq 0 \implies \alpha Z\in\Z$;\\
    Downward scalable and upward scalable: $\alpha \in [0,1]$ and $[1,\infty)$, respectively.
  \item Positive-linear: $Z_1,Z_2\in\Z, \alpha_1,\alpha_2 \geq 0 \implies \alpha_1Z_1 + \alpha_2Z_2 \in \Z$.
  \item Bounded-below: $Z\in\Z \implies$ $\inf Z > -\infty$.
  \end{enumerate}
\end{definition}

The term ``normalized'' is convenient shorthand for a sufficient condition implying $\Egu\, 0 = 0$ (see Proposition~\ref{prop:basic-facts-zero}); one can often substitute the weaker condition $\sup_{Z\in\Z} \inf Z = 0$.
Positive linearity is equivalent to $\Z$ being a convex cone.
In particular, positive linearity implies that $\Z$ is scalable and contains zero.
We will often assume that gamble spaces are normalized, and further require some of the other conditions.

\begin{proposition}\label{prop:basic-facts-zero}
  Let gamble space $(\Omega,\Z)$ be given.
  Then we have
  \begin{enumerate}
  \item $\Egu\, 0 \geq 0$ if and only if $\Z$ is arbitrage-free.
  \item $\Egu\, 0 \leq 0$ when $\Z$ contains zero.
  \end{enumerate}
  In particular, $\Egu\, c = c$ for all $c\in\reals$ when both conditions hold ($\Z$ is normalized).
\end{proposition}
\begin{proof}
  \mbox{}
  \begin{enumerate}
  \item
    The arbitrage-free condition, $\inf_{\omega\in\Omega} Z(\omega) \leq 0$ for all $Z\in\Z$, is equivalent to $\sup_{Z\in\Z} \inf_{\omega\in\Omega} Z(\omega) \leq 0$.
    This condition is satisfied if and only if
    \begin{align*}
      \Egu 0 = \inf_{Z\in\Z} \sup_{\omega\in\Omega} 0 - Z(\omega) = -\sup_{Z\in\Z} \inf_{\omega\in\Omega} Z(\omega) \geq 0~.
    \end{align*}
  \item 
    If $\Z$ contains $0$, we have $\displaystyle\sup_{Z\in\Z} \inf_{\omega\in\Omega} Z(\omega) \geq \sup_{Z\in\{0\}}\inf_{\omega\in\Omega} Z(\omega) = 0$.
  \end{enumerate}
  Finally, when both conditions hold, we have $\Egu\, 0 = 0$, and thus $\Egu\, c = c$ for all $c\in\reals$ by Proposition~\ref{prop:basic-facts-no-assumptions} (translation).
\end{proof}

We now state several properties of $\Egu$ that show it can behave similarly to its measure-theoretic counterpart.
We begin with 1-homogeneity.
(See also E2 in \S~\ref{sec:axioms-fa}.)

\begin{proposition}\label{prop:basic-facts-homogeneity}
  Let gamble space $(\Omega,\Z)$ and variable $X:\Omega\to\extreals$ be given.
  Then
  \begin{enumerate}
  \item\label{fact:1-homog} $\Egu(cX) = c\, \Egu X$ for all $c\geq 0$ (1-homogeneity) when $\Z$ is arbitrage-free and scalable.
  \end{enumerate}
  When $\Z$ is scalable, the statement holds for $c>0$ (positive homogeneity).
\end{proposition}
\begin{proof}
  For $c>0$, by scalability we have $(1/c)\Z = \Z$, giving
  \begin{align*}
    \Egu (cX)
    &= \inf_{Z\in\Z} \sup_{\omega\in\Omega} c X(\omega) - Z(\omega)
    \\
    &= \inf_{Z\in (1/c)\Z} \sup_{\omega\in\Omega} c X(\omega) - c Z(\omega)
    \\
    &= c \inf_{Z\in \Z} \sup_{\omega\in\Omega} X(\omega) - Z(\omega) \\
    &= c\, \Egu X~.
  \end{align*}
  As $\Z$ contains zero, Proposition~\ref{prop:basic-facts-zero} gives the case $c=0$ when $\Z$ is additionally arbitrage-free.
\end{proof}
While positive homogeneity does not require gambles to be arbitrage-free, it is somewhat trivial, since $\inf Z > 0$ for some $Z$, so $\Egu X = -\infty$ unless possibly when $\sup X = \infty$.

\begin{proposition}\label{prop:basic-facts-assumptions}
  Let $(\Omega,\Z)$ be an arbitrage-free, positive-linear gamble space.
  For all $X,Y:\Omega\to\extreals$ we have
  \begin{enumerate}
  \item\label{fact:subadditivity} $\Egu (X + Y) \leq \Egu X + \Egu Y$ (subadditivity).
  \item\label{fact:lower-leq-upper}
    $\Egl X \leq \Egu X$.
  \end{enumerate}
\end{proposition}
\begin{proof}
  \mbox{}
  \begin{enumerate}
  \item
    From Proposition~\ref{prop:egu-x-infinite}, if $X$ or $Y$ take on $\infty$, then we have $\Egu X + \Egu Y = \infty = \Egu[X+Y]$.
    From Proposition~\ref{prop:gambles-not-neg-infty}, we may assume without loss of generality that $Z > -\infty$ for all $Z\in\Z$.
    Thus, expressions of the form $X - Z$ only take values in $\reals\cup\{-\infty\}$.
    
    By definition, we have $\Egu X = \inf_{Z \in \Z} \sup X - Z$ and $\Egu Y = \inf_{Z \in \Z} \sup Y - Z$; let $\{Z_1^{n}\}_{n\in\N}$ and $\{Z_2^{n}\}_{n\in\N}$ be sequences in $\Z$ achieving these infima.
    Then we have
    \begin{align*}
      \Egu X + \Egu Y
      &= \lim_{n\to\infty} \left[\sup_{\omega \in \Omega} X(\omega) - Z_1^n(\omega)\right] + \lim_{n\to\infty} \left[\sup_{\omega \in \Omega} Y(\omega) - Z_2^n(\omega)\right]
      \\
      &= \lim_{n\to\infty} \left[\sup_{\omega \in \Omega} \left(X(\omega) - Z_1^n(\omega)\right) + \sup_{\omega \in \Omega} \left(Y(\omega) - Z_2^n(\omega)\right)\right]
      \\
      &\geq \lim_{n\to\infty} \sup_{\omega \in \Omega} \left[(X(\omega) - Z_1^n(\omega)) + (Y(\omega) - Z_2^n(\omega))\strut\right]
      \\
      &= \lim_{n\to\infty} \sup_{\omega \in \Omega} \left[X(\omega) + Y(\omega) - (Z_1^n+Z_2^n)(\omega)\strut\right]
      \\
      &\geq \inf_{Z\in\Z} \sup_{\omega \in \Omega} \left[(X+Y)(\omega) - Z(\omega)\strut\right]
      \\
      &= \Egu[X + Y]~,
    \end{align*}
    where we used the fact that $Z_1^n + Z_2^n \in \Z$ for all $n\in\N$ by positive linearity.

  \item From Proposition~\ref{prop:egu-x-infinite}, if $X(\omega)\in\{-\infty,\infty\}$ for any $\omega\in\Omega$, the statement holds trivially, as either $\Egl X = -\infty$ or $\Egu X = \infty$.
    It remains to consider $X:\Omega\to\reals$, in which case $X-X = 0$.
    By Proposition~\ref{prop:basic-facts-zero} and subadditivity, we have $0 = \Egu(X - X) \leq \Egu X + \Egu(-X) = \Egu X - \Egl X$.
  \end{enumerate}
\end{proof}

\begin{remark}\label{rem:upper-lower-inequality}
  The condition $\Egl X \leq \Egu X$ holds under weaker conditions.
  In particular, we show in \S~\ref{sec:conditions-upper-lower} that $\Egu \geq \Egl$ if and only if $\Z+\Z$ (the Minkowski sum) is arbitrage-free.
  We will also see a simple example where the inequality fails despite $\Z$ being arbitrage-free.
\end{remark}

\begin{proposition}\label{prop:basic-facts-probability}
  Let $(\Omega,\Z)$ be a gamble space, and let $A\subseteq\Omega$.
  Then
  \begin{enumerate}
  \item $\Pgl A = 1 - \Pgu A^c$.
  \end{enumerate}
  If gambles are furthermore arbitrage-free and positive-linear, then
  \begin{enumerate}[resume]
  \item $0 \leq \Pgl A \leq \Pgu A \leq 1$.
  \item $\Pgu A + \Pgu A^c \geq 1$.
  \end{enumerate}
\end{proposition}
\begin{proof}
  Again we prove each in turn.
  \begin{enumerate}
  \item We appeal to Proposition~\ref{prop:basic-facts-no-assumptions} (translation).
    $\Pgl A = \Egl \ones_A = -\Egu(-\ones_A) = -\Egu(1-\ones_A) + 1 = 1 - \Egu \ones_{A^c} = 1 - \Pgu A^c$.
  \item From Proposition~\ref{prop:basic-facts-assumptions}, we have $\Pgl A \leq \Pgu A$.
    From Proposition~\ref{prop:basic-facts-zero} and Proposition~\ref{prop:basic-facts-no-assumptions} (translation) we have $\Egu 0 = 0$ and $\Egu 1 = 1$.
    Combining with Proposition~\ref{prop:basic-facts-no-assumptions} (monotonicity), we have $\Pgl A = \Egl \ones_A = -\Egu(-\ones_A) \geq -\Egu 0 = 0$ and $\Pgu A = \Egu \ones_A \leq \Egu 1 = 1$.
  \item From the above we have $\Pgu A + \Pgu A^c = \Pgu A + (1 - \Pgl A) \geq 1$.
  \end{enumerate}
\end{proof}

\subsection{Ville and almost sure events}
\label{sec:prob-one}

Recall that we write ``$A$ g.t.a.s.'' (game-theoretically almost surely) if $\Pgl A = 1$, or equivalently, $\Pgu A^c = 0$ (Proposition~\ref{prop:basic-facts-probability}).
In \S~\ref{sec:introduction}, we gave a gambling strategy risking at most \$1 that produces infinite wealth on $(\ALLN)^c$, and said this strategy implies $\ALLN$ g.t.a.s.
We now formalize the connection between these statements, along with several generalizations reminiscent of Ville's Theorem.
(See \S~\ref{sec:cond-expect-supermtg} for sequential versions.)
In particular, we will show for downward-scalable gambles spaces that $\Pgu A \leq \alpha$ if and only if for any $c < 1/\alpha$ there is a gambling strategy which, when given an initial capital of 1, does not risk bankruptcy and achieves capital at least $c$ on $A$.

\begin{lemma}\label{lem:prob-one}
  Let $(\Omega,\Z)$ be a downward-scalable gamble space, and $A\subseteq\Omega$.
  \begin{enumerate}
  \item If there exists $Z^*\in\Z$ bounded from below such that $Z^*(\omega) = \infty$ for all $\omega\notin A$, then $A$ holds g.t.a.s.
  \item For $\alpha \in (0,1]$, if there exists $Z^*\in\Z$ with $Z^* \geq -1$ such that $1 + Z^*(\omega) \geq 1/\alpha$ for all $\omega\in A$, then $\Pgu A \leq \alpha$.
  \end{enumerate}
\end{lemma}
\begin{proof}
  We begin with the second statement.
  Let $Z = \alpha Z^*$.
  By downward scalability, $Z \in \Z$.
  As $1 + Z^* \geq 0$, we have $\alpha + Z = \alpha(1 + Z^*) \geq 0$.
  For $\omega\in A$, we have $\alpha + Z(\omega) \geq \alpha(1+Z^*(\omega)) \geq 1$ by assumption.
  Thus $\alpha + Z \geq \ones_A$, giving $\Pgu A \leq \alpha$.

  For the first statement, let $b = \inf Z^*$.
  Let $Z' = Z^*$ if $b \geq -1$, and $Z' = (-1/b) Z^*$ otherwise.
  We have $Z' \in \Z$ by downward scalability, and $Z' \geq -1$ in both cases.
  As $Z'$ is infinite on $A^c$, this $Z'$ satisfies the condition of the second statement for $A^c$ and all $\alpha > 0$ simultaneously.
  Hence $\Pgu A^c \leq \alpha$ for all $\alpha > 0$, giving the result.
  \footnote{If $\Z$ is additionally normalized, the first statement becomes $\Pg A = 1$, with $\Pgl A \leq 1$ following from Proposition~\ref{prop:basic-facts-no-assumptions} and~\ref{prop:basic-facts-zero}.}
\end{proof}

It is natural to ask whether the converse holds.
For instance, does every g.t.a.s.\ $A$ have a corresponding gambling strategy that starts at 1, does not risk banruptcy, and becomes infinite when the $A$ fails?
The answer is not quite, because it could be that $\Pgl A = 1$ yet the supremum implicit in that statement is only achieved by a sequence of gambling strategies.
For example, if we have $\Omega = [0,1]$, $\Z = \{\omega \mapsto \beta \omega \mid \beta \in \reals\}$, then $\Pgl [0,1/2] = 1$ but no $Z\in\Z$ can take on infinite values on $(1/2,1]$.
(See Example~\ref{ex:minimax-fail-omega-01}.)

Weakening the conditions slightly to allow for sequences does indeed give a converse.
The first statement becomes:
$\Pgl A = 1$ if and only if there is a gambling strategy risking only \$1 (or any bounded amount) and earning an arbitrarily large amount on $A^c$.
The converse does rely on upward scalability, however, as one must be able to scale up any strategy that replicates the indicator variable for $A$.
We may also insist that the strategy risk less and less money.

\begin{proposition}\label{prop:prob-one-iff}
  \everymath{\displaystyle}
  Let $(\Omega,\Z)$ be a scalable gamble space.
  Let $A\subseteq\Omega$ and $\alpha > 0$.
  Then the following are equivalent:
  \begin{enumerate}
  \item $\Pgu A \leq \alpha$
  \item For all $\epsilon > 0$ there exists $Z^*\in\Z$
  with $Z^* \geq -1$ and $1 + Z^* \geq 1/\alpha - \epsilon$ on $A$.
  \end{enumerate}
  In other words, we may write
  $\Pgu A = \inf\{\alpha > 0 \mid \exists Z\in\Z, Z \geq -1, 1 + Z \geq \tfrac 1 \alpha \text{ on } A\}$~.
\end{proposition}
\begin{proof}
  $2 \implies 1$:
  As before let $Z' = \alpha Z^*$.
  By similar reasoning, $\alpha(1 + Z^*) \geq 0$ and on $A$ we have $\alpha(1 + Z^*) \geq 1 - \alpha\epsilon$.
  So $\alpha(1+\epsilon) + Z' \geq \ones_A$, whence $\Pgu A \leq \alpha(1 + \epsilon)$.
  Taking $\epsilon\to 0$ gives the result.
  
  $1 \implies 2$:
  Let $\epsilon > 0$.
  From eq.~\eqref{eq:prop-upper-ex-rep-cost}, for all $\delta>0$ we have $Z$ such that $\alpha + \delta + Z \geq \ones_A$.
  We will take $\delta \leq \epsilon \alpha^2 / (1 - \epsilon \alpha)$.
  Let $Z^* = \tfrac 1 {\alpha+\delta} Z$.
  Then $Z^* \geq -1$.
  On $A$, we have $Z \geq 1 - \alpha - \delta$, giving
  $Z^* = \tfrac 1 {\alpha+\delta} Z \geq \tfrac {1 - \alpha - \delta} {\alpha+\delta} = \tfrac 1 {\alpha+\delta} - 1 \geq$ on $A$.
  By our choice of $\delta$, we have $\tfrac 1 {\alpha+\delta} \geq \tfrac 1 \alpha - \epsilon$, giving the result.
\end{proof}

Observing that $\Pgu A^c \leq 0 \iff \Pgu A^c \leq \alpha$ for all $\alpha > 0$, we have the following.
\begin{corollary}\label{cor:prob-one-iff}
  Let $(\Omega,\Z)$ be a scalable gamble space.
  Let $A\subseteq\Omega$.
  Then the following are equivalent:
  \begin{enumerate}
  \item $A$ holds g.t.a.s.\ (i.e., $\Pgl A = 1$)
  \item For all $c > 0$ there exists $Z^*\in\Z$
  with $Z^* \geq -1$ and $1 + Z^* \geq c$ on $A^c$.
  \end{enumerate}
\end{corollary}
By scalability, one could strengthen the requirements on $Z^*$ in Corollary~\ref{cor:prob-one-iff} to require an arbitrarily small capital risk, e.g. $Z^* \geq -1/c$, without changing the result.

\citet[Proposition 6.7]{shafer2019game} give sufficient conditions, namely a certain continuity property (see Axioms E1--E5 in \S~\ref{sec:axioms-fa}), for Lemma~\ref{lem:prob-one}(1) to be an if and only if.
The converse of Lemma~\ref{lem:prob-one}(2) does not generally hold, however.
Thus, in light of the axiomatic approach of \citet{shafer2001probability,shafer2019game}, it is interesting that only scalability (essentially their E2, E3, and E4) is needed for Proposition~\ref{prop:prob-one-iff} and Corollary~\ref{cor:prob-one-iff}.

\subsection{Understanding the infinite cases}
\label{sec:infinite-cases}

We will find it helpful to rule out various corner cases that arise when $X$ and $Z$ take on infinite values.
These cases bear similarity to those already discussed in Proposition~\ref{prop:egu-equiv-defs}; here we focus more on ruling out cases that result in suboptimal strategies for either Gambler or World.
The first observation is that without loss of generality Gambler can ignore gambles which result in an infinite replication cost.

\begin{lemma}\label{lem:remove-bad-gambles}
  Let $(\Omega,\Z)$ be a gamble space and $X : \Omega \to \extreals$ a variable.
  Then $\Egu_\Z X = \Egu_{\Z\setminus\Z'} X$ where $\Z' \subseteq \{Z\in\Z \mid \sup (X - Z) = \infty\}$.
\end{lemma}
\begin{proof}
  \begin{align*}
    \Egu_{\Z} X
    &= \inf_{Z\in\Z} \sup (X - Z)
    \\
    &=
      \min\left(\infty, \inf_{Z\in\Z\setminus\Z'} \sup (X - Z)\right)
    \\
    &=
      \inf_{Z\in\Z\setminus\Z'} \sup (X - Z)
    \\
    &= \Egu_{\Z\setminus\Z'} X~.
  \end{align*}
\end{proof}

In particular, when $X$ takes on the value $\infty$, all gambles give an infinite replication cost.
\begin{proposition}\label{prop:egu-x-infinite}
  If $X(\omega) = \infty$ for any $\omega\in\Omega$, then $\Egu X = \infty$ regardless of $\Z$.
\end{proposition}
\begin{proof}
  As $\sup(X-Z) = \infty$ for all $Z:\Omega\to\extreals$,
  Lemma~\ref{lem:remove-bad-gambles} gives $\Egu_\Z X = \Egu_{\emptyset} X = \infty$.
\end{proof}

Similarly, Gambler can always ignore gambles taking on value $-\infty$.
This statement holds even when every gamble takes on $-\infty$; in this case $\Egu_\Z = \Egu_\emptyset = \infty$.
\begin{proposition}\label{prop:gambles-not-neg-infty}
  For any gamble space $(\Omega,\Z)$ we have 
  $\Egu_{\Z} = \Egu_{\Z^{(-\infty,\infty]}}$, where $\Z^{(-\infty,\infty]} = \{Z \in \Z \mid Z > -\infty\}$.
\end{proposition}
\begin{proof}
  For the first statement, if $Z(\omega) = -\infty$, then $\sup X-Z \geq X(\omega) - Z(\omega) = \infty$.
  We now apply Lemma~\ref{lem:remove-bad-gambles} to $\Z \setminus \Z^{(-\infty,\infty]}$.
\end{proof}

Finally, suppose $X(\omega) = -\infty$ for some $\omega$.
If Gambler chooses $Z > -\infty$, then such an $\omega$ would result in a $-\infty$ payoff to World, so World would never play such an $\omega$.
Hence we may restrict to the remaining outcomes without loss of generality.
Here we define $\Z|_{\Omega'} = \{Z|_{\Omega'} \mid Z\in\Z\}$ for any $\Omega' \subseteq \Omega$.
\begin{proposition}
  \label{prop:eliminate-X-neg-infinity}
  Let $(\Omega,\Z)$ be a nonempty gamble space such that $Z > -\infty$ for all $Z\in\Z$.
  Let $X$ be a variable.
  Then $\Egu_\Z X = \Egu_{\Z|_{\Omega'}} X|_{\Omega'}$, where $\Omega' = \{\omega\in\Omega \mid X(\omega) > -\infty\}$.
\end{proposition}
\begin{proof}
  If $X(\omega) = -\infty$, then $X(\omega) - Z(\omega) = -\infty$ for all $Z\in\Z$.
  Thus, for all $Z\in\Z$, we have $\sup_{\omega\in\Omega} X(\omega) - Z(\omega) = \max(-\infty, \sup_{\omega\in\Omega'} X(\omega) - Z(\omega)) = \sup_{\omega\in\Omega'} X(\omega) - Z(\omega)$.
  The result follows.
\end{proof}

\subsection{Operations and restrictions on gambles}
\label{sec:bounded-below}

We now state some basic facts about restrictions on $X$ and $\Z$ and how they relate.
We begin with some ways to restrict $\Z$ regardless of $X$.
First, any gamble which takes value $-\infty$ is useless.
Perhaps less obvious is that we may take all gambles to be real-valued without loss of generality.
For this latter claim, we introduce the \emph{downward closure} of $\Z$, which we take to be all real-valued gambles weakly dominated by $\Z$.
\begin{definition}[Downward closure]
  \label{def:dcl}
  Given $\Z \subseteq (\Omega \to \extreals)$, define the \emph{downward closure} $\dcl(\Z) \subseteq (\Omega \to \reals)$ of $\Z$ by $\dcl(\Z) := \{ Z':\Omega \to \reals \mid \exists Z\in\Z,\; Z'\leq Z \} = (\Z+ (-\infty,0]^\Omega)\cap\reals^\Omega $.
\end{definition}

We will also use the \emph{operator closure} $\tilde\Z = \{Z\in\reals^\Omega \mid \Egu Z \leq 0\}$, particularly in \S~\ref{sec:axioms-fa}.
(This operation is essentially the \emph{acceptance set} of a financial risk measure; see \S~\ref{sec:financial-risk-measures}.)
We consider a version for $Z\in\extreals^\Omega$ as well.
If all gambles in $\Z$ are real-valued, then $\Z \subseteq \dcl(\Z) \subseteq \tilde\Z$, but if any elements of $\Z$ take on infinite values the first inclusion will not hold.
The second inclusion can be strict: taking
$\Z = \{y\mapsto c \mid c < 0\}$, we have $\Egu 0 = 0$ but $0\notin \dcl(\Z)$.

\begin{proposition}
  \label{prop:restricting-gambles-all-X}
  \label{prop:dcl-wlog}
  For gamble space $(\Omega,\Z)$, let $\overline \Z = \{Z\in\extreals^\Omega \mid \Egu Z \leq 0\}$ and $\tilde \Z = \overline \Z \cap \reals^\Omega$.
  Then 
  \(\Egu_\Z = \Egu_{\dcl(\Z)} = \Egu_{\overline\Z} = \Egu_{\tilde\Z}\).
\end{proposition}
\begin{proof}
  For the first equality, take any $X : \Omega \to \extreals$.
  Let $Z'\in\dcl(\Z)$ and $\alpha \in \reals$ such that $Z' + \alpha \geq X$.
  By definition of $\dcl$, we have $Z\in\Z$ with $Z' \leq Z$, and thus $Z + \alpha \geq X$.
  Thus
  \begin{align}
    \Egu_{\Z} X
    &= \inf\{\alpha\in\reals \mid \exists Z\in\Z \text{ s.t. } Z + \alpha \geq X\} \nonumber
    \\
    &\leq \inf\{\alpha\in\reals \mid \exists Z'\in\dcl(\Z) \text{ s.t. } Z' + \alpha \geq X\}
      = \Egu_{\dcl(\Z)} X\,. \label{eq:dcl-inequality}
  \end{align}
  For the reverse inequality, if $\Egu_\Z X = \infty$, we are done by eq.~\eqref{eq:dcl-inequality}.
  Assume then that $\Egu_\Z X < \infty$, which from Proposition~\ref{prop:egu-x-infinite} implies $X < \infty$.
  Let $\alpha\in\reals$, $\alpha > \Egu_\Z X$, so that from the definition of $\Egu$ we have $Z\in\Z$ with $Z + \alpha \geq X$.
  From the first statement, we may assume $Z > -\infty$ without loss of generality.
  Define $Z':\Omega\to\reals$ by
  \begin{align}
    \label{eq:dcl-construction-z-prime}
    Z'(\omega) =
    \begin{cases}
      X(\omega) - \alpha & X(\omega) > -\infty
      \\
      \min(0,Z(\omega)) & X(\omega) = -\infty
    \end{cases}~.
  \end{align}
  Clearly $Z' + \alpha \geq X$.
  As $Z + \alpha \geq X$, in both cases we also have $Z' \leq Z$.
  As $X < \infty$ and $Z > -\infty$, we have $Z' : \Omega \to \reals$.
  Thus $Z' \in \dcl(\Z)$, giving $\Egu_{\dcl(\Z)} X \leq \alpha$.
  Taking $\alpha \to \Egu_\Z X$ gives the result.

  For the third equality, note that from Proposition~\ref{prop:egu-x-infinite}, $\Egu Z \leq 0 \implies Z < \infty$, so we must have $\overline \Z \subseteq (\reals\cup\{-\infty\})^\Omega$.
  Since $\overline\Z \cap (\infreals)^\Omega = \overline\Z \cap \reals^\Omega = \tilde\Z$, Proposition~\ref{prop:gambles-not-neg-infty} now gives $\Egu_{\overline\Z} = \Egu_{\tilde\Z}$.

  For the final equality, as we have already shown $\Egu_\Z = \Egu_{\dcl(\Z)}$, it suffices to show $\Egu_\Z \leq \Egu_{\tilde\Z} \leq \Egu_{\dcl(\Z)}$.
  The second inequality follows from the fact that $\dcl(\Z) \subseteq \tilde\Z$.
  For the first inequality, consider $X\in\extreals^\Omega$.
  If $\Egu_{\tilde\Z} X = \infty$, we are done; otherwise let $\alpha \in \reals$, $\alpha > \Egu_{\tilde\Z} X$.
  By definition of $\Egu$ we have $\tilde Z \in \tilde\Z$ such that $\tilde Z + \alpha \geq X$.
  By definition of $\tilde Z$, we have $\Egu_\Z \tilde Z \leq 0$.
  So fixing $\epsilon > 0$, we have some $Z\in\Z$ with $Z + \epsilon \geq \tilde Z$.
  Thus $Z + \alpha + \epsilon \geq \tilde Z + \alpha \geq X$, and we conclude $\alpha + \epsilon \geq \Egu_\Z X$.
  Taking $\epsilon \to 0$ and $\alpha \to \Egu_{\tilde\Z} X$ gives $\Egu_\Z X \leq \Egu_{\tilde\Z} X$.
\end{proof}
As with Proposition~\ref{prop:gambles-not-neg-infty}, the statement of Proposition~\ref{prop:dcl-wlog} holds even when $\dcl(\Z)$, $\overline\Z$, or $\tilde\Z$ are empty.

We now consider restrictions on $X$ which translate to gambles.
We first discuss boundedness from above and below, and then move on to more general statements.

\begin{definition}[Bounded below]
  We say a variable $X:\Omega \to \extreals$ is \emph{bounded below} if $\inf X > -\infty$, and \emph{bounded above} if $\sup X < \infty$.
  Given $\Z \subseteq (\Omega \to \extreals)$, define its \emph{bounded below subset} by $\bb(\Z) := \{Z\in\Z \mid \inf Z > -\infty\}$.
\end{definition}

Similar to why choosing $Z(\omega)=-\infty$ would be foolish no matter what $X$ is, when trying to replicate a bounded-below variable $X$, choosing a gamble $Z$ with $\inf Z = -\infty$ would also be foolish: World can simply make $Z$ unboundedly negative while $X$ remains bounded from below, again yielding an infinite penalty $\sup (X - Z) = \infty$.
From Lemma~\ref{lem:remove-bad-gambles}, when $X$ is bounded below, we may assume without loss of generality that gambles are also bounded below.
A similar statement holds when $X$ is bounded from above, where now one needs to invoke the downward closure first.
These statements are especially useful for bounded $X$ such as indicators $\ones_A$, and for nonnegative varibales (cf.~\S~\ref{sec:multiplicative}).

\begin{proposition}\label{prop:bounded-X-gambles}
  Let $(\Omega,\Z)$ be a gamble space, and $X:\Omega\to\extreals$ a variable.
  \begin{enumerate}
  \item If $X$ is bounded below, then $\Egu_{\Z} X = \Egu_{\bb(\Z)} X = \Egu_{\bb(\dcl(\Z))} X$.
  \item If $X$ is bounded above, then $\Egu_{\Z} X = \Egu_{\ba(\dcl(\Z))} X$ where $\ba(\Z) := \{Z\in\Z \mid \sup Z < \infty\}$.
  \end{enumerate}
  In particular, when $X$ is bounded, then $\Egu_\Z X = \Egu_{\Z'} X$ where $\Z'$ are the bounded elements of $\dcl(\Z)$.
\end{proposition}
\begin{proof}
  For the first statement, let $\inf X = b > -\infty$.
  For any variable $Z$ with $\inf Z = -\infty$, we have $\sup X-Z \geq \sup b - Z = b - \inf Z = b - (-\infty) = \infty$.
  Lemma~\ref{lem:remove-bad-gambles} now gives $\Egu_{\Z} X = \Egu_{\bb(\Z)} X$
  and $\Egu_{\dcl(\Z)} X = \Egu_{\bb(\dcl(\Z))} X$.
  Proposition~\ref{prop:dcl-wlog} gives the remaining equality.

  For the second statement, let $\sup X = b < \infty$.
  Let $\alpha > \Egu_\Z X = \Egu_{\dcl(\Z)}$, $\alpha\in\reals$.
  For $Z\in\dcl(\Z)$ with $Z + \alpha > X$,
  let $Z' = \min(Z,b-\alpha) \in \ba(\dcl(\Z))$.
  By construction $Z' + \alpha > X$ as well, giving $\alpha > \Egu_{\ba(\dcl(\Z))}$.
  Thus $\Egu_{\ba(\dcl(\Z))} X \leq \Egu_{\Z} X$.
  The reverse inequality follows from Proposition~\ref{prop:dcl-wlog} and the fact that $\ba(\dcl(\Z)) \subseteq \dcl(\Z)$.
\end{proof}

\begin{remark}\label{rem:bounded-above-below}
  In general, one cannot restrict to bounded above/below gambles if $X$ is not bounded above/below.
  For any $X$ with $\sup X = \infty$, for bounded above $Z$ we have $\sup X-Z = \infty$, giving $\Egu_{\ba(\dcl(\Z))} X = \infty$.
  Yet $\dcl(\Z)$ could contain $Z$ that replicates $X$ within a finite constant.
  For example, consider gamble space $(\Omega,\Z)$ where $\Omega = \reals$ and $\Z = \{\omega\mapsto \beta \omega \mid \beta\in\reals\}$.
  For the identity $X:\omega\mapsto\omega$, we have $\Egu_\Z X = 0$ via the choice $Z = X \in \Z$, but $\Egu_{\ba(\dcl(\Z))} X = \infty$ by the argument above.
  The same example shows that when $X$ is not bounded below, we can have $\Egu_\Z X < \Egu_{\bb(\Z)}$: here $\bb(\Z) = \{0\}$ and thus $\Egu_{\bb(\Z)} X = \sup (X - 0) = \infty$.

  The reader may note that we did not state $\Egu_\Z X = \Egu_{\ba(\Z)} X$ when $X$ is bounded above: this statement fails in general.
  Consider the same $\Z$ but now $X(\omega) = \max(\omega,1)$.
  Clearly $\Egu X = 0$ via the same choice $Z:\omega\mapsto\omega$.
  But $\ba(\Z) = \{0\}$ as with $\bb$, giving $\Egu_{\ba(\Z)} X = \sup (X-0) = 1$.
\end{remark}

As we saw in Proposition~\ref{prop:bounded-X-gambles} and Remark~\ref{rem:bounded-above-below}, it seems that when $X$ satisfies a certain property (like bounded below or bounded above), we may restrict the gambles $\Z$ to only those sharing that property, though in some cases we must first replace $\Z$ with $\dcl(\Z)$.
The following proposition shows that indeed this observation holds quite generally as long as the property in question is closed under translation.

\begin{proposition}\label{prop:egu-dcl-preserves-properties}
  Let $(\Omega,\Z)$ be a gamble space and $\G \subseteq (\infreals)^\Omega$ any set closed under translation, i.e., $X\in\G,\alpha\in\reals \implies X+\alpha\in\G$.
  Then $\Egu_\Z X = \Egu_{\dcl(\Z)\cap\G} X$ for any $X\in\G$.
\end{proposition}
\begin{proof}
  Let $\Z' = \dcl(\Z)\cap\G$.
  As $\Z' \subseteq \dcl(\Z)$, we have $\Egu_\Z = \Egu_{\dcl(\Z)} \leq \Egu_{\Z'}$, with the equality coming from Proposition~\ref{prop:dcl-wlog}.
  For the reverse inequality, if $\Egu_{\Z'} X = \infty$ we are done; otherwise $X<\infty$ from Proposition~\ref{prop:egu-x-infinite} and thus $X\in\G\cap\reals^\Omega$.
  Let $\alpha \in \reals$ with $\alpha > \Egu_\Z X$.
  By definition of $\Egu$, there exists $Z\in\Z$ with $Z + \alpha \geq X$.
  Thus $X - \alpha \leq Z$.
  As $X - \alpha \in \reals^\Omega$, we have $X - \alpha \in \dcl(\Z)$.
  By assumption, we also have $X - \alpha \in \G$.
  We conclude $\Egu_{\Z'} X \leq \alpha$, and taking $\alpha \to \Egu_\Z X$ gives $\Egu_{\Z'} X \leq \Egu_\Z X$.
\end{proof}

To see the power of this observation, suppose $\Omega$ is a topological space.
Let $C(\Omega)$ be the set of continuous functions $\Omega\to\reals$,
and $C_b(\Omega) \subseteq C(\Omega)$ the set of bounded continuous functions.

\begin{corollary}\label{cor:continuous-gambles-wlog}
  Let $(\Omega,\Z)$ be a gamble space such that $\Omega$ is a topological space.
  If $X:\Omega\to\reals$ is continuous, we have $\Egu_\Z X = \Egu_{\dcl(\Z) \cap C(\Omega)} X$.
  If $X$ is additionally bounded, we have $\Egu_\Z X = \Egu_{\dcl(\Z) \cap C_b(\Omega)} X$.
\end{corollary}
In \S~\ref{sec:measurability} we will also apply this observation to measurable variables.

\subsection{Sequential gambles}
\label{sec:sequential-gambles}

To capture sequential settings like the law of large numbers in \S~\ref{sec:introduction}, we will specify a triple $(\Y,\{\hat\Z_t\}_t,T)$: $\Y$ is the set of per-round outcomes, $\hat\Z_t \subseteq \Y \to \extreals$ the set of gambles on round $t$, and $T \in \N \cup \{\infty\}$ the \emph{time horizon}, i.e., how many rounds will be played.
(See \protocol~\ref{alg:simple-repeated}.)
The setting in \S~\ref{sec:motivating-example} corresponds to the choice $\Y = [-1,1]$.

\begin{algorithm}[t]
  \caption{Simple sequential gambles}
  \For{$t = 1, 2, \ldots, T$}{
    Gambler chooses $\hat Z_t \in \hat \Z_t$ \\
    World chooses $y_t \in \Y$ \\
    Gambler receives $\hat Z_t(y_t)$ \\
  }
  \label{alg:simple-repeated}
\end{algorithm}

\begin{example}[Bounded sequential gambles]
  \label{ex:bounded-sequential}
  The law of large numbers in \S~\ref{sec:introduction} was set in the simple sequential gamble space $(\Y,\hat\Z,\infty)$ with $\Y = [-1,1]$ and $\hat\Z = \{y\mapsto \beta y \mid \beta \in \reals\}$.
  See Definition~\ref{def:simple-sequential-gambles}.
\end{example}

A natural extension of Example~\ref{ex:bounded-sequential} would allow the means $\mu_t$ to vary in each time step, so that $\hat\Z_t = \{y\mapsto \beta (y - \mu_t) \mid \beta \in \reals\}$.
Rather than having $\{\mu_t\}_t$ be a fixed sequence, we might want it to be merely predictable in the usual sense; to this end,
\citet{shafer2019game} often introduce a third player into the game, Forecaster, whose role is to select $\mu_t$ or other parameters at the beginning of the round.
Yet from the point of view of Gambler, worst-case guarantees correspond to World and Forecaster conspiring against her, so without loss of generality they are the same player.
In other words, even when including other players, game-theoretic statements still boil down to a 2-player zero-sum game.
We may therefore merge Forecaster and World into one player without loss of generality.
\footnote{In other settings, such as defensive forecasting, Forecaster plays the role of Gambler, and one can merge the other players into World.}
In doing so, we must take care with the time indices, since now World reveals the $\mu_{t+1}$ at the end of round $t$, along with $y_t$.
One could more generally allow some context $w_{t+1}$ to be revealed, as shown in \protocol~\ref{alg:adding-context}.

We can capture such ``context'' by writing $\Y = (\hat\Y,\W)$ where $y\in\hat\Y$ is thought of as the actual outcome, and $w\in\W$ is the context for the next round.
For example, we might take $\hat\Y=\W=[-1,1]$, and write $y_t = (\hat y_t,w_{t+1}) = (\hat y_t,\mu_{t+1})$ to define the gambles $\hat\Z_t$ above.
But now $\hat\Z_t$ depends on $\mu_t = w_t$, and thus on $y_{t-1}$, the previous outcome.
To capture such context, therefore, we must extend our definitions to allow for this dependence; for the sake of generality, we may allow $\hat\Z_t = \hat\Z^{(s)}$ to depend on the entire situation $s = y_{1..t-1}$.
This extension leads to a \emph{sequential gamble space}.
See \protocol~\ref{alg:equivalent-protocol} for the special case where only additional context is revealed, giving an equivalent way to express \protocol~\ref{alg:adding-context}.

\begin{algorithm}[!h]
  \caption{Adding context to sequential gambles}
  \For{$t = 1, 2, \ldots, T$}{
    Forecaster chooses parameters $w_t \in \reals^d$ \\
    Gambler chooses $Z_t \in \hat\Z_t^{(w_t)}$ \\
    World chooses $y_t \in \Y$ \\
    Gambler receives $Z_t(y_t)$ \\
  }
  \label{alg:adding-context}
\end{algorithm}
\begin{algorithm}[!h]
  \caption{Adding context, as a sequential gamble space}
  $s \gets \emptystring$, $w_1\in\reals$ given \\
  \For{$t = 1, 2, \ldots, T$}{
    Gambler chooses $Z_t \in \hat\Z^{(s)} := \hat\Z_t^{(w_t)}$ \\
    World chooses $y_t = (\hat y_t,w_{t+1}) \in \Y := \hat \Y \times \reals^d$ \\
    Gambler receives $Z_t(\hat y_t)$ \\
    $s\gets s \oplus y_t$
  }
  \label{alg:equivalent-protocol}
\end{algorithm}

Let us now define sequential gamble spaces formally.
Let $\Y$ be a set of per-round outcomes.
\footnote{It is also natural to consider settings where the set of available outcomes in each round itself depends on the history.  For the most part, this generalization is not needed, but it can be a more elegant way to capture settings where the outcomes can vary, or even statements like the tower rule, where it would be natural to focus on two-round gamble spaces; see \S~\ref{sec:general-seq-gambles} for a more general setting. }
Let time horizon $T\in\N\cup\{\infty\}$ be given.
Let $\Omega = \Y^T$ be the set of \emph{outcomes}.
Given $y\in\Y^t$, we write $y_{1..i} := (y_1,\ldots,y_i) \in \Y^i$ to be the first $i$ elements of $y$, where $y_{1..0} := \emptystring$ is the empty sequence and $\Y^0 = \{\emptystring\}$.
For $t\in\N\cup\{\infty\}$, let $\Y^{<t} := \bigcup_{0\leq i<t} \Y^i$.
Define the set of \emph{situations} to be
$\Y^{<T}$.
The outcomes $\Omega = \Y^T$ are not situations, but can be thought of as ``terminal'' situations.
Given $s\in\Y^t$ and $y\in\Y$, the sequence $s\oplus y\in\Y^{t+1}$ is the concatenation of $s$ and $y$.

For each situation $s\in\Y^{<T}$ we are given a set of \emph{per-round gambles} $\hat \Z^{(s)} \subseteq (\Y\to\extreals)$ available in that situation.
We let $\Psi$ be the set of all gambling strategies $\psi : \Y^{<T} \to (\Y\to\extreals)$ such that for $\psi(s) \in \hat\Z^{(s)}$ for each $s\in\Y^{<T}$.
We can equivalently represent the available per-round gambles via a single set $\hat\Z \subseteq (\Y^{<T}\to\extreals)$, by defining
$\hat \Z^{(s)} := \{\hat Z(s\oplus \cdot):\Y\to\extreals \mid \hat Z\in\hat\Z\}$.
Because of this equivalent representation, we will often write $\hat\Z \subseteq (\Y^{<T}\to\extreals)$ as shorthand for the indexed set $\{\hat \Z^{(s)}\}_s$.

For example, in simple sequential gambles (\protocol~\ref{alg:simple-repeated}), we have $\hat\Z^{(s)} = \hat\Z_{|s|+1}$ for all $s\in\Y^{<T}$.

The cumulative gamble $Z^\psi_t$ for strategy $\psi\in\Psi$ up to time $t\in\N$ is simply the sum of the resulting per-round payoffs,
\begin{align}\label{eq:cumulative-gamble-finite}
  Z^\psi_t(y_{1..t}) = \sum_{i=1}^t \psi(y_{1..i-1})(y_i) = \sum_{i=1}^t \hat Z_i(y_{1..i})~.
\end{align}
For $t=\infty$, this sum may fail to converge; we take the limit infimum as a pessimistic evaluation of Gambler's profit:
\begin{align}\label{eq:cumulative-gamble-infinite}
  Z^\psi_\infty(y) = \liminf_{t\to\infty} Z^\psi_t(y_{1..t})~.
\end{align}
Finally, for $t \in \N \cup \{\infty\}$, we let $\Z_{t} = \{Z^\psi_t \mid \psi\in\Psi\}$ be the set of partial cumulative gambles up to time $t$.

\begin{definition}[Sequential gamble space]
  \label{def:sequential-gamble-space}
  Let per-round outcomes $\Y$ and time horizon $T\in\N\cup\{\infty\}$ be given.
  Let per-round gambles $\hat\Z \subseteq (\Y^{<T}\to\extreals)$ be given, from which we can define $\{\hat\Z^{(s)}\}_{s\in\Y^{<T}}$ as above.
  Then we define the \emph{sequential gamble space} $(\Y,\hat\Z,T)$ to be the gamble space $(\Y^T,\Z_T)$, as defined following eq.~\eqref{eq:cumulative-gamble-infinite}.
\end{definition}

Some special cases of sequential gamble spaces will arise frequently.
We say $(\Y,\hat\Z,T)$ is \emph{real-valued} if $\hat\Z \subseteq (\Y^{<T}\to\reals)$.
We say $(\Y,\hat\Z,T)$ is \emph{sequentially positive-linear} or \emph{sequentially normalized} if the $\hat\Z^{(s)}$ are respectively positive linear or normalized for all $s\in\Y^{<T}$.
For brevity, we will often omit the extra ``sequential'', e.g., we will speak of sequentially normalized gamble spaces $(\Y,\hat\Z,T)$.

Sequentially positive-linear gamble spaces feature prominently in the sequel.
Note however that when $T=\infty$, the global gamble space $(\Y^T,\Z_T)$ is typically \emph{not} positive linear in the sense of Definition~\ref{def:gamble-space-conditions}, because of the limit infimum in eq.~\eqref{eq:cumulative-gamble-infinite}.
Nonetheless the gambles do have considerable structure; see Lemma~\ref{lem:convexity-dcl} and \S~\ref{sec:axioms-fa}.

We will often find it convenient in examples to work with the special case of ``simple'' sequential gamble spaces, where $\hat\Z^{(s)}$ depends only on $|s|$.

\begin{definition}[Simple sequential gamble space]
  \label{def:simple-sequential-gambles}
  Let $\Y$ be a set of per-round outcomes, $T\in\N\cup\{\infty\}$ a time horizon, and $\{\hat\Z_t\subseteq(\Y\to\extreals)\}_{t=1}^T$ sets of per-round gambles.
  The \emph{simple sequential gamble space} $(\Y,\{\hat\Z_t\}_t,T)$ is the sequential gamble space $(\Y,\hat\Z,T)$ given by
  $\hat\Z^{(s)} = \hat\Z_{|s|+1}$ for all $s\in\Y^{<T}$.
  When we have some $\hat\Z\subseteq(\Y\to\extreals)$ such that $\hat\Z_t = \hat\Z$ for all $t\leq T$, we call it a \emph{simple repeated gamble space} and write it more simply as $(\Y,\hat\Z,T)$.
  
\end{definition}

The following conventions will simplify the exposition surrounding sequential gamble spaces.
Several results may be stated for $T=\infty$ only, using the fact that for $T < \infty$ one can always define $\hat\Z^{(s)} = \{0\}$ for all $|s|\geq T$.
Similarly, when $0\in\hat\Z^{(s)}$ for all $s\in\Y^{<T}$, we will often consider $\Z_t$ to be a subset of $\Z_T$ when $t\leq T$.
Finally, we will sometimes use the functions $Y_t:\Y^T\to\Y$ given by $Y_t(y) = y_t$ to write the per-round outcomes.

It may not be completely obvious that Definition~\ref{def:simple-sequential-gambles} and \protocol~\ref{alg:simple-repeated} align; the following establishes the equivalence of these two for finite time horizons.
\begin{proposition}\label{prop:simple-sequential-one-shot}
  Let $(\Y,\hat\Z,T)$ be a real-valued simple repeated gamble space for $T\in\N$.
  Then for any variable $X:\Y^T\to\extreals$, we have
  \begin{align}
    \Egu X
    &= \inf_{Z_1\in\hat\Z} \sup_{y_1\in\Y} \inf_{Z_2\in\hat\Z} \sup_{y_2\in\Y} \cdots \inf_{Z_T\in\hat\Z} \sup_{y_T\in\Y} X(y_{1..T}) - \sum_{t=1}^T Z_t(y_t)~.
  \end{align}
\end{proposition}
The proof, which we give in \S~\ref{sec:tower-properties-sequential-minimax}, follows by the observation that Gambler deciding a gamble in each round, based on the outcomes so far, is equivalent to Gambler specifying a full contingency plan up front.
That is, we can encode all of Gambler's choices on the right-hand side, each contingent on the choices of World thus far, in a strategy $\psi$ for the left-hand side.
This observation is essentially a tower property for upper expectations (Proposition~\ref{prop:tower-property-egu}).

\begin{remark}\label{remark:defining-Z-infty}
  \citet[eq.~(7.19)]{shafer2019game} essentially define $\Z_T$ to be $\bb(\Z_T)$.
  \footnote{As justification for this restriction, \citet[Exercise 8.3]{shafer2019game} shows that $\Egl \leq \Egu$ can fail if one does not restrict to $\bb(\Z_T)$ in their setting, even when the sequential gamble space is arbitrage-free and positive linear.  This statement would appear to contradict Proposition~\ref{prop:basic-facts-assumptions}.  The discrepancy is the weak inequality in their replication definition of $\Egu$, while we use a strict inequality in eq.~\eqref{eq:prop-upper-ex-rep-cost}.}
  Indeed, from Proposition~\ref{prop:bounded-X-gambles}, we may do so without loss of generality when reasoning about bounded-below variables $X$.
  But as we saw in Remark~\ref{rem:bounded-above-below}, for unboundedly negative $X$, defining $\Z_T$ to be $\bb(\Z_T)$ can lead to results which are both counter-intuitive and do not match their measure-theoretic counterparts.

  To cast that example as a sequential gamble space, take the simple repeated gamble space $(\Y,\hat\Z,T)$ where $\Y = \reals$, $\hat\Z = \{y\mapsto \beta y \mid \beta\in\reals\}$, and any $T \geq 1$.
  Let $X(y) = y_1$, which has $\inf X = -\infty$.
  Then as before, $\Egu_{\bb(\Z_T)} X = \sup X = \infty$, since the only bounded-below gamble is $Z=0$, i.e., $\bb(\Z_T) = \{0\}$.
  Yet the analogous measure-theoretic statement would be $\E Y_1 = 0$.
  Taking no restriction on the gambles also yields a zero game-theoretic upper expectation, $\Egu_{\Z_T} X = 0$, by taking $\beta = 1$.

  That said, as we will see in \S~\ref{sec:prices-probability}, there are good reasons to restrict $\Z_T$ in some way.
  For one, without restriction the gambles can fail to have ``consistent'' probability measures (Example~\ref{ex:seq-consistent-but-no-consistent}).
  While restricting all the way to $\bb(\Z_T)$ may be too far, as illustrated above, adding back all finite-time gambling strategies is a useful compromise.
  This is the approach we take in Proposition~\ref{prop:sequentially-consistent-is-consistent}, Corollary~\ref{cor:sequential-price-inequalities}, and Corollary~\ref{cor:ep0star-equals-ep}.
\end{remark}

\begin{remark}[The no-bankruptcy condition]
  \label{rem:no-bankruptcy}
  In the literature on game-theoretic probability, restrictions on gambles are often stated via a no-bankruptcy condition.
  In \S~\ref{sec:motivating-example} and Proposition~\ref{prop:prob-one-iff}, we instead imposed a lower bound on $Z$, namely $Z\geq -1$.
  In sequential settings, this condition is on the limit infimum of $Z^\psi_t$ as in eq.~\eqref{eq:cumulative-gamble-infinite}.
  In general, if $\hat\Z$ is scalable and arbitrage-free, the two conditions are equivalent.
  But if even one round allows arbitrage, the liminf definition can be strictly weaker.
  As a simple example, one can take any sequential gamble space and add two rounds to the beginning: fixing some $c>0$, in the first the only gamble is $-c$, and in the second the only gamble is $c$.
  Then one has not changed the liminf, as the first two rounds exactly cancel out, but now Gambler's capital becomes arbitrarily negative after the first round.
  For this reason, the liminf definition is more robust.
\end{remark}

\subsection{Conditional expectations, supermartingales, tower property}
\label{sec:cond-expect-supermtg}

In sequential settings, we have two types of upper expectations to study: the ``global'' notion on the gamble space $(\Y^T,\Z_T)$, and the per-round upper expectation on $(\Y,\hat\Z^{(s)})$.
Conditional upper expectations give rise to the latter.

\begin{definition}[Conditional game-theoretic upper expectation]
  \label{def:conditional-egu}
  
  Let $(\Y,\hat\Z,T)$ be a sequential gamble space.
  Let $t \in \{1,\ldots,T-1\}$ and $s\in\Y^t$.
  Define $\hat\Z|_s \subseteq (\Y^{<T-t} \to \extreals)$ by $\hat\Z|_s = \{\hat Z(s \oplus \, \cdot \,) : \hat Z \in \hat\Z\}$.
  For any $X:\Y^T\to\extreals$, we define $\Egu[X\mid s] := \Egu[X(s \oplus \, \cdot \,)]$ with respect to the sequential gamble space $(\Y,\hat\Z|_s,T-t)$.
  
\end{definition}

For example, suppose $(\Y,\hat\Z,T)$ is a simple repeated gamble space for $T\in\N\cup\{\infty\}$.
Then for variable $X:\Y^T\to\reals$, and any $s\in\Y^t$, $t<T$, the conditional game-theoretic upper expectation $\Egu[ X \mid s ]$ is $\Egu[X(s\oplus \,\cdot\, )]$ with respect to the simple repeated gamble space $(\Y^{T-t},\hat\Z,T-1)$.

We will primarily use the conditional upper expectation on a single following round, i.e.\ the gamble space $(\Y,\hat\Z^{(s)})$.

In this case, if $X:\Y^t\to\extreals$ and $s\in\Y^{t-1}$ we have
$\hat\Z|_s = \hat\Z^{(s)}$ by definition, and thus
\begin{align}
  \label{eq:conditional-expectation}
  \Egu[ X \mid s] &= \Egu_{\hat\Z^{(s)}} X(s \oplus \, \cdot \,) = \inf_{Z \in \hat\Z^{(s)}} \sup_{y\in\Y} X(s\oplus y) - Z(y)~.
\end{align}
With this next-round conditional expectation, we can define supermartingales.

\begin{definition}[Game-theoretic supermartingale]
  \label{def:supermartingale}
  Let $(\Y,\hat\Z,T)$ be a sequential gamble space.
  A sequence $\{X_t:\Y^t\to\extreals\}_{t\leq T}$ is a \emph{game-theoretic supermartingale} if for all $t< T$ and situations $s\in\Y^t$, we have 
  \begin{align}
    \label{eq:supermartingale-def}    
    \Egu[ X_{t+1} \mid s] \leq X_t(s)~,
  \end{align}
  and a \emph{game-theoretic martingale} if $\Eg[ X_{t+1} \mid s] = X_t(s)$.
\end{definition}

One can think of supermartingales as the wealth process resulting from strategies for Gambler, though perhaps some money is discarded in the process, typically for convenience or tractability.
In particular, we have the following.

\begin{proposition}\label{prop:gambling-strategy-supermartingale}
  Let $(\Y,\hat\Z,T)$ be a sequential gamble space and $Z^\psi\in\Z_T$.
  If $\psi(s)(y)\in\reals$ for all $s\in\Y^{<T},y\in\Y$, then $\{Z^\psi_t\}_t$ is a game-theoretic supermartingale.
\end{proposition}
\begin{proof}
  Let $t< T$ and $s \in \Y^t$ be given.
  \begin{align*}
    \Egu[ Z^\psi_{t+1} \mid s]
    &= \inf_{\hat Z \in \hat\Z^{(s)}} \sup_{y\in\Y} Z^\psi_{t+1}(s\oplus y) - \hat Z(y)
    \\
    &= \inf_{\hat Z \in \hat\Z^{(s)}} \sup_{y\in\Y} \left(Z^\psi_t(s) + \psi(s)(y)\right) - \hat Z(y)
    \\
    &\leq \sup_{y\in\Y} \left(Z^\psi_t(s) + \psi(s)(y)\right) - \psi(s)(y)
    \\
    &= Z^\psi_t(s)~.
  \end{align*}
\end{proof}

For real-valued sequential gamble spaces, the Doob/Levy supermartingale for $X:\Y^T\to\extreals$ is the sequence $\{X_t\}_t$ given by $X_t(s) = \Egu[X \mid s]$.
The supermartingale property follows from the tower property, which we now show.

\begin{proposition}\label{prop:tower-property-egu}
  Let $(\Y,\hat\Z,T)$ be a real-valued sequential gamble space.
  Let $X:\Y^T\to\extreals$, and consider $\Y^t\ni s\mapsto\Egu[X\mid s]$ to be a variable on $\Y^t$ for $t<T$.
  Then $\Egu[\Egu[X\mid\cdot\,]] = \Egu X$.
\end{proposition}
\begin{proof}
  Recall that $\Psi = \{\psi:\Y^{<T}\to \extreals^\Y :
  \forall s\in\Y^{<T}\; \psi(s)\in\hat\Z^{(s)}\}$ is the set of valid gambling strategies.
  Let $\Psi_t$ be the same set but with domain $\Y^{< t}$, and let $\Psi_{t+1,T}^{(s)}$ denote the set of $\psi:\Y{<T-t}\to \extreals^\Y$ such that $\forall s'\in\Y^{<T-t}\; \psi(s')\in\hat\Z^{(s\oplus s')}\}$.
  Observe that a choice $\psi\in\Psi_t$ and set of choices $\{\psi^{(s)} \in \Psi_{t+1,T}^{(s)}\}_{s\in\Y^t}$ gives rise to a choice $\psi' \in \Psi$ given by $\psi'(s) = \psi(s)$ if $|s|\leq t$ and $\psi'(s) = \psi^{(s_{1..t})}(s_{t+1..|s|})$ for $|s|\geq t+1$.
  Conversely, any $\psi'\in\Psi$ decomposes into some $\psi\in\Psi_t$ and $\{\psi^{(s)} \in \Psi_{t+1,T}^{(s)}\}_{s\in\Y^t}$.
  \begin{align*}
    \Egu[\Egu[ X \mid Y_{1..t}]]
    &=
      \inf_{Z_t\in\Z_t} \sup_{s\in\Y^t} \Egu[ X \mid s] - Z_t(s)
    \\
    &=
      \inf_{\psi\in\Psi_t} \sup_{s\in\Y^t} \left(\inf_{\psi^{(s)} \in \Psi_{t+1,T}^{(s)}} \sup_{s'\in\Y^{T-t}} X(s\oplus s') - Z^{\psi^{(s)}}(s') \right) - Z^{\psi}(s)
    \\
    &\stackrel{(i)}{=}
      \inf_{\substack{\psi\in\Psi_t,\\
    \{\psi^{(s)} \in \Psi_{t+1,T}^{(s)}\}_{s\in\Y^t}}}
    \sup_{s\in\Y^t} \left(\sup_{s'\in\Y^{T-t}} X(s\oplus s') - Z^{\psi^{(s)}}(s') \right) - Z^{\psi}(s)
    \\
    &\stackrel{(ii)}{=}
      \inf_{\substack{\psi\in\Psi_t,\\
    \{\psi^{(s)} \in \Psi_{t+1,T}^{(s)}\}_{s\in\Y^t}}}
    \sup_{\substack{s\in\Y^t\\s'\in\Y^{T-t}}} X(s\oplus s') - Z^{\psi^{(s)}}(s') - Z^{\psi}(s)
    \\
    &=
      \inf_{\psi\in\Psi}
    \sup_{y\in\Y^T}
    X(y) - Z^\psi(y)
    \\
    &=
      \Egu X~.
  \end{align*}
  Equality (i) follows from the observation that the choices in the inner infimum over $\psi^{(s)}$ can equivalently be determined ahead of time.
  Equality (ii) is more subtle: due to our arithmetic conventions around $\infty$, the term $Z^\psi(s)$ can be combined inside the suprema as it is real-valued.
  (See Remark~\ref{rem:real-valued-tower-rule}.)
\end{proof}

\begin{remark}\label{rem:real-valued-tower-rule}
  The restriction that $\hat\Z$ be real-valued is necessary in Proposition~\ref{prop:tower-property-egu}.
  Consider the following simple example.
  Let $\Y=\reals$, $T=2$, and define the per-round gamble spaces by $\hat\Z^{(\emptystring)} = \{Z_1:y\mapsto \infty y\}$ and $\hat\Z^{(y_1)} = \{0\}$ for all $y_1\in\reals$.
  In other words, $Z_1(0) = 0$ and $Z_1(y_1) = \infty$ otherwise.
  Note that $\Z_T = \Z_2 = \{(y_1,y_2)\mapsto \infty y_1\}$, i.e., there is only a single choice available, and it is equal to $Z_1 + 0 = Z_1$.

  Now consider $X(y_1,y_2) = y_1y_2$.
  We have $\Egu X = 0$: clearly $\Egu \leq 0$ as $X$ is finite and $Z_1 \geq X$, and this inequality is tight at $y_1 = 0$.
  But looking at the conditional upper expectations, we see $\Egu[X \mid y_1] = 0$ when $y_1=0$, and $\Egu[X \mid y_1] = \infty$ when $y_1 \neq 0$.
  Thus $\Egu[\Egu[X \mid y_1]] = \infty$ as well.

  Beyond the tower property failing, we can also see that, despite the fact that $\Egu_\Z = \Egu_{\dcl(\Z)}$ from Proposition~\ref{prop:dcl-wlog}, one cannot replace $\hat\Z^{(s)}$ by $\dcl(\hat\Z^{(s)})$ and preserve the global upper expectation.
  Indeed, if we took any $Z_1' \in \reals^\Y$ with $Z'_1 \leq Z_1$, then say for $y_1 = 1$ we would have $Z'_1(1) \in \reals$, and thus $X(y_1,y_2) = y_1 y_2 > Z'_1$ for $y_2 > Z'_1(1)$.
  For any finite initial capital $\alpha\in\reals$, the same would hold for $y_2 > Z'_1(1) + \alpha$.
  Thus we would have $\Egu_{\Z_2'} X = \infty$ for this per-round-dcl version of $\Z_2$.
\end{remark}

\begin{corollary}\label{cor:finite-stopping}
  Let $(\Y,\hat\Z,T)$ be a real-valued sequential gamble space and $\{X_t\}_t$ a  game-theoretic supermartingale.
  Then for all $t\in\N, t \leq T$, we have $\Egu X_t \leq X_0$.
\end{corollary}
\begin{proof}
  The statement is trivial for $t=0$.
  For $t>0$, we have
  $\Egu[X_t \mid Y_{t-1}] \leq X_{t-1}$ by the supermartingale condition.
  Proposition~\ref{prop:tower-property-egu} and monotonicity give
  $\Egu X_t = \Egu[\Egu[X_t \mid Y_{t-1}]] \leq \Egu[X_{t-1}] \leq X_0$, the last inequality by induction.
\end{proof}

As noted above, conditional expectations can be thought of separating outcomes and gambles into what has been revealed thus far, and what remains.
The tower property, Proposition~\ref{prop:tower-property-egu}, then relates this separation back to the full gamble space: the two views are equivalent.
Using these concepts, Proposition~\ref{prop:simple-sequential-one-shot} follows from the observation that
\begin{align*}
  \Egu X
  &= \Egu[\Egu[\cdots \Egu[\Egu[X\mid Y_{1..T-1}]\mid Y_{1..T-2}] \cdots \mid Y_{1..2} ] \mid Y_1]
    \numberthis \label{eq:iterated-tower-property}
  \\
  &= \inf_{Z_1\in\hat\Z} \sup_{y_1\in\Y} \inf_{Z_2\in\hat\Z} \sup_{y_2\in\Y} \cdots \inf_{Z_T\in\hat\Z} \sup_{y_T\in\Y} X(y_{1..T}) - \sum_{t=1}^T Z_t(y_t)~.
\end{align*}
Eq.~\eqref{eq:iterated-tower-property} for finite $T$ follows by induction with the definition of game-theoretic upper expectation (Definition~\ref{def:egu}) as the base case, and Proposition~\ref{prop:tower-property-egu} giving $\Egu X = \Egu[\Egu[X \mid Y_{1..T-1}]]$.

We can also ``invert'' the definition of supermartingales to write $\Egu X$ and $\Pgu A$ in a form similar to Ville's Theorem: the minimum capital needed for a supermartingale to replicate $X$ or $\ones_A$, respectively.
Indeed, \citet{shafer2001probability,shafer2019game} essentially define game-theoretic expectations and probabilities in this way.
These statements hold for arbitrary gamble spaces.
When gambles are scalable, we can state the latter as the smallest $\alpha$ for which we can start from \$1 and making arbitrarily close to $1/\alpha$ on $A$.

\begin{theorem}\label{thm:ville-expectation}
  Let $(\Y,\hat\Z,T)$ be a real-valued sequential gamble space.
  Then for all $X:\Y^T\to\extreals$ we have
  \begin{align}
    \label{eq:ville-supermartingale-egu}
    \Egu X = \inf\{ X_0 \mid \{X_t \in \reals\}_t \text{ game-theoretic supermartingale}, X_T > X \}~,
  \end{align}
  where $X_\infty := \liminf_{t\to\infty} X_t$.
\end{theorem}
\begin{proof}
  
  The $\leq$ direction follows from Proposition~\ref{prop:gambling-strategy-supermartingale} and the assumption that the gamble space is real-valued.

  For the reverse inequality, for all $t<T$ and $s\in\Y^{t-1}$,
  the supermartingale condition gives $\Egu[ X_{t+1} \mid s] \leq X_t(s)$ on gamble space $(\Y,\hat\Z^{(s)})$.
  From the definition of $\Egu$, 
  we have some $Z^{(s)} \in \hat\Z^{(s)}$ such that $Z^{(s)}(\cdot) + X_t(s) + \epsilon 2^{-t} \geq X_{t+1}(s\oplus \,\cdot\,)$.
  As all quantities in this expression are real-valued by assumption, we have $X_t(s\oplus \,\cdot\,) - Z^{(s)}(\cdot) \leq X_t(s) + \epsilon 2^{-t}$.
  Letting $\psi : s \mapsto Z^{(s)}$, we now have for all $y\in\Y^T$,
  \begin{align*}
    Z^\psi(y)
    &= \liminf_{n\to T} \sum_{t=1}^n \psi(y_{1..t})(y_{t})
    \\
    &\geq \liminf_{n\to T} \sum_{t=1}^n X_t(y_{1..t}) - X_{t-1}(y_{1..t-1}) - \epsilon 2^{-t}
    \\
    &\geq \liminf_{n\to T} X_n(y_{1..n}) - X_0 - \epsilon
    \\
    &= X_T(y) - X_0 - \epsilon~.
  \end{align*}
  So $Z^\psi + X_0 + \epsilon \geq X_T > X$, whence $\Egu X \leq X_0 + \epsilon$.
  Taking $\epsilon\to 0$ gives $\Egu X \leq X_0$.
\end{proof}

\begin{corollary}\label{cor:proto-ville}
  Let $(\Y,\hat\Z,T)$ be a sequentially normalized gamble space.
  Then for all $A\subseteq\Y^T$ we have
  \begin{align*}
    \Pgu A = \inf\{ X_0 \mid \{X_t\in\reals\}_t \text{ game-theoretic supermartingale}, X_T \geq \ones_A \}~.
  \end{align*}
  If $(\Y,\hat\Z,T)$ is scalable, then  
  \begin{align*}
    \Pgu A = \inf\left\{ \alpha > 0 \,\middle|\, \{X_t \geq 0\}_t \text{ game-theoretic supermartingale}, X_0 = 1, X_T \geq \frac 1 \alpha \text{ on } A\right\}~.
  \end{align*}
\end{corollary}
\begin{proof}
  The result follows if we can show that gambles can be chosen to be real-valued without loss of generality.
  This follows from the boundedness of $X = \ones_A$.
  More generally let $X$ be arbitrary with $b := \sup |X| < \infty$.
  Then by normalization and translation $\Egu X \in [-b,b]$.
  Let $\alpha > \Egu X$ and $Z^\psi\in\Z_T$ with $Z^\psi + \alpha \geq X$.
  We must have $\psi(\cdot)(\cdot) > -\infty$, since otherwise $Z^\psi(y) = -\infty$ for at least one $y\in\Y^T$.
  Now recursively define $\psi'(s)(y) = 2b - Z^{\psi'}(s)$ if $\psi(s)(y) = \infty$, with $\psi'(s\oplus s') = 0$ for all $s'\in\Y^{T-|s|}$, and otherwise $\psi'(s,y) = \psi(s,y)$. 
  By assumption, we have $Z^{\psi'} + \alpha \geq X$ both on paths where $\psi$ was finite, in which case $Z^\psi = Z^{\psi'}$, or otherwise, in which case $Z^{\psi'} = 2b$.
  
\end{proof}

\subsection{Online machine learning}
\label{sec:online-learning}

Let us briefly see how to cast adversarial online machine learning as a sequential gamble space.
In one common setting, we have some abstract per-round outcome set $\Y$, action space $\A$, and a scoring rule (negative loss function) $S:\A\times\Y\to\extreals$.
For example, a binary prediction setting could take $\A = [0,1]$, $\Y = \{-1,1\}$, and $S$ to be log score, $S(p,1) = \log p$ and $S(p,0) = \log(1-p)$.
An online learning algorithm $\Alg$ must choose $a_t$ as a function of $y_{1..t-1}$.

In adversarial online learning, one seeks worst-case learning guarantees.
These guarantees cannot be in the sense of absolute performance, as World can simply choose an outcome sequence making Gambler maximally incorrect on every round.
Instead, we measure performance as the \emph{regret} relative to some benchmark, such as the performance of the best fixed action in hindsight,
\begin{align}
  \label{eq:regret}
  \Reg_T(\Alg,y) = \sup_{a\in\A} \sum_{t=1}^T S(a,y_t) - \sum_{t=1}^T S(a_t,y_t)~,
\end{align}
where $a_t$ is the sequence of actions chosen by $\Alg$.
The most pessimistic view of the algorithm's performance is the \emph{worst-case regret} over all sequences $y\in\Y^T$,
\begin{align}
  \label{eq:minimax-regret}
  \Reg_T(\Alg) = \sup_{y\in\Y^T} \Reg_T(\Alg,y)~,
\end{align}
The best possible worst-case regret $\Reg^*_T(S)$, often called the \emph{minimax regret}, is simply the infimum of $\Reg_T(\Alg)$ over all possible algorithms.

The setting described above is exactly the simple repeated gamble space $(\Y,\hat\Z,T)$ with $\hat\Z = \{ y \mapsto S(a,y) \mid a\in\A \}$.
In general, this gamble space is not sequentially normalized.
Assuming without loss of generality that $S(a,\cdot)$ is unique for each $a\in\A$, the set of possible algorithms $(\Y^*\to\A)$ is in bijection with the set of strategies $\Psi = (\Y^*\to\hat\Z)$, and in turn with the cumulative gambles $Z^\psi \in \Z_T$.
Thus, phrased in terms of gambles spaces, the minimax regret is simply the replication cost of the benchmark $X(y) = \sup_{a\in\A} \sum_{t=1}^T S(a,y_t)$:
\begin{align}
  \Reg^*_T(S)
  &=
    \inf_{\Alg\in (\Y^*\to\A)} \sup_{y\in\Y^T} \left( \sup_{a\in\A} \sum_{t=1}^T S(a,y_t) - \sum_{t=1}^T S(\Alg(y_{1..t-1}),y_t) \right)
  \\
  &=
    \inf_{Z\in\Z_T} \sup_{y\in\Y^T} X(y) - Z(y)
  \\
  &=
  \Egu X~.
\end{align}
In other words, the minimax regret of a particular online learning setting is exactly the game-theoretic upper expectation of the benchmark $X$, with respect to gambles given by the allowed algorithms.
In particular, for any algorithm, we must have $\Reg_T(\Alg) \geq \Egu X$.

Moreover, any online learning algorithm $\Alg$ induces a particular game-theoretic supermartingale $\{X^\Alg_t\}_t$, given by
\begin{align}
  \label{eq:online-learning-supermartingale}
  X^\Alg_t(y_{1..t}) = \sum_{i=1}^t S(\Alg(y_{1..i-1}),y_i) + \Reg_T(\Alg)~,
\end{align}
which satisfies both $X^\Alg_T \geq X$ and $X^\Alg_0 = \Reg_T(\Alg)$.
That is, $\Alg$ induces a game-theoretic supermartingale that replicates $X$, as in eq.~\eqref{eq:ville-supermartingale-egu}, and $\Reg^*_T(S) = \inf\{X^\Alg_0 \mid \Alg\in(\Y^*\to\A)\}$ is the minimum ``starting capital'' needed to replicate $X$.
As we explore further in \S~\ref{sec:online-learning-supermartingales}, this particular perspective is implicit in the ``Relax and Randomize'' framework of \citet[eq.~(4)]{rakhlin2012relax}.
Their ``admissible relaxations'' $\Rel_T$ are also game-theoretic supermartingales that replicate the benchmark $X$, which one can in turn replicate on each round with a particular algorithm.

Let us illustrate the connections between online learning and game-theoretic probability in the bounded outcome setting of \S~\ref{sec:motivating-example} and Example~\ref{ex:lln-gambles-multiplicative}.
Here $\hat\Z = \{y\mapsto \beta y \mid \beta \in \reals\}$.
As we saw in Example~\ref{ex:lln-gambles-multiplicative}, the logarithmic version $\hat\Z^{\log} = \{y\mapsto \log(1 + \alpha y) \mid \alpha \in \reals\}$ is a generalization of the log scoring rule example above, defining $S(\alpha,y) = \log(1+\alpha y)$.
As $S$ is concave in $\alpha$, this $\hat\Z^{\log}$ setting is an instance of online convex optimization.
Let us take any online learning algorithm $\Alg$, such as online gradient descent, online Newton step, or follow-the-regularized-leader (FTRL), which achieves $o(T)$ regret with respect to the benchmark $X(y) = \sup_{\alpha\in[-1,1]} \sum_{t\leq T} \log(1 + \alpha y_t)$ defined above.
\footnote{See \citet{hazan2016introduction,orabona2019modern} for an overview of online learning and online optimization algorithms.  These algorithms typically require the gradients of $S$ to be bounded; as discussed below, without loss of generality we may restrict e.g.\ $|\alpha| \leq 1/2$, giving bounded gradients.}

If $y\notin\ALLN$, one can verify by a simple Taylor approximation that $X(y) = \Omega(T)$.
This fact continues to hold when restricting $\alpha \in [-c,c]$ for any $c \in (0,1)$.
In other words, there exists some constant-fraction betting strategy $\psi^\alpha$ in $(\Y,\hat\Z,T)$, the original gamble space of \S~\ref{sec:motivating-example}, which sets $C_0 = 1$ and bets an $\alpha_t = \alpha$ fraction of the current capital $C_t$, giving $C_T = 1 + Z^{\psi^\alpha}_T(y) = \prod_{t\leq T} (1+\alpha y_t) = \exp(\Omega(T))$.

Back in gamble space $(\Y,\hat\Z^{\log},T)$, we thus have $Z^\Alg_T \geq X - o(T) \geq \Omega(T) - o(T) = \Omega(T)$.
Letting $\psi^*$ be the strategy in $(\Y,\hat\Z,T)$ that sets $C_0 = 1$ and bets $\alpha_t = \Alg(y_{1..t-1})$, we thus have $Z^{\psi^*}_T \geq -1$ and $C_T = 1 + Z^{\psi^*}_T = \exp(Z^\Alg_T) = \exp(\Omega(T))$ on $(\ALLN)^c$.
Thus, not only does such a $\psi^*$ exhibit $\Pgu (\ALLN)^c = 0$, it even achieves an exponential growth in capital on $(\ALLN)^c$.

From this example, we can see that a valid $\psi^*$ could even have $\Reg_T(\Alg) = \Theta(T)$, meaning an algorithm could fail to have no regret in $\Z^{\log}$ but still exhibit $\Pgu (\ALLN)^c = 0$.
Roughly speaking then, one could think of online learning guarantees as finer grained than those of game-theoretic probability.

\citet{orabona2016coin} show how to use strategies $\psi^*$ for this same setting to develop parameter-free online optimization algorithms.
\citet{rakhlin2012relax} also give parameter-free algorithms, implicitly via game-theoretic supermartingales (\S~\ref{sec:online-learning-supermartingales}), though it is not clear whether they can also be viewed as reductions from this same setting.

\section{Consistency, prices, and minimax duality}
\label{sec:prices-probability}

In this section, we introduce several new upper expectation operators which help relate the game-theoretic upper expectation to the more familiar landscape of measure-theoretic probability.
These operators can be seen as placing restrictions on World, as follows.
\begin{enumerate}[(i)]
\item $\Epucons$: World is passive, probabilistic, and consistent with the gambles (i.e., constrained so that no gamble makes money in expectation).
\item $\Epu$: World is adversarial and probabilistic but cannot respond to Gambler.
\item $\Egu$: World is worst-case responding to Gambler's choices.
\end{enumerate}
The corresponding quantities $\Epucons X$, $\Epu X$, $\Egu X$ can be thought of as \emph{prices} at which Gambler would be willing to sell the variable $X$, given the corresponding assumption (i--iii) about World.
Intuitively, these prices should increase as one moves from (i) to (ii) to (iii), as the increasing power of World should make it harder for Gambler to replicate $X$.

We can write these operators formally as follows.
Here $\Delta_0(\Z)$ is the set of consistent probability measures (Definition~\ref{def:consistency}), and $\mdcl(\Z)$ is a set of measurable gambles (Definition~\ref{def:mdcl}).
\begin{equation}
  \label{eq:intro-price-defs-chain}
  \begin{tikzpicture}[baseline]
    \matrix [matrix of math nodes,row sep=0.1cm,column sep=0.1cm] (m) {
      \text{(i)} && \text{(ii)} && \text{(iii)}
      \\
      {\displaystyle \sup_{P\in\Delta_0(\Z)} \E_P X} & \leq
      &\displaystyle \sup_{P\in\Delta(\Omega)} \inf_{Z\in\mdcl(\Z)} \E_P[ X - Z]
      &\leq
      &\displaystyle \inf_{Z\in\Z} \sup_{\omega\in\Omega} X(\omega) - Z(\omega)~.
      \\[-6pt]
      \rotatebox{90}{$:=$}&&\rotatebox{90}{$:=$}&&\rotatebox{90}{$:=$}
      \\[-4pt]
      \Epucons X && \Epu X && \Egu X
      \\ };
  \end{tikzpicture}
\end{equation}
Combining these inequalities with the interpretation of these quantities as prices under beliefs (i)--(iii), we can visualize them like an financial order book as in Fig.~\ref{fig:intro-price-chain}.

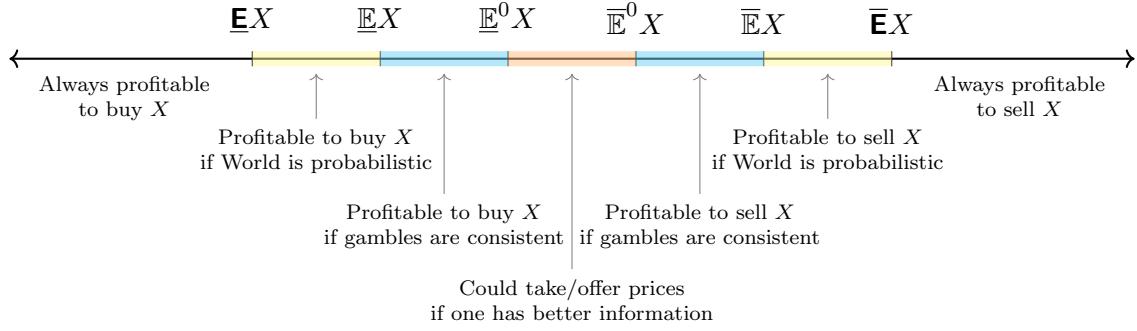
\begin{figure}[t]
  \centering
  \begin{tikzpicture}[xscale=1.7]
    
    \draw[<->, thick] (-0.9,0) -- (7.9,0);
    
    \foreach \i/\t in {1/{$\Egl X$},2/{$\Epl X$},3/{$\Epl^0 X$},4/{$\Epucons X$},5/{$\Epu X$},6/{$\Egu X$}} {
      \draw (\i,0.1) -- (\i,-0.1) node[above=3mm] {\t};
    }
    
    \foreach \i/\col in {1/yellow, 2/cyan, 3/orange, 4/cyan, 5/yellow} {
      \draw[rounded corners=5mm, line width=2mm, \col!50, opacity=0.5] (\i,0) -- ++(1,0);
    }
    
    \foreach \i/\j/\t in {
      1/1/{Profitable to buy $X$\\ if World is probabilistic},
      2/2/{Profitable to buy $X$\\ if gambles are consistent},
      3/3/{Could take/offer prices\\ if one has better information},
      4/2/{Profitable to sell $X$\\ if gambles are consistent},
      5/1/{Profitable to sell $X$\\ if World is probabilistic}} {
      \draw[<-, gray] (\i+0.5,-0.2) -- (\i+0.5,0.2-\j) node[below, align=center, font=\scriptsize, black] {\t};
    }
    \node[below=1mm, align=center, font=\scriptsize] at (0,0) {Always profitable\\ to buy $X$};
    \node[below=1mm, align=center, font=\scriptsize] at (7,0) {Always profitable\\ to sell $X$};
    
  \end{tikzpicture}
  \caption{The chain of prices of some variable $X$ implied by eq.~\eqref{eq:intro-price-defs-chain}.  Observe that negations of these statements can be inferred too: if it is always profitable to buy $X$ at a price under the assumption listed, it is never profitable to sell at that price.  Hence, these prices are from the perspective of one agent only; if both agents have the same assumption about World, no trade would occur unless both agents are indifferent.  Taken in pairs, one can think of these quantities as the bid-ask spreads offered by market makers with assumptions (i-iii) about World.
    }
  \label{fig:intro-price-chain}
\end{figure}

The inequalities are established in Theorem~\ref{thm:chain-of-price-inequalities}; let us sketch the proof.
By definition, we must have $\E_P  X \leq \E_P [X-Z]$ for all $P\in\Delta_0(\Z)$.
The first inequality now follows from taking a supremum over $P$, and the fact that a supremum over a larger set is weakly larger.
For the second inequality, first observe that without loss of generality World could choose a probability measure $P\in\Delta(\Omega)$ instead of a specific outcome in the definition of $\Egu X$.
Then this inequality holds simply because it is weakly better to play second in a zero-sum game, after seeing the move of your opponent.

These operators and their relationship become more complex when considering sequential gambles.
In particular, it is natural to consider the sequential version of $\Epucons$, denoted $\Epuseq$, where World is required to play a \emph{sequentially consistent} probability measure, whose conditional measures on each round are consistent for the gambles $\hat\Z^{(s)}$ on that round.
The set of such sequentially consistent measures is denoted $\Delta_0^T(\hat\Z)$.
It turns out that we must restrict the gambles allowed for (ii) to some $\Z'$, since otherwise sequentially consistent probability measures can fail to be consistent, and moreover the chain of inequalities can fail (Example~\ref{ex:seq-consistent-but-no-consistent}).
Fortunately, choices like $\Z' = \bb(\Z)$ are rich enough that we still have $\Epuseq_{\Z'} = \Epuseq_\Z$ and $\Egu_{\Z'} = \Egu_\Z$ for bounded-below variables, so in particular we can conclude $\Epuseq \leq \Egu$ (Corollary~\ref{cor:sequential-price-inequalities}).

\begin{equation}
  \label{eq:sequential-price-defs-chain}
  \begin{tikzpicture}[baseline]
    \matrix [matrix of math nodes,row sep=0.1cm,column sep=0.1cm] (m) {
      \text{(i)} && \text{(ii)} && \text{(iii)}
      \\
      {\displaystyle \sup_{P\in\Delta_0^T(\hat\Z)} \E_P X} & \leq
      &\displaystyle \sup_{P\in\Delta(\Y^T)} \inf_{Z\in\mdcl(\Z')} \E_P[ X - Z]
      &\leq
      &\displaystyle \inf_{Z\in\Z} \sup_{\omega\in\Y^T} X(\omega) - Z(\omega)~.
      \\[-6pt]
      \rotatebox{90}{$:=$}&&\rotatebox{90}{$:=$}&&\rotatebox{90}{$:=$}
      \\[-4pt]
      \Epuseq X && \Epu_{\Z'} X && \Egu X
      \\ };
  \end{tikzpicture}
\end{equation}

As we will see in \S~\ref{sec:translating-mtp-gtp}, many measure-theoretic results of interest can be phrased in the form $\Epuseq X \leq c$ for some $X$ and $c$.
For example, since the probability measures $P$ which are sequentially consistent with the gambles for bounded outcomes in \S~\ref{sec:introduction} and Example~\ref{ex:bounded-sequential} are the martingale measures, the measure-theoretic bounded law of large numbers can be phrased as $\Epuseq \ones_{(\ALLN)^c} \leq 0$.
In light of these sequential price inequalities~\eqref{eq:sequential-price-defs-chain}, a corresponding game-theoretic version $\Egu X \leq c$ is therefore stronger, as we would have $\Epuseq X \leq \Egu X \leq c$.
Conversely, we could \emph{derive} the game-theoretic version if we had price equality, $\Epuseq X = \Egu X$.
After showing the inequalities~(\ref{eq:intro-price-defs-chain}, \ref{eq:sequential-price-defs-chain}), we will give sufficient conditions for the equality $\Epuseq X = \Epu X$ (Corollary~\ref{cor:ep0star-equals-ep}).
Thus, under these conditions, the equality $\Epuseq X = \Egu X$ is equivalent to $\Epu X = \Egu X$, which is a minimax theorem (Proposition~\ref{prop:price-equality-minimax}).

\subsection{Measurable variables and gambles}
\label{sec:measurability}

To connect the preceding game-theoretic definitions with their measure-theoretic counterparts, we will equip $\Omega$ with a $\sigma$-algebra $\Sigma$.
Unless otherwise stated, $\Sigma$ is arbitrary.
When $\Omega \subseteq \reals^d$, we generally assume $\Sigma$ is the Borel $\sigma$-algebra.
Given $\Omega$ and $\Sigma$, we define $\Delta(\Omega)$ to be the corresponding set of probability measures.
We sometimes use the notation $\X$ for the set of measurable measurable functions $\Omega\to\extreals$, and $\X_b\subseteq\X$ the set of bounded measurable functions.

Similar to the downward closure, we will often make use of the measurable version.

\begin{definition}[Measurable downward closure]
  \label{def:mdcl}
  Given $\Z \subseteq (\Omega \to \extreals)$, define the \emph{measurable downward closure} by $\mdcl(\Z) = \dcl(\Z) \cap \X$.
\end{definition}

The following result states that we preserve the game-theoretic upper expectation when replacing $\Z$ by the measurable gambles $\mdcl(\Z)$.
This statement is absolutely essential when attempting to draw connections with measure-theoretic probability.
Of course, one of the benefits of game-theoretic probability is avoiding the need to discuss measurability.
In previous work such as \citet[Theorem 9.3, Corollary 9.18]{shafer2019game}, however, deducing measure-theoretic results from game-theoretic ones requires one to establish the measurability of the gambling strategy.
Yet this proposition states that gambling strategies can \emph{always} be taken to be measurable when the variable $X$ in question is measurable.
The proof is immediate from Proposition~\ref{prop:egu-dcl-preserves-properties} with the class $\G$ of measurable variables.

\begin{proposition}\label{prop:mdcl}
  Let $(\Omega,\Z)$ be a gamble space, and $X:\Omega\to\infreals$ a measurable variable.
  Then $\Egu_\Z X = \Egu_{\mdcl(\Z)} X$.
\end{proposition}

\begin{remark}
  An argument with a similar goal appears in the proof of Ville's Theorem in \citet[Theorem 9.3]{shafer2019game}.
  That argument seems to rely on every variable to be ``priced'', in particular so that the game-theoretic upper and lower expectations are both equal to a measure-theoretic expectation.
  Proposition~\ref{prop:mdcl} is thus much more general.
  In particular, it allows for much smaller gamble spaces, with a nontrivial gap between upper and lower expectations.
\end{remark}

\begin{remark}
  For non-measurable $X$, the statement
  $\Egu_\Z X = \Egu_{\mdcl(\Z)} X$
  need not hold.
  Take for example $X = \ones_A$ for some non-measurable set $A$ other than $\emptyset$ or $\Omega$.
  Let $\Z = \{Z_\beta: \omega\mapsto \beta(2\ones\{\omega\in A\}-1) \mid \beta\in\reals\}$.
  These gambles offer even odds on whether the outcome falls in $A$: $Z_\beta$ is $\beta$ when $\omega\in A$ and $-\beta$ otherwise.
  By the same argument as in Example~\ref{ex:coin-formal}, we have $\Eg X = \Pg A = 1/2$.
  The gambling strategy that replicates $X = \ones_A$, however, is non-measurable.
  
  For any $Z\in\mdcl(\Z)$, by definition we have $Z$ measurable and $Z\leq Z_\beta$ for some $\beta\in\reals$.
  For $\omega\in A$ we have $X(\omega) - Z(\omega) \geq 1 - \beta$.
  For $\omega\notin A$ we have $X(\omega) = -Z(\omega) \geq \beta$.
  Thus, $\sup X - Z \geq \max(\beta,1-\beta)$.
  Now suppose for a contradiction that some sequence $\{Z_n\}_n \subseteq \mdcl(\Z)$ achieved $\lim_{n\to\infty} \sup X - Z_n = 1/2$.
  The corresponding $\beta_n$ must therefore satisfy $\lim_{n\to\infty} \beta_n = 1/2$.
  We conclude that $Z_n \to \ones_A - 1/2$ pointwise.
  As the $Z_n$ are measurable, $\ones_A - 1/2$ would be measurable as the pointwise limit of measurable functions, a contradiction.
  Thus $\Egu_{\mdcl(\Z)} X > 1/2$.
\end{remark}

\begin{remark}
  The restriction $X > -\infty$ in Proposition~\ref{prop:mdcl} may also be necessary.
  Without it, one could try define $Z_n(\omega) = \inf_{\omega\in\Omega'} Z_n'(\omega)$ on $\Omega'$, but this infimum could be $-\infty$.
  One could alternatively partition $\Omega'$ into measurable subsets, but one would need to take care that the infimum is not $-\infty$ on these partitions.
  
\end{remark}

\subsection{Consistency and sequential consistency}
\label{sec:consistency-seq-def}

To define the prices (i) where World must play a consistent probability measure, we introduce this notion formally.
Roughly, a probability measure is consistent if no gamble is profitable in expectation, and sequentially consistent if that holds in (almost) every round.

\begin{definition}[Consistency]
  \label{def:consistency}
  We say $P\in\Delta(\Omega)$ and measurable $Z:\Omega\to\extreals$ are \emph{consistent} if $\E_P Z \leq 0$.
  Similarly, a set $\P \subseteq \Delta(\Omega)$ is \emph{consistent with} $\Z$ if $\E_P Z \leq 0$ for all $P\in\P, Z\in\Z$.
  For gamble space $(\Omega,\Z)$, we define
  \begin{align}
    \label{eq:consistent-dist}
    \Delta_0(\Z) &:= \{P\in\Delta(\Omega) \mid \E_P Z \leq 0 \;\forall Z\in\mdcl(\Z)\}~,
  \end{align}
  to be the set of all consistent probability measures.
\end{definition}

\begin{definition}[Sequential consistency]
  \label{def:seq-consistency}
  Let $(\Y,\hat\Z,T)$ be a sequential gamble space.
  We say $P\in\Delta(\Y^T)$ is \emph{sequentially consistent} if $P$ admits a regular conditional probability measure
  \footnote{The measure $P(\cdot \mid s)$ is \emph{regular} if (i) $P(\cdot \mid s) \in \Delta(\Y)$ holds $P$-a.s., and (ii) $s\mapsto P(\cdot \mid s)$ is a measurable function.  (Recalling that $\Y$ is a measurable space, we equip $\Y^{<T}$ with the disjoint union $\sigma$-algebra on $\Y^{<T} = \bigsqcup_{t< T} \Y^t$.)}
  and
  \begin{align}
    \label{eq:2}
    P(Y_{t+1} \mid Y_{1..t}) \in \Delta_0(\hat\Z^{(Y_{1..t})})
  \end{align}
  holds $P$-a.s.\ for all $t<T$.
  We define $\Delta_0^T(\hat\Z) \subseteq \Delta(\Omega)$
  to be the set of probability measures sequentially consistent with $\hat\Z$.
\end{definition}

\begin{remark}[Constructing sequentially consistent measures]
  Let $(\Y,\hat\Z,T)$ be a sequential gamble space, where $\Y$ equipped with sigma-algebra $\Sigma$.
  Define a \emph{measurable kernel} to be a function $\kappa:\Y^{<T} \to \Delta(\Y)$
  such that $s \mapsto \kappa(s)(A)$ is measurable for all $A \in \Sigma$.
  The Ionescu-Tulcea Theorem gives that any measurable kernel $\kappa$ defines a unique probability measure $P^\kappa$ on $\Y^T$.
  We say $\kappa$ is a \emph{consistent kernel} if it is a measurable kernel with $\kappa(s) \in \Delta_0(\hat\Z^{(s)})$ for all $s\in\Y^{<T}$.
  Then $P\in\Delta(\Y^T)$ is sequentially consistent if and only if $P = P^\kappa$ for some consistent kernel $\kappa:\Y^{<T} \to \Delta(\Y)$.
\end{remark}

To relate the global and local views, and in particular to relate global consistency to local consistency, it will be convenient to restrict to \emph{disintegrable} spaces, where all probability measures admit regular conditional probability measures.

\begin{definition}[Disintegrable]
  \label{def:disintegrable}
  A sequential gamble space $(\Y,\hat\Z,T)$ is \emph{disintegrable} if every $P\in\Delta(\Y^T)$ admits a regular conditional probability measure.
\end{definition}
One convenient sufficient condition is that $\Y$ be a Polish space (a separable completely metrizable topological space)~\citep[\S~21.4]{fristedt1996modern}.

A natural question is the relationship between consistency and sequential consistency.
Our first result shows that consistency is a weakly stronger condition.

\begin{proposition}\label{prop:consistent-implies-seq-consistent}
  Let $(\Y,\hat\Z,T)$ be a disintegrable, sequentially normalized gamble space.
  Then $\Delta_0(\Z_T) \subseteq \Delta_0^T(\hat\Z)$.
  When $T=\infty$, we also have $\Delta_0(\bigcup_{t\in\N} \Z_t) \subseteq \Delta_0^\infty(\hat\Z)$.
\end{proposition}
\begin{proof}
  Fix some $P\in\Delta(\Y^T)$ with regular conditional probability measure $P(\cdot \mid \cdot)$.
  Supposing $P\notin\Delta_0^T(\hat\Z)$, we will show $P\notin\Delta_0(\Z_k)$ for some $k \in \N$, $k\leq T$, which covers both cases.
  
  Let $\Sc_+ = \{s\in\Y^{<T} \mid P(Y_{|s|+1} \mid s) \notin \Delta_0(\hat\Z^{(s)})\}$ be the situations where $P$ is sequentially inconsistent.
  Let $A_k = \{y\in\Y^T \mid y_{1..k} \in \Sc_+\}$ be the event that there is a sequentially inconsistent prefix of $y$ of length $k$, and $A_k = \emptyset$ for $k>T$.
  As $P\notin\Delta_0^T(\hat\Z)$, we have $P\left(\bigcup_{k\in\N} A_k\right) > 0$.
  We conclude $P(A_k) > 0$ for some $k \in \N$.

  For each $s\in \Sc_+$, let $\hat Z^s\in\mdcl(\hat\Z^{(s)})$ be a gamble such that $\E_P[\hat Z^s \mid s] > 0$.
  Consider the strategy $\psi(s) = \hat Z^s$ if $s\in \Sc_+\cap \Y^k$ and $\psi(s) = 0$ otherwise; observe $Z^\psi \in \mdcl(\Z_k)$.
  We have $\E_P[ Z^\psi \mid Y_{1..k} ] > 0$ for $Y_{1..k}\in \Sc_+$ and $\E_P[ Z^\psi \mid Y_{1..k} ] = 0$ otherwise.
  As $P(Y_{1..k}\in\Sc_+) = P(A_k) > 0$, Markov's inequality gives
  $\E_P Z^\psi = \E_P[ \E_P[ Z^\psi \mid Y_{1..k} ] ] > 0$.
\end{proof}

Interestingly, the converse need not hold: sequentially consistent probability measures can fail to be consistent in the global sense.
In other words, even if gambles are not profitable in every round, the cumulative gamble can still have a positive expected value across all rounds.

\begin{example}\label{ex:seq-consistent-but-no-consistent}
  
  Consider the simple repeated gamble space $(\Y,\hat\Z,\infty)$ where $\Y=[-1,1]$, equipped with the Borel $\sigma$-algebra, and $\hat\Z = \{y\mapsto\beta y \mid \beta \in \reals\}$.
  As $\Y$ is Polish, the gamble space is disintegrable.
  Consider $P\in\Delta(\Y^\infty)$ which is the i.i.d.\ probability measure with $P(Y_t=1 \mid Y_{1..t-1}) = P(Y_t=-1 \mid Y_{1..t-1}) = 1/2$ for all $t$.
  Clearly $P\in\Delta_0^\infty(\hat\Z)$.
  Now consider $Z^\psi \in \Z_\infty$ for the strategy $\psi(s) = (y\mapsto y)$ if $Z^\psi(s) \neq 1$ and $\psi(s) = (y\mapsto 0)$ otherwise, i.e., bet \$1 on $y=1$ in each round until the total winnings equal 1, then stop.
  Then we have $\E_P Z^\psi = 1 > 0$.
  We conclude $P\notin\Delta_0(\Z_\infty)$.
  In fact, for the special case $\Y = \{-1,1\}$, we have $\Delta_0^\infty(\hat\Z) = \{P\}$, and thus $\Delta_0(\Z_T) = \emptyset$ since $\Delta_0(\Z_\infty) \subseteq \Delta_0^\infty(\hat\Z)$ from Proposition~\ref{prop:consistent-implies-seq-consistent}.
\end{example}

Example~\ref{ex:seq-consistent-but-no-consistent} relies on a particular violation of the optional stopping theorem: a $P$-(super)martingale $\{Z_t\}_t$ with $Z_0 = 0$, and a stopping time $\tau$ with $\E_P Z_\tau > 0$.
Intuitively, these are the only problematic gambles: if we restrict gambles to $\Z'$ such that $\E_P Z_\tau \leq 0$ for all sequentially consistent $P$, all $Z\in\Z'$, and all stopping times $\tau$, then sequential consistency should imply consistency.
In particular taking $\Z' = \bb(\Z_\infty)$, aligning with \citet{shafer2019game} as discussed in Remark~\ref{remark:defining-Z-infty}, would suffice.
(Recall from Proposition~\ref{prop:bounded-X-gambles} that this restriction does not change $\Egu X$ for bounded-below $X$.)
Yet we will eventually need a larger set of gambles, so here we additionally allow all finite stopping times, another case where the optional stopping theorem holds.

\begin{proposition}\label{prop:sequentially-consistent-is-consistent}
  Let $(\Y,\hat\Z,\infty)$ be a sequentially normalized gamble space.
  Then $\Delta_0^\infty(\hat\Z) \subseteq \Delta_0(\Z')$, where $\Z' = \bb(\Z_\infty) \cup \bigcup_{t\in\N} \Z_t$.
\end{proposition}
\begin{proof}
  Let $P\in\Delta_0^\infty(\hat\Z)$.
  First consider $Z^\psi = Z^\psi_k \in \Z_k$ for some $k \in \N$.
  As each term in the sum $Z^\psi_k = \sum_{t=1}^T \psi(Y_{1..t-1})(Y_t)$ has non-positive expectation, we have $\E_P[Z^\psi] \leq 0$.
  Now consider $Z^\psi \in \bb(\Z_\infty)$.
  By definition, $Z^\psi = \liminf_{t\to\infty} Z^\psi_t$.
  By definition of $\bb(\cdot)$, we have $b:= \inf Z^\psi > -\infty$.
  For any $t$, no arbitrage gives $\inf Z^\psi \leq \inf Z^\psi_t$.
  (To see this inequality, note that for all $s\in\Y^t$, we have $\inf_{y_{t+1}\in\Y} Z^\psi_{t+1}(s\oplus y_{t+1}) = Z^\psi_t(s) + \inf_{y_{t+1}\in\Y} \psi(s)(y_{t+1}) \leq Z^\psi_t(s)$ by no-arbitrage of $\hat\Z^{(s)}$; taking the infimum over $s$, the result now follows by induction.)
  
  Thus $Z^\psi_t \geq b$ for all $t$.
  By Fatou's Lemma, $\E_P[Z^\psi] = \E_P[\liminf_{T\to\infty} Z^\psi_t] \leq \liminf_{T\to\infty} \E_P[Z^\psi_t] \leq 0$ by the above.
  
\end{proof}

Given a sequential gamble space $(\Y,\hat\Z,T)$ for $T<\infty$, we can simply pad $\hat\Z^{(s)} = \{0\}$ for $|s|\geq T$ and the lemma applies for $\Z' = \Z_T$.

\subsection{Price definitions}
\label{sec:price-definitions}

We now define the prices (ii) and (i).
With $\mdcl(\Z)$ defined above, we can immediately define the price (ii) for a probabilistic World that must play first.

\begin{definition}[Measure-theoretic upper expectation]
  Let $(\Omega,\Z)$ be a gamble space, and $X:\Omega\to\extreals$ a measurable variable.
  Then we define
  \begin{align}
    \Epu X &:= \sup_{P\in\Delta(\Omega)} \inf_{Z\in\mdcl(\Z)} \E_P [X-Z]~,
  \end{align}
  and $\Epl X := -\Epu(-X)$.
  If $\E_P[X-Z]$ is not defined, we define it to be $\infty$.
\end{definition}

For the price (i) when World must play consistent probability measures, we break consistency into two cases: consistency and sequential consistency.

\begin{definition}[(Sequentially) consistent upper expectation]
  \label{def:sequential-consistency}
  For gamble space $(\Omega,\Z)$ and measurable variable $X:\Omega\to\extreals$, we define
  \begin{align}
    \Epucons X &:= \sup_{P\in\Delta_0(\Z)} \E_P X~,
  \end{align}
  and $\Epl^0 X := -\Epucons(-X)$.
  Similarly, for a sequential gamble space $(\Y,\hat\Z,T)$, we define
  \begin{align}
    \Epuseq X &:= \sup_{P\in\Delta_0^T(\hat\Z)} \E_P X~,
  \end{align}
  and $\Epl^* X := - \Epuseq(-X)$.
  In both cases, if $\E_P X$ is undefined, we define it to be $\infty$.
\end{definition}

\subsection{The chain of global price inequalities}
\label{sec:chain-price-inequalities-global}

We now turn to the inequalities stated above, beginning with the non-sequential version~\eqref{eq:intro-price-defs-chain}.

\begin{theorem}\label{thm:chain-of-price-inequalities}
  Let $(\Omega,\Z)$ be a gamble space and $X:\Omega\to\infreals$ a measurable variable.
  Then $\Egl X \leq \Epl X \leq \Epl^0 X$ and $\Epucons X \leq \Epu X \leq \Egu X$.
  If $\Delta_0(\Z) \neq \emptyset$, then
  \begin{align}\label{eq:price-inequality-chain}
    \Egl X
    \leq
    \Epl X
    \leq
    \Epl^0 X
    \leq 
    \Epucons X
    \leq
    \Epu X
    \leq
    \Egu X~.
  \end{align}
\end{theorem}
\begin{proof}
  We will assume here that all expectations are defined; see \S~\ref{sec:proofs-undef-expect} for the undefined cases.
  Proposition~\ref{prop:mdcl} gives $\Egu X = \Egu_{\mdcl(\Z)} X$.
  We first observe that randomization cannot hurt World.
  Letting $\delta_\omega \in \Delta(\Omega)$ be the point mass on $\omega$, we have
  $\sup_{\omega\in\Omega} X(\omega) - Z(\omega) = \sup_{\omega\in\Omega} \E_{\delta_\omega}[X - Z] \leq \sup_{P\in\Delta(\Omega)} \E_P[X - Z]$ for all measurable $Z$.
  As $\E_P[X - Z] \leq \sup_{\omega\in\Omega} X(\omega) - Z(\omega)$ for any $P\in\Delta(\Omega)$, we have
  \begin{align}
    \Egu X
    &= \inf_{Z\in\mdcl(\Z)} \sup_{\omega\in\Omega} X(\omega) - Z(\omega)
      \nonumber
    \\
    \label{eq:nature-can-randomize}
    &= \inf_{Z\in\mdcl(\Z)} \sup_{P\in\Delta(\Omega)} \E_P[X - Z]~.
  \end{align}
  We now apply a standard observation from game theory, that playing second is weakly better than playing first:
  for any $P' \in \Delta(\Omega)$ we have $\E_{P'}[X-Z] \leq \sup_{P\in\Delta(\Omega)} \E_P[X - Z]$ and thus
  \[\inf_{Z\in\mdcl(\Z)} \E_{P'}[X-Z] \leq \inf_{Z\in\mdcl(\Z)} \sup_{P\in\Delta(\Omega)} \E_P[X - Z]~.\]
  As $\Epu X$ corresponds to taking a supremum over all choices of $P'$, we have
  \begin{align*}
    \Epu X
    &= \sup_{P\in\Delta(\Omega)} \inf_{Z\in\mdcl(\Z)} \E_P[X - Z]
    \\
    &\leq
      \inf_{Z\in\mdcl(\Z)} \sup_{P\in\Delta(\Omega)} \E_P[X - Z] \leq \Egu X~.
  \end{align*}
  We also have $\Egl X \leq \Epl X$ by the identities $\Egl X = -\Egu(-X)$ and $\Epl X = -\Epu(-X)$.

  Moving to $\Epucons$, we have
  \begin{align*}
    \Epu X
    &=    \sup_{P\in\Delta(\Omega)} \inf_{Z\in\mdcl(\Z)} \E_P[X - Z] \\
    &\geq \sup_{P\in\Delta_0(\Z)} \inf_{Z\in\mdcl(\Z)} \E_P[X - Z] \\
    &\geq \sup_{P\in\Delta_0(\Z)} \E_P X \\
    &=    \Epucons X~,
  \end{align*}
  by definition of $\Delta_0(\Z)$ and linearity of expectation.
  As above, we also conclude $\Epl X \leq \Epl^0 X$ via the identities $\Epl X = -\Epu(-X)$ and $\Epl^0 X = -\Epucons(-X)$.
  Finally, when $\Delta_0(\Z)\neq \emptyset$, we have
  \begin{align*}
    \Epl^0 X = \inf_{P\in\Delta_0(\Z)} \E_P X \leq \sup_{P\in\Delta_0(\Z)} \E_P X = \Epucons X~,
  \end{align*}
  which completes the chain of inequalities.
\end{proof}

A key consequence of Theorem~\ref{thm:chain-of-price-inequalities} is a fact that bridges game-theoretic and measure-theoretic probability: under a probability measure $P$ sequentially consistent with the gambles, game-theoretic supermartingales are $P$-supermartingales.
\begin{proposition}\label{prop:supermartingale-to-measure}
  Let $(\Y,\hat\Z,T)$ be a sequential gamble space, and $\{X_t\}_{t\leq T}$ a game-theoretic (super)martingale.
  For any $P\in\Delta_0^T(\hat\Z)$, the sequence $\{X_t\}_{t\leq T}$ is a $P$-(super)martingale.
\end{proposition}
\begin{proof}
  Let $P\in\Delta_0^T(\hat\Z)$, $t<T$, and $s\in\Y^t$.
  By definition, $P(Y_{t+1}\mid Y_{1..t}) \in \Delta_0(\hat\Z^{(Y_{1..t})})$ holds $P$-a.s.
  From Theorem~\ref{thm:chain-of-price-inequalities}, we then have
  \begin{align*}
    \E_P [X_{t+1} \mid s] \leq \Epucons [X_{t+1} \mid s] \leq \Egu[ X_{t+1} \mid s] \leq X_t(s) \quad P\text{-a.s.},
  \end{align*}
  so $\{X_t\}_{t\leq T}$ is a $P$-supermartingale.
  If $\{X_t\}_{t\leq T}$ is a game-theoretic martingale, then by the same theorem applied to $\{-X_t\}_{t\leq T}$, the original sequence is a $P$-martingale.
\end{proof}

\subsection{Price equality and minimax duality}
\label{sec:price-equality-minimax}

A natural question in light of Theorem~\ref{thm:chain-of-price-inequalities} is when price inequalities are actually equalities.
Let us first treat the non-sequential versions, $\Epucons X = \Epu X$ and $\Epu X = \Egu X$.
Our first result gives sufficient conditions for $\Epucons X = \Epu X$: when gambles are upward scalable, and $X$ is upper bounded.
The second shows why $\Epu X = \Egu X$ is a minimax theorem.
Finally, we give a simple minimax theorem for finitely-generated gamble spaces.

The price equality $\Epucons X = \Epu X$ essentially says that World need not consider inconsistent probability measures; they will never increase World's payoff.
Our sufficient condition for this equality is that gambles be upward scalable, which allows Gambler to infinitely penalize World for being inconsistent.

\begin{theorem}
  \label{thm:ep0-equals-ep}
  Let $(\Omega,\Z)$ be an upward-scalable gamble space such that $0\in\Z$.
  Let $X:\Omega\to\extreals$ be measurable and bounded above.
  Then $\Epucons X = \Epu X$.
\end{theorem}
\begin{proof}
  Let $b := \sup X < \infty$.
  Take any $P\in\Delta(\Omega)$.
  If $P\notin\Delta_0(\Z)$, then there exists $Z^*\in\mdcl(\Z)$ such that $c := \E_P Z^* > 0$.
  By upward scaling, there exists $\alpha > 0$ arbitrarily large so that $\alpha Z^* \in \mdcl(\Z)$, and thus $\E_P \alpha Z^* = \alpha c$.
  We conclude $\inf_{Z\in\mdcl(\Z)} \E_P[X - Z] \leq \inf_{\alpha > 0} \E_P[X - \alpha Z^*] = b - \sup_{\alpha > 0} \alpha c = -\infty$.
  
  Thus,
  \begin{align*}
    \Epu X
    &= \sup_{P\in\Delta(\Omega)} \inf_{Z\in\mdcl(\Z)} \E_P[X - Z] \\
    &= \max\left(-\infty,\sup_{P\in\Delta_0(\Omega)} \inf_{Z\in\mdcl(\Z)} \E_P[X - Z] \right)\\
    &= \sup_{P\in\Delta_0(\Z)} \inf_{Z\in\mdcl(\Z)} \E_P[X - Z] \\
    &= \sup_{P\in\Delta_0(\Z)} \E_P X - \sup_{Z\in\mdcl(\Z)} \E_P[Z] \\
    &= \sup_{P\in\Delta_0(\Z)} \E_P X \\
    &= \Epucons X~,
  \end{align*}
  where the penultimate equality follows from the definition of $\Delta_0(\Z)$ and the fact that $0\in\Z$.
\end{proof}

Let us now consider the price equality $\Epu X = \Egu X$.
As alluded to in the proof of Theorem~\ref{thm:chain-of-price-inequalities}, this condition is equivalent to minimax duality.

\begin{proposition}\label{prop:price-equality-minimax}
  Let $(\Omega,\Z)$ be a gamble space, and $X:\Omega\to\infreals$ measurable.
  Then $\Epu X = \Egu X$ if and only if
  \begin{equation}
    \sup_{P\in\Delta(\Omega)} \inf_{Z\in\mdcl(\Z)} \E_P [X-Z] = \inf_{Z\in\mdcl(\Z)} \sup_{P\in\Delta(\Omega)} \E_P[X-Z]~.
  \end{equation}  
\end{proposition}
\begin{proof}
  Proposition~\ref{prop:mdcl} gives $\Egu X := \Egu_{\Z} X = \Egu_{\mdcl(\Z)} X$.
  The result now follows from eq.~\eqref{eq:nature-can-randomize}.
\end{proof}

As motivated in \S~\ref{sec:introduction}, we can use minimax duality, and scalable gambles, to ``lift'' measure-theoretic statements to game-theoretic ones.
There we had $u(P,\psi) = \E_P[X - Z^\psi]$, and discussed how the game-theoretic law of large numbers could be derived from the measure-theoretic one if we had minimax duality.
We detail this general approach in \S~\ref{sec:translating-mtp-gtp}.

\begin{remark}\label{remark:epu-convexify-intuition}
  Intuitively, by forcing World to play a probability measure, $\Epu$ behaves like $\Egu$ after convexifying the gambles.
  More precisely, for finite $\Omega$, we have $\Epu_\Z = \Egu_{\overline \conv(\Z)}$, where $\overline \conv(\Z)$ is the closed convex hull of $\Z$; see \S~\ref{sec:structure-epu-finite-omega}.
  Thus, leveraging Proposition~\ref{prop:dcl-wlog}, minimax duality holds if $\dcl(\Z)$ is already convex; for finite $\Omega$, closure comes for free.
  In fact, this statement is an if and only if: if $\dcl(\Z)$ is not convex, one can find an $X$ for which minimax duality fails (\S~\ref{sec:structure-epu-finite-omega}).
  These statements become more subtle with larger $\Omega$ and require careful treatment of the topology, thus necessitating the work in \S~\ref{sec:minimax-theorems}.
\end{remark}

\subsection{Sequential price inequalities}
\label{sec:chain-price-inequalities-sequential}

We are now in a position to prove a chain of sequential price inequalities analogous to Theorem~\ref{thm:chain-of-price-inequalities}, as foreshadowed in eq.~\eqref{eq:sequential-price-defs-chain}.

\begin{corollary}\label{cor:sequential-price-inequalities}
  Let $(\Y,\hat\Z,\infty)$ be a sequentially normalized gamble space, and let $X:\Y^T\to\infreals$ be measurable and bounded below.
  Then $\Egl X \leq \Epl_{\Z'} X \leq \Epl^* X$ and $\Epuseq X \leq \Epu_{\Z'} X \leq \Egu X$,
  where $\Z' = \bb(\Z_\infty) \cup \bigcup_{t\in\N} \Z_t$.
  If we further have $\Delta_0^\infty(\hat\Z) \neq \emptyset$, 
  then
  \begin{align}\label{eq:sequential-price-inequality-chain}
    \Egl X
    \leq
    \Epl_{\Z'} X
    \leq
    \Epl^* X
    \leq 
    \Epuseq X
    \leq
    \Epu_{\Z'} X
    \leq
    \Egu X~.
  \end{align}
\end{corollary}
\begin{proof}
  Proposition~\ref{prop:sequentially-consistent-is-consistent} states $\Delta_0^\infty(\hat\Z) \subseteq \Delta_0(\Z')$, giving
  \[\Epuseq X := \sup_{P \in \Delta_0^\infty(\hat\Z)} \E_P X \leq \sup_{P \in \Delta_0(\Z')} \E_P X~.\]
  Theorem~\ref{thm:chain-of-price-inequalities} gives $\Epu_{\Delta_0(\Z')} X \leq \Epu_{\Z'} X \leq \Egu_{\Z'} X$.
  From Proposition~\ref{prop:gambles-not-neg-infty}, as $X$ is bounded below, we have $\Egu X := \Egu_{\Z_\infty} X = \Egu_{\bb(\Z_\infty)} X$.
  As $\bb(\Z_\infty) \subseteq \Z' \subseteq \Z_\infty$, we have $\Egu_{\Z_\infty} X \leq \Egu_{\Z'} X \leq \Egu_{\bb(\Z_\infty)}$ giving $\Egu X = \Egu_{\Z'} X$.
  When $\Delta_0^\infty(\hat\Z)\neq\emptyset$, we have
  \[\Epl^* X := \inf_{P\in\Delta_0^\infty(\hat\Z)} \E_P X \leq \sup_{P\in\Delta_0^\infty(\hat\Z)} \E_P X := \Epuseq X~,\]
  completing the chain.
\end{proof}

\begin{remark}
  \label{remark:self-replication-paradox}
  Let us return to Example~\ref{ex:seq-consistent-but-no-consistent}, where for $\Y = \{-1,1\}$ we had a particular gamble $Z^\psi \in \Z_\infty$ where $\psi$ bet \$1 on $y=1$ in each round but stopped when the total winnings reached 1.
  As we saw, letting $X = Z^\psi$, we have $\Epuseq X = 1$ but $\Epucons X = \Epu X = -\infty$, as the gamble space $(\Y^\infty,\Z_\infty)$ is upward scalable yet lacks consistent probability measures.

  Let us now compute $\Egu X$.
  It may seem self-evident that one can replicate $X = Z^\psi$ with zero initial capital: simply execute the same strategy $\psi$, and round for round, no matter the outcome, Gambler's capital will be the same as $Z^\psi$.
  Indeed, the ``pessimism'' in the choice $X = Z^\psi = \liminf_{t\to\infty} Z^\psi_t$ is actually the most favorable for Gambler.
  There is a subtlety, however: one cannot replicate gambling strategies, only variables.
  And for an outcome sequence where both $X$ and $Z^\psi$ plummet to $-\infty$, such as $y = -1$ in every round, our definition of the replition cost in \S~\ref{sec:basic-definitions} is $(-\infty) - (-\infty) = \infty$.
  To properly replicate this case, then, Gambler has no choice but to start with \$1 and refrain from gambling at all, giving $\Egu X = 1$.

  In summary, we have $\Epuseq X = \Egu X = 1$ but $\Epucons X = \Epu X = -\infty$ for this $X$.
  We can instead have $\Epucons X = \Epu X = 0$ by adding $0$ to $\Y$, i.e., for $\Y' = \{-1,0,1\}$.
  (Clearly $\Epucons X \leq 0$ by the point mass on the zero sequence, but as before if $\E_P X > 0$ then by definition $P$ is not consistent with $Z^\psi = X$.)
\end{remark}

\subsection{Sequential price equality and composite Ville}
\label{sec:seq-price-equal}

Turning to price equality, the sequential version of Theorem~\ref{thm:ep0-equals-ep}, from eq.~\eqref{eq:sequential-price-inequality-chain}, is $\Epuseq X = \Epu_{\Z'} X$ for some restriction $\Z'$ on $\Z_T$.
Intuitively, $\Z'$ must be large enough that Gambler can punish World for being sequentially inconsistent, but not so large that sequentially consistent probability measures can fail to be consistent.
Taking $\Z' = \bb(\Z_T)$ would suffice if $X$ was bounded below, but these statements do not always hold for settings of interest.
As discussed in Example~\ref{ex:seq-consistent-but-no-consistent}, a nice compromise is to allow all finite-horizon gambles, even those unboundedly negative, since one still has $\E_P Z \leq 0$ by sequential consistency.

\begin{corollary}\label{cor:ep0star-equals-ep}
  Let $(\Y,\hat\Z,\infty)$ be a disintegrable, upward-scalable, sequentially normalized gamble space, and $X:\Y^T\to\extreals$ measurable and bounded above.
  Then $\Epuseq X = \Epucons_{\Z'} X = \Epu_{\Z'} X$
  where $\Z' = \bb(\Z_\infty) \cup \bigcup_{t\in\N} \Z_t$.
\end{corollary}
\begin{proof}
  Sequential normality implies $0\in\Z_\infty$, and thus $0\in\Z'$.
  Propositions~\ref{prop:consistent-implies-seq-consistent} and~\ref{prop:sequentially-consistent-is-consistent} give $\Delta_0(\bigcup_{t\in\N} \Z_t) \subseteq \Delta_0^\infty(\hat\Z) \subseteq \Delta_0(\Z')$.
  As $\bigcup_{t\in\N} \Z_t \subseteq \Z'$, we also have $\Delta_0(\bigcup_{t\in\N} \Z_t) \supseteq \Delta_0(\Z')$,
  giving $\Delta_0(\bigcup_{t\in\N} \Z_t) = \Delta_0^\infty(\hat\Z) = \Delta_0(\Z')$.
  We conclude $\Epuseq X = \Epucons_{\Z'} X$.
  Theorem~\ref{thm:ep0-equals-ep} gives $\Epucons_{\Z'} X = \Epu_{\Z'} X$.
\end{proof}

\begin{remark}
  We will rely on conditions for $\Delta_0^T = \Delta_0$ so that the chain of price inequalities holds.
  But from Remark~\ref{remark:self-replication-paradox}, clearly the price equality we often care about, $\Epuseq = \Egu$, can hold even when the chain breaks down.
  It is an interesting open question to prove more general price equality conditions that hold even in those cases.
  
\end{remark}

Finally, as discussed in \S~\ref{sec:intro-composite-ville}, we now discuss formally why a sequential minimax theorem yields a composite version of Ville's Theorem.

\begin{proposition}
  \label{prop:minimax-implies-composite-ville}
  Let $(\Y,\hat\Z,T)$ be a disintegrable sequential gamble space which is sequentially scalable and arbitrage free.
  Let $\P = \Delta_0^T(\hat\Z)$.
  If $\Epuseq X = \Egu X$ for all bounded measurable $X:\Y^T\to\reals$, or all indicators $X = \ones_A$ for measurable $A\subseteq \Y^T$, then we have
  \begin{align*}
    \sup_{P\in\P} P(A)
    &= \inf\left\{ \alpha > 0 : \{X_t \geq 0\}_t \; \P\text{-supermartingale}, X_0 = 1, X_T \geq \frac 1 \alpha \text{ on } A\right\}~.
  \end{align*}
\end{proposition}
\begin{proof}
  The definition of $\P$ gives $\sup_{P\in\P} P(A) = \Epuseq \ones_A = \Egu \ones_A = \Pgu A$.
  From Corollary~\ref{cor:proto-ville}, we have
  \begin{align*}
    \Pgu A
    &= \inf\left\{ \alpha > 0 : \{X_t \geq 0 \}_t \text{ game-theoretic supermartingale}, X_0 = 1, X_T \geq \frac 1 \alpha \text{ on } A\right\}
    \\
    &\geq \inf\left\{ \alpha > 0 : \{X_t \geq 0\}_t \;\P\text{-supermartingale}, X_0 = 1, X_T \geq \frac 1 \alpha \text{ on } A\right\}~.
  \end{align*}
  For the other inequality, let $\{X_t\}_t$ be a nonnegative $\P$-supermartingale with $X_0=1$ and $X_T \geq 1/\alpha$ on $A$, for some $\alpha > 0$.
  Then Ville's inequality~\citep{ville1939etude,durrett2019probability} gives $P(A) \leq P(\exists t, \, X_t \geq 1/\alpha) \leq \alpha$.
\end{proof}

\section{Minimax theorems}
\label{sec:minimax-theorems}

In \S~\ref{sec:introduction}, we saw that statements in game-theoretic probability can be viewed as minimax theorems.
In \S~\ref{sec:prices-probability}, we broke this claim down further: the statement that game-theoretic and measure-theoretic expectations agree can be thought of as the price equality $\Epuseq X = \Egu X$.
When gambles are scalable, we have $\Epuseq X = \Epu X$ (Gambler can force sequential consistency) so all that remains is showing $\Epu X = \Egu X$ (minimax duality).
In this section, we develop minimax theorems to establish this latter equality.
We will apply these results to game-theoretic probability in \S~\ref{sec:translating-mtp-gtp}.

After reviewing existing minimax theorems in the literature, we prove several new ones, each borrowing ideas from nearby disciplines.
The first, a simple application of Sion's theorem, holds only for continuous variables $X$ on a compact space $\Omega$.
The second and third give sufficient conditions for non-sequential minimax duality.
The fourth (Theorem~\ref{thm:finite-sequential-minimax}) is our main result of this section, which shows how to convert non-sequential minimax duality to finite-time sequential minimax duality, via a backward induction argument.
We conclude in \S~\ref{sec:axioms-fa} with a connection to finitely additive probability measures, along with several counterexamples illustrating the challenges in proving a more general minimax theorem.
Still, we believe such a theorem is possible (\S~\ref{sec:future-work}).

\subsection{Existing general minimax theorems in game-theoretic probability}
\label{sec:existing-minimax}

Aside from results for specific gamble spaces, there are several general minimax results in the literature on game-theoretic probability.
The first, a game-theoretic Ville's Theorem, is essentially the case where there is a unique consistent probability measure.

For a measurable space $\Y$, \citet[\S~9.1]{shafer2019game} define a \emph{probability forecasting system} to be an indexed set $\{P^{(s)}\in\Delta(\Y)\}_{s\in\Y^*}$ of probability measures such that the map $s \mapsto P^{(s)}(A)$ is measurable for all measurable $A\subseteq\Y$.
The Ionescu-Tulcea Theorem ensures that $\{P^{(s)}\in\Delta(\Y)\}_{s\in\Y^*}$ uniquely determines a probability measure $P$ on $\Y^\infty$.

For a set of probability measures $\P \subseteq \Delta(\Omega)$, define
\begin{align}
  \Z_0(\P) := \{ Z:\Omega\to\reals \text{ measurable} \mid \E_P Z \leq 0 \; \forall P\in\P\}
  \label{eq:consistent-gambles}
\end{align}
to be the set of gambles consistent with $\P$.

\begin{theorem}[{{\citep[Theorem 9.3]{shafer2019game}}}]
  Let $\Y$ be a measurable space, $\{P^{(s)}\}_{s\in\Y^*}$ a probability forecasting system, and $P\in\Delta(\Y^\infty)$ the unique probability measure it induces.
  Consider the sequential gamble space $(\Y,\hat\Z,\infty)$, where $\hat\Z^{(s)} = \Z_0(\{P^{(s)}\})$.
  Then we have $\E_P X = \Epseq X = \Eg X$ for all bounded measurable $X$.
\end{theorem}

Composite results appear in the literature as well.
\citet{vovk2010prequential} give a minimax duality result for binary sequences, i.e., when $\Y=\{-1,1\}$ and $\P$ is now a set of probability measures.
This result is a special case of \citep[Theorem 9.7]{shafer2019game}, which establishes minimax duality when $\Y$ is a finite set and the gambles $\hat\Z^{(s)}$ satisfy a certain continuity property.

\subsection{Minimax duality for continuous variables}

Perhaps the best-known generalization of von Neumann's minimax theorem is due to Sion.

\begin{theorem}[Sion {\citep[Thm.~1]{sion1958general}}]\label{thm:sion}
  Let $\X$ be a convex subset of a linear topological space and $\Y$ a convex subset of a linear topological space, where at least one of $\X$ and $\Y$ is compact.  If $f:\X\times\Y\to\reals$ satisfies
  \begin{enumerate}
  \item $\forall x\in \X$, $f(x,\cdot)$ is upper semicontinuous and quasi-concave on $\Y$, and
  \item $\forall y\in \Y$, $f(\cdot,y)$ is lower semicontinuous and quasi-convex on $\X$,
  \end{enumerate}
  then $\min_{x\in \X}\sup_{y\in \Y} f(x,y)=\sup_{y\in \Y}\min_{x\in \X}f(x,y)$.
\end{theorem}

We will apply Sion's minimax theorem to show a minimax duality result for continuous variables.
Recall from Corollary~\ref{cor:continuous-gambles-wlog} that
$\Egu_\Z X = \Egu_{\dcl(\Z) \cap C(\Omega)} X$ when $X\in C(\Omega)$ and $\Egu_\Z X = \Egu_{\dcl(\Z) \cap C_b(\Omega)} X$ when $X \in C_b(\Omega)$, where $C(\Omega)$ and
$C_b(\Omega) \subseteq C(\Omega)$ are the sets of continuous and bounded continuous functions $\Omega\to\reals$, respectively.

\begin{lemma}\label{lem:convexity-dcl}
  If $\Z \subseteq (\Omega\to\extreals)$ is convex, then $\dcl(\Z)$ and $\mdcl(\Z)$ are convex.
\end{lemma}
\begin{proof}
  We have $\dcl(\Z) = (\Z + (-\infty,0]^\Omega) \cap \reals^\Omega$, both operations preserving convexity.
  Similarly, $\mdcl(\Z) = \dcl(\Z) \cap \X$ where $\X$ is the set of measurable functions to the reals, a convex set.
\end{proof}

\begin{theorem}\label{thm:minimax-continuous-compact-Omega}
  Let $(\Omega,\Z)$ be a gamble space such that $\Omega$ is a compact Hausdorff topological space and $\Z$ is convex.
  Then $\Egu X = \Epu X$ for all $X\in C_b(\Omega)$.
  
\end{theorem}
\begin{proof}
  We will apply Sion's minimax theorem to $f(Z,P) = \E_P[ X - Z ]$ for the spaces $\Z' = \mdcl(\Z) \cap C_b(\Omega)$ and $\P = \Delta(\Omega)$, respectively.
  As $X,Z\in C_b(\Omega)$, we indeed have $f(Z,P) \in \reals$ for all $Z\in\Z'$ and $P\in\P$.
  We equip $\Z'$ with the sup norm, and $\P$ with the weak topology.
  From Prokhorov's theorem, $\P$ is compact in the weak topology~\cite[\S~11]{dudley2002real}.
  Convexity of $\Z'$ follows from Lemma~\ref{lem:convexity-dcl} and convexity of $C_b(\Omega)$.
  Convexity of $\P$ is immediate.
  For fixed $P\in\P$, the map $Z\mapsto f(Z,P)$ is affine and 1--Lipschitz in the sup norm, hence convex and continuous.
  For fixed $Z\in\Z'$, the map $P\mapsto f(Z,P)$ is affine and continuous under weak convergence by the Portmanteau theorem \cite[Theorem 2.1]{billingsley1999convergence}.
  
  Theorem~\ref{thm:sion} (Sion) now gives
  \begin{align*}
    \Egu_{\Z'} X =
    \inf_{Z\in\Z'} \sup_{P\in\P} \E_P[ X - Z ]
    &=
      \sup_{P\in\P} \inf_{Z\in\Z'} \E_P[ X - Z ]
      = \Epu_{\Z'} X~.
  \end{align*}
  Corollary~\ref{cor:continuous-gambles-wlog} gives $\Egu_\Z X = \Egu_{\Z'} X$.
  Theorem~\ref{thm:chain-of-price-inequalities} and Proposition~\ref{prop:mdcl} give $\Epu_{\mdcl(\Z)} X \leq \Egu_{\mdcl(\Z)} X = \Egu_\Z X$.
  We have $\Epu_{\mdcl(\Z)} X = \Epu X$ by definition, and $\Epu_{\mdcl(\Z)} X \leq \Epu_{\Z'} X$ as the latter set of gambles is smaller.
  
\end{proof}

\subsection{Consistent and finitely generated gambles}
\label{sec:finitely-generated}

When one starts with a set of probability measures $\P\subseteq\Delta(\Omega)$, and takes all consistent gambles $\Z = \Z_0(\P)$, it is straightforward to verify that minimax duality holds for $(\Omega,\Z)$.

\begin{theorem}\label{thm:consistent-gamble-minimax}
  Let $\P \subseteq \Delta(\Omega)$ be nonempty, and $\Z = \Z_0(\P)$ as defined above.
  Then $\sup_{P\in\P} \E_P X = \Epucons X = \Epu X = \Egu X$ for all bounded measurable $X$.
\end{theorem}
\begin{proof}
  Let $c := \sup_{P\in\P} \E_P X$.
  As $X$ is bounded, we have $c \in \reals$.
  As $\E_P[X - c] \leq 0$ for all $P\in\P$, we have $X - c \in \Z$ by definition.
  Hence $\Egu X \leq c$.
  The rest follows from $\P \subseteq \Delta_0(\Z)$ and Theorem~\ref{thm:chain-of-price-inequalities}, which together give $\sup_{P\in\P} \E_P X \leq \Epucons X \leq \Epu X \leq \Egu X$.
\end{proof}

We now show how to view a recent characterization of e-variables in~\citet{larsson2025variables} as a minimax theorem.
To state their result and the resulting minimax theorem, let us first introduce and recall some definitions.

\begin{definition}
  $\Z$ is \emph{finitely generated} if $\Z = \{\sum_i \alpha_i g_i \mid \alpha_i \geq 0\}$ for some functions $\{g_i:\Omega\to\reals\}_{i=1}^k$, $k\in\N$.
\end{definition}
\begin{definition}
  $\Z$ has \emph{full support} if for all $\omega\in\Omega$, $\exists P\in\Delta_0(\Z)$ with $P(\{\omega\})>0$.
\end{definition}
Recall that the set of e-variables $\Evar(\P)$ for a set of probability measures $\P \subseteq \Delta(\Omega)$ is given by
\begin{align}
  \label{eq:e-var-def}
  \Evar(\P) &= \{E:\Omega\to[0,\infty] \text{ measurable } \mid \E_P E \leq 1\;\forall P\in \P\}~.
\end{align}
Let us say that an event holds $\P$-quasi-surely (or $\P$-q.s.) if it holds $P$-a.s.\ for all $P\in\P$.

\begin{theorem}[{{\cite[Theorem 9.2]{larsson2025variables}}}]
  Let $\Z$ be measurable and finitely generated, and set $\P = \Delta_0(\Z)$.
  Then
  \begin{align}
    \label{eq:e-var-finitely-gen}
    \Evar(\P) = \{E:\Omega\to[0,\infty] \text{ measurable } \mid \exists Z\in\Z \text{ s.t. } E \leq 1 + Z \text{ holds } \P\text{-q.s.}\}~.
  \end{align}
\end{theorem}

To leverage this result, let us first show how $\Evar$ relates to $\Z_0$.
In words, the gambles consistent with $\P$ which are also bounded below are precisely the set of e-variables up to a shift by 1 and an arbitrary finite scaling.

\begin{lemma}\label{lem:z0-evar}
  $\bb(\Z_0(\P)) = \{\alpha (E - 1) \mid E \in \Evar(\P), \alpha \geq 0\}$
\end{lemma}
\begin{proof}
  Let $E\in\Evar(\P)$, and take $Z = \alpha(E-1)$.
  We have $Z \in [-\alpha,\infty]$ and $\E_P Z = \alpha \E_P [E-1] \leq 0$.
  Thus $Z \in \bb(\Z_0(\P))$ as desired.
  
  For the reverse inclusion, let $Z\in\Z_0(\P)$.
  Let $\gamma = -1/\inf Z$ if $\inf Z < 0$, and $\gamma = 1$ otherwise.
  Then $\gamma > 0$, and $\gamma Z \geq -1$ as either $\inf Z \geq 0$, in which case $\gamma=1$ and $\inf \gamma Z = \inf Z \geq 0 > -1$, or $\inf Z \leq -1$, in which case we have $\inf \gamma Z = (-1/\inf Z) \inf Z = -1$.
  Now $E := 1 + \gamma Z$ has $E \in [0,\infty]$ and $\E_P E = 1 + \gamma E_P Z \leq 1$ for all $P\in\P$.
  Thus $E\in\Evar(\P)$, and $Z = (1/\gamma)(E - 1)$.
\end{proof}

\begin{theorem}\label{thm:finitely-generated-minimax}
  Let $\Z$ be finitely generated, measurable, and full support.
  Then $\Egu X = \Epu X = \Epucons X$ for all bounded measurable $X:\Omega\to\reals$.
\end{theorem}
\begin{proof}
  Let $\Z = \{\sum_i \alpha_i g_i \mid \alpha_i \geq 0\}$ for $g_i$ measurable.
  Let $\P = \Delta_0(\Z)$.
  As $\Z$ is full support, eq.~\eqref{eq:e-var-finitely-gen} becomes
  \begin{align}
    \label{eq:e-var-finitely-gen-full-supp}
    \Evar(\P)
    &= \{E:\Omega\to[0,\infty] \text{ measurable } \mid \exists Z\in\Z \text{ s.t. } E \leq 1 + Z\}
    \\
    &= \{E:\Omega\to[0,\infty] \mid E-1\in\mdcl(\Z)\}~.
  \end{align}
  From Lemma~\ref{lem:z0-evar} we have
  \begin{align*}
    \bb(\Z_0(\P))
    &= \{\alpha (E - 1) \mid E \in \Evar(\P), \alpha \geq 0\}
    \\
    &= \{\alpha (E - 1) \mid E:\Omega\to[0,\infty],\; E-1\in\mdcl(\Z),\; \alpha \geq 0\}
    \\
    &= \{Z':\Omega\to[-\alpha,\infty] \mid Z'\in\mdcl(\Z),\; \alpha \geq 0\}
    \\
    &= \bb(\mdcl(\Z))~.
  \end{align*}
  As $X$ is bounded below and measurable, from Propositions~\ref{prop:bounded-X-gambles} and~\ref{prop:mdcl} we have $\Egu X = \Egu_{\bb(\mdcl(\Z))} X = \Egu_{\bb(\Z_0(\P))} X = \Egu_{\Z_0(\P)} X$.
  The rest now follows from Theorem~\ref{thm:consistent-gamble-minimax}.
  
\end{proof}

Unlike the previous minimax theorems we have seen, Theorem~\ref{thm:finitely-generated-minimax} makes no assumptions about $\Omega$.
Moreover, many $\Z$ in specific settings are finitely generated.
Indeed, every non-sequential example we have considered thus far is of this form: for Example~\ref{ex:single-variance}, we had
$\Z = \{ \omega\mapsto \beta(\omega-c) + \alpha((\omega-c)^2-v) \mid \alpha,\beta\in\reals\}$, which corresponds to $g_1:\omega\mapsto\omega-c, g_2 = -g_1, g_3:\omega\mapsto (\omega-c)^2-v, g_4 = -g_3$.
See the discussion in that example for why we restrict constraints to be real-valued.

Theorem~\ref{thm:finitely-generated-minimax} does not generally apply in sequential settings; even the condition that $\E_P[ Y_t \mid Y_{1..t-1} ] = 0$ cannot be expressed via a finite set of constraints.
As we will see, however, we may extend Theorem~\ref{thm:finitely-generated-minimax} to finite time horizons using a backward induction argument reminiscent of Pascal's.

\subsection{Extending minimax duality to finite time via tower properties}
\label{sec:tower-properties-sequential-minimax}

We now show that minimax duality at each round implies minimax duality over any finite time horizon.
The proof uses a backward-induction technique from online machine learning~\citep{abernethy2009stochastic}, whose logic echoes the argument already employed by Pascal (Fig.~\ref{fig:pascal-fermat}).
For this backward indunction to proceed, we will make use of tower properties of conditional versions of $\Epu$ and $\Epuseq$.

\begin{definition}[Conditional measure-theoretic upper expectations]
  \label{def:conditional-epu}
  Let $(\Y,\hat\Z,T)$ be a sequential gamble space.
  Let $t \in \{1,\ldots,T-1\}$ and $s\in\Y^t$.
  Define $\hat\Z|_s \subseteq (\Y^{<T-t}\to\extreals)$ as in Definition~\ref{def:conditional-egu}.
  For any measurable $X:\Y^T\to\extreals$ , we define $\Epu[X\mid s] := \Epu[X(s \oplus \, \cdot \,)]$ and $\Epuseq[X\mid s] := \Epuseq[X(s \oplus \, \cdot \,)]$, both with respect to the sequential gamble space $(\Y,\hat\Z|_s,T-t)$.
\end{definition}
Note in particular that for $s\in\Y^{T-1}$ we have $\Epuseq[X\mid s] = \Epucons[X(s\oplus\,\cdot\,)]$, the latter on gamble space $(\Y,\hat\Z^{(s)})$.

To illustrate the intuition behind the tower properties, consider the case $T=2$.
For ease of exposition, suppose $(\Y,\hat\Z,T=2)$ is measurable (so we may dispense with $\mdcl$).
Informally, the basic logic is as follows
\begin{align*}
  \Epu[\Epu[ X \mid Y_1]]
  &= \sup_{P_1\in\Delta(\Y)} \inf_{Z_1\in\hat\Z^{(\epsilon)}} E_{P_1}\left[ \sup_{P_2\in\Delta(\Y)} \inf_{Z_2\in\hat\Z^{(Y_1)}} E_{P_2}[ X - Z_2 ] - Z_1\right]
  \\
  &= \sup_{P_1\in\Delta(\Y)} \inf_{Z_1\in\hat\Z^{(\epsilon)}} E_{P_1}\left[ \sup_{P_2\in\Delta(\Y)} \left(\inf_{Z_2\in\hat\Z^{(Y_1)}} E_{P_2}[ X - Z_2 ]\right) - Z_1\right]
  \\
  &\stackrel{(i)}{=}
    \sup_{\substack{P_1\in\Delta(\Y),\\
  \{P_2^{(y_1)}\in\Delta(\Y)\}_{y_1\in\Y}}} \inf_{\substack{Z_1\in\hat\Z^{(\epsilon)},\\
  \{Z_2^{(y_1)}\in\hat\Z^{(y_1)}\}_{y_1\in\Y}}}
  E_{P_1}\left[ E_{P_2^{(Y_1)}} \left[X - Z_2^{(Y_1)}\right] - Z_1 \right]
  \\
  &\stackrel{(ii)}{=}
    \sup_{P\in\Delta(\Y^2)} \inf_{Z\in\Z_2} \Ep_P\left[ X - Z \right]
  \\
  &= \Epu X~.
\end{align*}
Here equality (i) would follow from the observation that World's choice of $P_2$ does not depend on Gambler's choice of $Z_1$, as indicated by the parentheses.
We would need to establish several points to make the above rigorous.
First, the map $y_1 \mapsto \Epu[X\mid y_1]$ would need to be measurable.
Second, for equality (ii), we would need to show that $P_2^{(y_1)}$ and $Z_2^{(y_1)}$ could be chosen in a measurable way.

Assuming a tower property held for $\Epu$ as above, the basic logic one would hope for would be as follows.
Assume the sequential gamble space $(\Y,\hat\Z,T)$ (a) is disintegrable, (b) is sequentially measurable in the sense that the map $s\mapsto \Egu[ X \mid s ]$ is measurable for any bounded measurable $X$, and (c) satisfies minimax duality in each round, i.e., $\Egu_{\hat\Z^{(s)}} X = \Epu_{\hat\Z^{(s)}} X$ for all $s\in\Y^{<T}$ and all bounded measurable variables $X:\Y\to\reals$.
Then one could proceed by backward induction.
The base case $T=0$ would be trivial, and assuming $\Egu X = \Epu X$ for all bounded measurable $X:\Y^{T-1}\to\reals$, we would have
\begin{align*}
  \Egu X
  &= \Egu[ \Egu[ X \mid Y_{1..T-1} ] ] & \text{Proposition~\ref{prop:tower-property-egu}}
  \\
  &= \Egu[ \Epu[ X \mid Y_{1..T-1} ] ] & \text{per-round minimax assumption}
  \\
  &= \Epu[ \Epu[ X \mid Y_{1..T-1} ] ] & \text{sequential measurability, inductive hypothesis}
  \\
  &= \Epu X ~. & \text{tower property for $\Epu$}
\end{align*}

Unfortunately, this natural approach fails for subtle measurability reasons, as we now illustrate.
In particular, the sequential measurability assumption (b) above turns out to be overly restrictive, in that it rules out even very simple gamble spaces.
Moreover, even if sequential measurability held, additional structure is needed for equality (ii) above, for the measurable selection of $P_2$ and $Z_2$.

\begin{example}[Failing sequential measurability]
  \label{ex:seq-meas-fail}
  Take $\Y = [-1,1]$ equipped with the Borel $\sigma$-algebra.
  Suslin showed that there exist Borel sets $B \subseteq [0,1] \times [0,1] \subset \Y \times \Y$ such that $\pi_1(B) = \{ y_1 \in \Y \mid \exists y_2 \in \Y,\, (y_1,y_2) \in B\}$ is not Borel~\citep[Corollary 8.2.17]{cohn2013measure}.
  Let $\hat\Z = 0$, so that $\Epu[X \mid y_1] = \sup_{P\in\Delta(\Y)} E_P X(y_1,Y_2) = \sup_{y_2\in\Y} X(y_1,y_2)$.
  Letting $X = \ones_B$, we have $\Epu[X \mid y_1] = \ones_{\pi_1(B)}(y_1)$, which is not Borel measurable.

  The above behavior can also be seen in more familiar gamble spaces, like our running example $\hat\Z = \{y\mapsto \beta y \mid \beta\in\reals\}$ from \S~\ref{sec:motivating-example}.
  Let $C = \{(y_1,y_2)\in[0,1]\times[-1,1] \mid (y_1,|y_2|)\in B\} \subseteq \Y^2$
  and consider the Borel measurable variable $X = \ones_C$.
  First fix $y_1 \notin \pi_1(B)$.
  As $(y_1,|y_2|)\notin B$ for all $y_2\in\Y$, we have $X(y_1,\cdot) = 0$ and thus $\Epu[X \mid y_1] = 0$.
  Now fix $y_1 \in \pi_1(B)$.
  Let $a\in[0,1]$ such that $(y_1,a) \in B$.
  Then $(y_1,a),(y_1,-a) \in C$.
  Taking $P \in \Delta(\Y)$ with $P(\{a\}) = P(\{-a\}) = 1/2$, we have $P\in\Delta_0(\hat\Z)$ and thus $\Epu[X \mid y_1] \geq \Epucons X(y_1,\cdot) \geq \E_P X(y_1,Y_2) = \E_P \ones_C(y_1,Y_2) = 1$.
  As $X \leq 1$, we also have $\Epu[X \mid y_1] \leq 1$.
  Thus we once again have $\Epu[X \mid y_1] = \ones_{\pi_1(B)}(y_1)$, which is not Borel.
\end{example}

To circumvent these measurability issues, a now standard approach in control theory, dynamic programming, and mathematical finance relies on the theory of analytic sets and universal measurability~\citep{bertsekas1996stochastic,nutz2013constructing,bouchard2013arbitrage,bartl2020conditional}.
We will also focus on the stronger minimax duality $\Egu = \Epucons$ to simplify the exposition.

Let us briefly introduce some definitions; see     \citep[\S~7]{bertsekas1996stochastic}, \citep[\S~8.1-8.3]{cohn2013measure} for a thorough treatment.
Let $\Delta(\Y,\B)$ denote the set of Borel probability measures on $\Y$.
For any $P\in\Delta(\Y,\B)$, denote its completion by $\B_P$ and the corresponding measure $P^c$.
The \emph{universal $\sigma$-algebra} is given by $\U = \bigcap_{P\in\Delta(\Y,\B)} \B_P$.
Any $P\in\Delta(\Y,\B)$ admits a unique extension $\ext(P) := P^c|_\U \in\Delta(\Y,\U)$, the restriction to $\U$ of the completion of $P$.
(By definition, $\U \subseteq \B_P(\Y)$.)

For any $\U$-measurable $X:\Y\to\reals$, we define $E_P X := E_{\ext(P)} X$.
From this definition we may also extend the operators $\Epucons$ and $\Epuseq$, defined on $\Delta(\Y,\B)$ and $\Delta(\Y^T,\B)$, respectively, to $\U$-measurable $X$.
For example, we define $\Epuseq X := \sup_{P\in\Delta_0^T(\hat\Z)} E_{\ext{P}} X$.

A particularly useful class of $\U$-measurable functions is the class of upper semi-analytic functions.
A subset $A$ of a Polish space $\Y$ is \emph{analytic} if it is the image of a Polish space under a continuous function.
A function $f:\Y\to\extreals$ is \emph{upper semi-analytic} if its upper level sets are analytic, i.e., if $\{y \in \Y : f(y) > c\}$ is analytic $\forall c \in \reals$.
Every upper semi-analytic function is universally measurable~\citep[Proposition 7.42]{bertsekas1996stochastic}.
We thus have Borel measurable $\subseteq$ upper semi-analytic $\subseteq$ universally measurable, and generally these inclusions can be strict.
For example, the variable $\ones_{\pi_1(B)}$ from Example~\ref{ex:seq-meas-fail} is upper semi-analytic but not Borel.

\begin{condition}
  \label{cond:usa-minimax}
  Let $(\Y,\hat\Z,T)$ be a sequential gamble space for $T\in\N$.
  \begin{enumerate}
  \item $\Y$ is Polish.
  \item Borel, full support, finitely-generated, bounded below gambles: For some $k\in\N$ and a set of Borel measurable, bounded-below functions $\{g_{it} : \Y^t\times\Y \to \reals \}_{i\leq k, t< T}$ we have $\hat\Z^{(s)} = \{ y \mapsto \sum_{i=1}^k \alpha_i g_{i|s|}(s,y) \mid \alpha \in \reals^k_+ \}$ for all $s\in\Y^{<T}$.  Further assume that $\hat\Z^{(s)}$ is full support.
  \end{enumerate}
\end{condition}

The following statements reveal that upper semi-analytic functions are rich enough for the inductive argument to proceed.
Interestingly, the class of universally measurable functions is \emph{too} rich, in that there can exist universally measurable $X$ such that $s \mapsto \Epuseq[X \mid s]$ is not universally measurable  and thus part (2) below fails.

\begin{lemma}\label{lem:usa-facts}
  Let $(\Y,\hat\Z,T)$ satisfy Condition~\ref{cond:usa-minimax}.
  Let $t \in \{0,\ldots,T-1\}$.
  Let $X:\Y^T\to\reals$ be bounded and upper semi-analytic.
  Then the following hold.
  \begin{enumerate}
  \item Per-round minimax duality: $\Egu[X \mid s] = \Epuseq [X \mid s]$ when $t=T-1$.
  \item The map $\Y^t \ni s \mapsto \Epuseq[X \mid s]$ is upper semi-analytic.
  \item The tower property of $\Epuseq$: $\Epuseq[\Epuseq[X \mid Y_{1..t}]] = \Epuseq X$.
  \end{enumerate}
\end{lemma}
\begin{proof}
  We establish each statement in turn.
  \begin{enumerate}
  \item
    Let $X' = X(s\oplus\,\cdot\,):\Y\to\reals$.
    As $X'$ is upper semi-analytic, it is universally measurable~\citep[Proposition 7.42]{bertsekas1996stochastic}.
    Thus $\Egu X' = \Epucons_\U X'$ from Theorem~\ref{thm:finitely-generated-minimax} on gamble space $(\Y,\hat\Z^{(s)})$ where we equip $\Y$ with $\U(\Y)$.
    
    It remains to show $\Epucons_\B X' = \Epucons_\U X'$.
    First, some basic facts: each $Q\in\Delta(\Y,\U)$ can be restricted to $\B$, giving $\res(Q) := Q|_\B \in \Delta(\Y,\B)$.
    The operations $\res$ and $\ext$ are inverses: $\res \circ \ext (P) = (P^c|_\U)|_\B = P^c|_\B = P$, and as the completion of a probabily measure is uniquely determined~\citep{cohn2013measure}, we have $Q^c = (Q|_\B)^c$ and thus $\ext \circ \res (Q) = (Q|_\B)^c|_\U = Q^c|_\U = Q$.
    Thus, $\Delta(\Y,\U)$ and $\Delta(\Y,\B)$ are in bijection.
    Moreover, by definition of $\E_P X = \E_{\ext(P)} X$ for universally measurable $X$, this bijection preserves expectations of universally measurable $X$.
    We conclude $\Delta_0(\hat\Z^{(s)},\U) = \ext(\Delta_0(\hat\Z^{(s)},\B))$.
    Thus,
    \begin{align*}
      \Epucons_\U X'
      =
      \sup_{Q\in\Delta_0(\hat\Z^{(s)},\U)} \E_Q X'
      =
      \sup_{Q\in\ext(\Delta_0(\hat\Z^{(s)},\B))} \E_Q X'
      =
      \sup_{P\in\Delta_0(\hat\Z^{(s)},\B)} \E_{\ext(P)} X'
      =
      \Epucons_\B X'~.
    \end{align*}

  \item
    As detailed in \citet[Remark 2.8]{bartl2020conditional}, this statement follows if we can show that $\Delta_0^T(\hat\Z,\B)$ is an analytic subset of $\Delta(\Y^T,\B)$.
    In fact, it is Borel, as we briefly show.
    The constraint that $\E_P[g_{it}(Y_{1..t},Y_{t+1}) \mid Y_{1..t}] \leq 0$ hold $P$-a.s.\ can be equivalently written $\int_A g_{it}(y_{1..t+1}) dP(y) \leq 0$ for all $A\in\B(\Y^t)$, viewing each $A$ as a cylinder set $A \times \Y^{T-t} \subseteq \Y^T$.
    The map $f: P \mapsto \int_A g_{it}(y_{1..t+1}) dP(y)$ is Borel~\citep[Corollary 7.29.1]{bertsekas1996stochastic}, and thus the corresponding set $\P_{A,i,t} = \{P \in \Delta(\Y^T,\B) \mid \int_A g_{it}(y_{1..t+1}) dP(y) \leq 0\} = f^{-1}([-\infty,0])$ is Borel.
    As $\Y$ is Polish, each $\B(\Y^t)$ is countably generated, say by $\A_t \subset \B(\Y^t)$.
    Thus we may write $\Delta_0^T(\hat\Z,\B) = \bigcap_{t<T,i\leq k,A\in\A_t} \P_{A,i,t}$, which is Borel as the countable intersection of Borel sets.

  \item
    This statement is implied by the if direction of \citet[Theorem 1.2]{bartl2020conditional}, provided that we can show that, in their terminology, $\P := \Delta_0^T(\hat\Z,\B)$ is stable under pasting.
    (See also \citet[Theorem 2.3]{nutz2013constructing}.)

    Let $\Omega = \Y^T$ and $\G:=\sigma(Y_{1..t})$.
    For each $P\in\P$ let $P^{\G}(\omega)$ be a regular conditional probability given $\G$.
    Define the set-valued map
    \[
      \P_{\G}(\omega):=\{P^{\G}(\omega):\,P\in\P\}\subset\Delta(\Omega,\B).
    \]
    A $\G$-measurable kernel is a map
    \(
    R:\Omega\to \Delta(\Omega,\B)
    \)
    such that $\omega\mapsto R(\omega)(A)$ is $\G$-measurable for every $A\in\B$.
    For $Q\in\Delta(\Omega,\B)$ and $R$ a $\G$-measurable kernel, we say that a probability measure $P\in \Delta(\Omega,\B)$ equals $Q\otimes R$ if
    \begin{equation}\label{eq:otimes-def}
      \E_{\bar P}[X] = \E_Q\big[E_{R(\omega)}[X]\big]\qquad\text{for all upper semi-analytic } X\in\reals^\Omega~.
    \end{equation}
    
    To apply \citet[Theorem 1.2]{bartl2020conditional}, we must show
    \[
      \P = \P\otimes\P_{\G}
      :=\{Q\otimes R:\ Q\in\P,\ R(\cdot)\in\P_{\G}(\cdot)\ Q\text{-a.s.},\ R\ \G\text{-measurable}\}~.
    \]
    The inclusion $\P\subseteq\P\otimes\P_{\G}$ is immediate:
    given $P\in\P$, taking $R(\omega)=P^{\G}(\omega)\in\P_{\G}(\omega)$ gives $P=P\otimes R$.
    For the reverse inclusion $\P\otimes\P_{\G}\subseteq\P$, we use the fact that $\P = \Delta_0^T(\hat\Z)$ is defined by conditional constraints.
    Take $Q\in\P$ and a $\G$-measurable kernel $R$ with $R(\omega)\in\P_{\G}(\omega)$ $Q$-a.s., and let $P=Q\otimes R$.
    By definition, we have
    $Q(Y_{t'+1} \mid Y_{1..t'}) \in \Delta_0(\hat\Z^{(Y_{1..t'})})$ $Q$-a.s., and
    $R(Y_{1..t})(Y_{t'+1} \mid Y_{1..t'}) \in \Delta_0(\hat\Z^{(Y_{1..t'})})$ $R(Y_{1..t})$-a.s.\ for all $t' \geq t$.
    As $P(Y_{t'+1} \mid Y_{1..t'}) = Q(Y_{t'+1} \mid Y_{1..t'})$ for $t' < t$
    and $P(Y_{t'+1} \mid Y_{1..t'}) = R(Y_{1..t})(Y_{t'+1} \mid Y_{1..t'})$ for $t' \geq t$, both $P$-a.s.,
    we have $P \in \Delta_0^T(\hat\Z) = \P$ as well.

  \end{enumerate}
\end{proof}

\begin{theorem}\label{thm:finite-sequential-minimax}
  Let $(\Y,\hat\Z,T)$ be a sequential gamble space satisfying Condition~\ref{cond:usa-minimax}.
  Then we have $\Egu X = \Epuseq X$ for all bounded, Borel measurable $X:\Y^T\to\reals$.
\end{theorem}
\begin{proof}
  
  We first prove by induction that
  \begin{align}
    \label{eq:finite-minimax-stronger-induction}
    \text{$\Egu X = \Epuseq X$ for all bounded upper semi-analytic $X:\Y^T\to\reals$}~,
  \end{align}
  from which the result follows.
  For the base case $T=0$, as $\Y^T = \{\emptystring\}$, we have
  $\Egu X = X(\emptystring) = \Epuseq X$.
  Now consider $T\geq 1$, and suppose the statement holds for $T-1$.
  We have
  \begin{align*}
    \Egu X
    &= \Egu[ \Egu[ X \mid Y_{1..T-1} ] ] & \text{Proposition~\ref{prop:tower-property-egu}}
    \\
    &= \Egu[ \Epuseq[ X \mid Y_{1..T-1} ] ] & \text{Lemma~\ref{lem:usa-facts}(1)}
    \\
    &= \Epuseq[ \Epuseq[ X \mid Y_{1..T-1} ] ] & \text{Lemma~\ref{lem:usa-facts}(2), inductive hypothesis}
    \\
    &= \Epuseq X & \text{Lemma~\ref{lem:usa-facts}(3)}
  \end{align*}
\end{proof}

\subsection{When minimax duality fails: axioms and finite additivity}
\label{sec:axioms-fa}

Before applying our minimax theorems, it is instructive to ask when minimax duality \emph{fails}, i.e, when measure-theoretic probability and game-theoretic probability disagree.
More specifically, when do we have $\Epucons X < \Egu X$ (or $\Epuseq X < \Egu X$)?
As a starting point, we can look at properties of $\Epucons$, and ask when $\Egu $ shares those properties.

Let us consider the five axioms of \citet[\S~6.1]{shafer2019game}.
Let $f:\extreals^\Omega\to\extreals$ be an operator on variables.
Let $X,Y\in\extreals^\Omega$.
\begin{enumerate}[label=\textbf{E\arabic*.}]
\item $f(X+Y) \leq f(X)+f(Y)$.
\item $f(c X) = c f(X)$ for $c\in(0,\infty)$.
\item $f(X) \leq f(Y)$ when $X \leq Y$.
\item $f(c) = c$ for all $c\in\reals$.
\item
  \everymath{\displaystyle}
  If $X_1 \leq X_2 \leq \cdots \in [0,\infty]^\Omega$, 
      then $f(X_\infty) = \lim_{n\to\infty} f(X_n)$ where $X_\infty = \lim_{n\to\infty} X_n$.
\end{enumerate}
When $f$ satisfies E1--E5, Shafer and Vovk call $f$ an \emph{upper expectation}, and a \emph{broad-sense upper expectation} when it satisfies only E1--E4.
One can check that if $f(X) = \sup_{P\in\P} \E_P X$ for some set of probability measures $\P$, then $f$ does satisfy E1--E5 when restricting to bounded-below measurable variables.
(E5 follows from the monotone convergence theorem.)

From \S~\ref{sec:formalism-basic-facts}, we see that for arbitrage-free, positive-linear gamble spaces, $\Egu$ will satisfy E1--E4.
This result is essentially the same as \citet[Proposition 6.10, \S~6.5]{shafer2019game} showing, in their terminology, the equivalence of broad sense offers and broad sense upper expectations.
\begin{proposition}\label{prop:pos-linear-axioms}
  Let $(\Omega,\Z)$ be a gamble space.
  Then $\Egu$ satisfies E1--E4 if and only if the operator closure $\tilde\Z = \{Z\in\reals^\Omega \mid \Egu Z \leq 0\}$ is positive linear and arbitrage-free.
  
\end{proposition}
\begin{proof}
  For the if direction, the axioms follow from
  Propositions~\ref{prop:basic-facts-assumptions}, \ref{prop:basic-facts-homogeneity}, \ref{prop:basic-facts-no-assumptions}, and \ref{prop:basic-facts-zero}, respectively.
  For the only if direction, we have $\Egu_\Z = \Egu_{\tilde\Z}$ by Proposition~\ref{prop:dcl-wlog}.
  As $\Egu_{\tilde\Z} 0 = \Egu 0 = 0$, the gambles $\tilde\Z$ are arbitrage free from Proposition~\ref{prop:basic-facts-zero}.
  For positive linearity, let $Z_1,Z_2\in\tilde\Z$, and let $\alpha_1,\alpha_2\geq 0$.
  Then $\Egu Z_i \leq 0$, so E2 gives $\Egu \alpha_i Z_i \leq 0$, and thus E1 gives $\Egu(\alpha_1 Z_1 + \alpha_2 Z_2) \leq 0$.
  We conclude $\alpha_1 Z_1 + \alpha_2 Z_2 \in \tilde\Z$.  
\end{proof}

One need not look hard for an example that fails E5 however.

\begin{example}\label{ex:minimax-fail-omega-01}
  Take $\Omega=[0,1]$, $\Z = \{y\mapsto \beta y \mid \beta\in\reals\}$.
  As $\Z$ is positive linear and arbitrage-free, it satisfies E1--E4.
  It fails E5 however: take $X_n = \ones_{[1/n,1]}$, so $X_\infty = \ones_{(0,1]}$.
  Then $\Egu X_n = 0$ for all $n$, but $\Egu X_\infty = 1$ as any $Z\in\Z$ we have $\sup X_\infty - Z = 1$.
  (See \citet[Exercise 6.13]{shafer2019game}.)

  Since $\Egu$ fails E5, but $\Epucons$ always satisfies E5, there must be some $X$ where the two disagree.
  Indeed, we can take $X = X_\infty$.
  We have $\Delta_0(\Z) = \{\delta_0\}$, since $\E_P \beta Y \leq 0$ for all $\beta$ implies $\E_P Y = 0$ but $Y\geq 0$.
  Thus $\Egu X = 1 > 0 = X(0) = \E_{\delta_0} X = \Epucons X$.
\end{example}

\begin{remark}
  As minimax duality fails in Example~\ref{ex:minimax-fail-omega-01}, we can ask what assumption of Sion's theorem is violated.
  If one takes the weak topology on $\Delta(\Omega)$, we do have compactness, but the map $P \mapsto \E_P X$ is not upper semicontinuous as $X$ is not upper semicontinuous.
  If instead we used a topology making $P \mapsto \E_P X$ upper semicontinuous, such as the total variation topology, then $\Delta(\Omega)$ would no longer be compact, an issue since $\Z$ is not compact.

  Similarly, it may seem that Theorem~\ref{thm:finitely-generated-minimax} should apply here, as $\Z$ is finitely-generated, but $\Z$ does not have full support.
  In particular, it is not true that $\ones_{(0,1]} \in \dcl(\Z)$, even though $\ones_{(0,1]}$ is consistent with $\Delta_0(\Z)$.
\end{remark}

It may not come as a surprise that E1--E4 are not sufficient for $\Egu$ to be have like a measure-theoretic upper expectation, as E5 is reminiscent of countable (sub)additivity.
Indeed, \citet[Proposition 6.4]{shafer2019game} establishes countable subadditivity of $\Egu$ if it satisfies E1--E5.

To better understand these axioms, we now show a kind of converse: Axioms E1--E4 do characterize \emph{finitely-additive} measure-theoretic upper expectations.
The reason is that \emph{minimax duality always holds for finitely additive measures}, following results in the study of robust representations of financial risk measures (\S~\ref{sec:financial-risk-measures}).
The proofs appear in \S~\ref{sec:finit-addit-theory}.

Let $\Delta_f(\Omega)$ be the set of finitely additive probability measures on $\Omega$.
Given $\Z \subseteq \extreals^\Omega$, define $\Delta_{0,f}(\Z) := \{Q\in\Delta_f(\Omega) \mid \E_Q Z \leq 0 \;\forall Z\in\mdcl(\Z)\}$ to be the set of finitely additive consistent measures.

\begin{theorem}\label{thm:fin-add-minimax}
  Let $(\Omega,\Z)$ be an upward-scalable gamble space such that $\dcl(\Z)$ is convex.
  For all bounded measurable $X:\Omega\to\reals$, we have
  \begin{equation}
    \sup_{Q\in\Delta_f(\Omega)} \inf_{Z\in\mdcl(\Z)} \E_Q[ X - Z]
    =
    \inf_{Z\in\mdcl(\Z)} \sup_{Q\in\Delta_f(\Omega)} \E_Q[ X - Z]
    =
    \Egu X~.
  \end{equation}
\end{theorem}
\begin{corollary}\label{cor:fin-add-price-equality}
  Let $(\Omega,\Z)$ be an arbitrage-free positive linear gamble space.
  Then
  \begin{equation}
    \label{eq:fin-add-equality}
    \Egu X
    =
    \Epu^{0,f} X := 
    \sup_{Q\in\Delta_{0,f}(\Z)} \E_Q X~,
  \end{equation}
  for all $X\in\X_b$.
\end{corollary}

In essence, then, to have $\Egu X = \Epucons X$, one needs to show that $\Epu^{0,f} X = \Epucons X$.
In other words, one needs to ensure that no finitely additive measures $Q$ are ``exposed'' by $X$, meaning $\E_Q X > \sup_{P\in\Delta_0(\Z)} \E_P X$.
One way to eliminate all such $Q$ is to take $\Z$ to be sufficiently rich.
As we saw in Theorem~\ref{thm:consistent-gamble-minimax}, it certainly suffices to take all consistent gambles $\Z=\Z_0(\P)$ with respect to some $\P \subseteq \Delta(\Omega)$; in some sense this is the largest possible set of gambles with $\P \subseteq \Delta_0(\Z)$.
In general, one can remove gambles from this maximal $\Z$, but if one removes too many, even while preserving $\Delta_0(\Z)$, the set $\Delta_{0,f}$ may now include a finitely-additive $Q$ which is ``exposed'' by $X$, and thus for which minimax duality fails.
Indeed, in Example~\ref{ex:minimax-fail-omega-01}, it is precisely the gamble $\ones_{(0,1]}$ that is missing; even though including it would not change $\Delta_0(\Z)$, its presence enforces that consistent finitely additive probability measures are countably additive.

Similarly, one can interpret results like Theorem~\ref{thm:minimax-continuous-compact-Omega} that hold only for continuous $X$ as essentially relying on the fact that continuous functions are too coarse to ``expose'' finitely-additive $Q\in\Delta_{0,f}(\Z)$, even when $\Z$ has ``holes'' relative to $\Z_0$.

In light of this discussion, it may be tempting to think that Axiom E5 is enough for the countably additive representation to hold, i.e., for minimax duality.
Curiously, this is not the case: the following example violates minimax duality despite satisfying the continuity axiom E5.

\begin{example}
  \citet[Example 4.8]{delbaen2002coherent} gives an example $\Egu$ on $\Omega = [0,1]$ satisfying Axioms E1--E5 from \citet{shafer2019game} but every probability measure in $\Delta_{0,f}(\Z)$ is purely finitely additive.
  Hence $\Epucons = -\infty$.
  As a corollary, it is not possible to relax the assumption that $\Y$ be finite in \citet[Theorem 9.7]{shafer2019game} (discussed in \S~\ref{sec:existing-minimax}), even if $\Y$ is assumed to be compact.
  (In their notation, take $\Theta = \{\theta\}$ and $\Egu_\theta = \Egu$ above, and consider $X=0$.)
\end{example}

\begin{example}
  Let us continue the discussion from Example~\ref{ex:seq-consistent-but-no-consistent}, about the requirement that gambles be bounded below.
  In that example, we had $\Delta_0(\Z) = \emptyset$, but $\Epucons = \Epu$, so $\Epu = \Epucons = -\infty$.
  But $\Egu X$ is still finite for many choices of $X$, e.g., $X=1$, violating minimax duality.
  We can also take $\Y = \{-1,0,1\}$ so that $\Delta_0(\Z) \neq \emptyset$, as it contains the point measure on $y=(0,0,\ldots)$.
  Every outcome must have only finitely many nonzero elements almost surely for every $P\in\Delta_0(\Z)$, however.
  Letting $X(y) = \ones_{\{-1,1\}^\infty}$, we have
  $\Egu X = 1 > 0 = \Epucons X$.
\end{example}

\begin{example}
  Consider the sequential version of Example~\ref{ex:minimax-fail-omega-01}, the simple repeated gamble space $(\Y,\hat\Z,\infty)$ with $\Y=[0,1]$ and $\hat\Z = \{y\mapsto \beta y \mid \beta\in\reals\}$.
  \citet[Proposition 1.2]{shafer2019game} show that $\Egu \ones_{(\ALLN)^c} = 0 = \Epu \ones_{(\ALLN)^c}$, despite the fact that one generally does not have $\Egu X = \Epu X$ for $X\in\X_b$; indeed, the choice $X(y) = \ones\{y_1 > 0\}$ recovers the counterexample in Example~\ref{ex:minimax-fail-omega-01}.
  This observation suggests that minimax theorems may hold for indicators of tail events but not for all bounded measurable $X$.
\end{example}

\section{Translating between game-theoretic and measure-theoretic}
\label{sec:translating-mtp-gtp}

As alluded to at the top of \S~\ref{sec:prices-probability}, the price inequalities and equalities we have developed are the key to relating measure-theoretic and game-theoretic statements.
In particular, we will now see that every game-theoretic result implies a measure-theoretic one, and the two are equivalent when minimax duality holds.

\subsection{Statements as bounds on upper expectations}

It will be convenient to phrase measure-theoretic results as upper bounds on the expected value of some random variable $X$.
More precisely, for some set $\P$ of probability measures generated by some constraints, we would like to phrase results in the form: $\E_P X \leq c$ for all $P\in\P$.
Let us see a few examples of how this can be done.

\begin{example}[Chebyshev]
  \label{ex:chebyshev-conversion}
  The classic statement of Chebyshev's inequality states that, for any measurable random variable $Y$ with $\Ep Y = b$ and $\Ep (Y-b)^2 \leq \sigma^2$, and any $\alpha > 0$, we have $\Pp(|Y-b |\geq \alpha\sigma )\leq {\frac {1}{\alpha^{2}}}$.
  As is typically the case, the set $\P$ of probability measures is present but unacknowledged: it is the set of probability measures with $\E_P Y = b$ and $\E_P (Y-b)^2 \leq \sigma^2$.
  Actually, we can loosen this requirement to $\E_P (Y-b)^2 \leq \sigma^2$ where $b\in\reals$ is some constant, the mean of $Y$ or not.
  Thus, letting $\P = \{P \in \Delta(\reals) \mid \E_P (Y-b)^2 \leq \sigma^2\}$ and $X = (Y-b)^2$, Chebyshev's inequality is the statement: $\E_P X \leq \sigma^2$ for all $P\in\P$.
\end{example}  

Let us now revisit the motivating example from \S~\ref{sec:motivating-example}.

\begin{example}[Bounded law of large numbers]
  \label{ex:bounded-lln-conversion}
  One version of the bounded (martingale) law of large numbers states that for a martingale difference sequence $\{Y_t\}_t$ such that $|Y_t|\leq 1$ for all $t$, we have $\lim_{t\to\infty} \frac 1 t \sum_{i=1}^t Y_i = 0$ almost surely.
  The set $\P$ lurking behind this statement is the set of $P\in\Delta(\Omega)$ for which $\E_P[Y_t \mid Y_{1..t-1}] = 0$ holds $P$-a.s.\ for all $t$, i.e., the set of martingale measures.
  Letting $\ALLN = \{y \in [-1,1]^\infty : \lim_{n\to\infty} \frac 1 n \sum_{t=1}^n y_t = 0\}$ and $X = \ones_{(\ALLN)^c}$, the theorem can therefore be restated as $\E_P X \leq 0$ for all $P\in\P$.
\end{example}

Once measure-theoretic statements are of this form, the key to relating them to game-theoretic versions is to choose the gambles $\Z$ such that $\P=\Delta_0(\Z)$ or $\P = \Delta_0^T(\Z)$ for sequential settings.

\subsection{Game-theoretic to measure-theoretic (always)}
\label{sec:gtp-to-mtp-always}

The following two corollaries show that game-theoretic statements imply measure-theoretic ones.
In particular, if $\Egu X \leq c$, then by our price inequalities, we have $\Epucons X \leq \Egu X \leq c$ as well.
This statement in turn can be phrased as: $\E_P X \leq c$ for all $P\in\Delta_0(\Z)$.
For sequential settings, we would have $\Epuseq X \leq c$, which can be phrased as: $\E_P X \leq c$ for all $P\in\Delta_0^T(\Z)$.

\begin{corollary}
  \label{cor:gtp-to-mtp-single-outcome}
  Let $(\Omega,\Z)$ be a gamble space, and $X:\Omega\to\infreals$ measurable.
  Then
  \begin{align*}
    \Egu X \leq c
    & \quad\implies\quad
      \Epucons X \leq c
      \quad\iff\quad
      \E_P X \leq c \text{ for all } P\in\Delta_0(\Z)~.
  \end{align*}
\end{corollary}
\begin{proof}
  We have $\E_P X \leq \Egu X \leq c$ for any $P\in\Delta_0(\Z)$ by Theorem~\ref{thm:chain-of-price-inequalities}.
\end{proof}

\begin{corollary}
  \label{cor:gtp-to-mtp-sequential}
  Let $(\Y,\hat\Z,\infty)$ be a sequentially normalized gamble space, and let $X:\Y^T\to\infreals$ be measurable and bounded below.
  Then
  \begin{align*}
    \Egu X \leq c
    & \quad\implies\quad
      \Epuseq X \leq c
      \quad\iff\quad
      \E_P X \leq c \text{ for all } P\in\Delta_0^T(\hat\Z)~.
  \end{align*}
\end{corollary}
\begin{proof}
  We have $\E_P X \leq \Epuseq X \leq \Egu X \leq c$ for any $P\in\Delta_0^T(\hat\Z)$ by Corollary~\ref{cor:sequential-price-inequalities}.
\end{proof}

When applying Corollary~\ref{cor:gtp-to-mtp-sequential}, it will be useful to characterize $\Delta_0^T(\hat\Z)$ in terms of the per-round constraints given by the gambles, as follows.

\begin{lemma}\label{lem:per-round-constraints}
  Let $(\Y,\hat\Z,T)$ be a Borel sequential gamble space and let $t<T$.
  If $f:\Y^t\times\Y\to\reals$ is measurable and $f(s,\cdot)\in\hat\Z^{(s)}$ for all $s\in\Y^t$, then $\E_P[f(Y_{1..t},Y_{t+1}) \mid Y_{1..t}] \leq 0$ holds $P$-a.s.\ for all $P \in \Delta_0^T(\hat\Z)$.
\end{lemma}
\begin{proof}
  Let $P \in \Delta_0^T(\hat\Z)$.
  Let $A_t = \{s\in\Y^t \mid P(\cdot \mid Y_{1..t}=s) \in \Delta_0(\hat\Z^{(s)})\}$.
  For all $s\in A_t$, we thus have $\Delta_0(\hat\Z^{(s)})\neq \emptyset$.
  Applying Theorem~\ref{thm:chain-of-price-inequalities} on $\hat\Z^{(s)}$ gives $\E_P[f(Y_{1..t},Y_{t+1}) \mid Y_{1..t}=s] \leq \Egu[ f(s,\cdot) \mid s]$.
  Further, $\Egu[ f(s,\cdot) \mid s] \leq 0$ as $f(s,\cdot)$ is real-valued. 
  (See Example~\ref{ex:single-variance}.)
  By assumption, $Y_{1..t} \in A_t$ holds $P$-a.s., completing the proof.
\end{proof}

Let us return to Example~\ref{ex:bounded-lln-conversion}.
Here we have a simple repeated gamble space $(\Y,\hat\Z,\infty)$ with $y \mapsto y, y\mapsto -y \in \hat\Z$, and Lemma~\ref{lem:per-round-constraints} thus implies that $\{X_t\}_t$, $X_t = \sum_{i=1}^t Y_i$, is a measure-theoretic martingale with respect to sequentially consistent $P$.
(We could also first observe that it is a game-theoretic martingale and then apply Proposition~\ref{prop:supermartingale-to-measure}.)
The classic result of \citet[Proposition 1.2]{shafer2019game} states that $\Egu \ones_{(\ALLN)^c} = 0$, so in particular $\Egu \ones_{(\ALLN)^c} \leq 0$.
Corollary~\ref{cor:gtp-to-mtp-sequential} thus gives $\E_P \ones_{(\ALLN)^c} \leq 0$ for all $P\in\Delta_0^\infty(\hat\Z) = \P$.
As discussed in Example~\ref{ex:bounded-lln-conversion}, this statement is precisely the measure-theoretic version:
$P(\ALLN) = 1$ for all martingale measures $P$.

\begin{remark}
  \label{remark:gtp-to-mtp}
  It is interesting to contrast the above approach with, e.g., \citet[Proposition 9.17]{shafer2019game}, where measurability needed to be explicitly checked.
  The main workhorse in our case is Proposition~\ref{prop:mdcl}, which allows us to conclude the existence of measurable gambling strategies automatically.
  On the other hand, this existence result is not constructive, whereas \citet[Proposition 9.17, Corollary 9.18]{shafer2019game} constructs an explicit measurable strategy.
\end{remark}

\subsection{Measure-theoretic to game-theoretic (via minimax duality)}

When working with scalable gamble spaces, from Corollary~\ref{cor:gtp-to-mtp-sequential}, minimax duality allows us to translate from measure-theoretic statements to game-theoretic ones.
In the single-round case, given a statement of the form ``$\E_P X \leq c \text{ for all } P\in\P$'', we (a) find scalable gambles $\Z$ with $\Delta_0(\Z)\subseteq\P$, so that all consistent probability measures will satisfy the antecedent of the statement, and (b) show minimax duality $\Egu X = \Epu X$, so we may conclude $\Egu X = \Epucons X \leq c$.
For the sequential case, we proceed similarly, ensuring (a) $\Delta_0^T(\hat\Z)\subseteq\P$, often established from single-round price inequalities (Theorem~\ref{thm:chain-of-price-inequalities}), and (b) showing $\Epuseq X = \Egu X$ through a sequential minimax theorem such as Theorem~\ref{thm:finite-sequential-minimax}.

Let us begin with the Azuma--Hoeffding inequality.
\begin{theorem}[Measure-theoretic Azuma--Hoeffding]
  \label{thm:azuma-mtp}
  Suppose for $T\in\N$ the sequence $\{X_t\}_{t=1}^T$ is a (measure-theoretic) supermartingale.
  If $|X_t-X_{t-1}|\leq c_t$ a.s.\ for all $t\leq T$ for constants $c_t\geq 0$,
  then for all $\epsilon > 0$ we have
  \begin{align}
    \label{eq:azuma}
\Pp[X_T-X_0\geq \epsilon] \leq \exp \left(- {\epsilon ^{2} \over 2\sum _{t=1}^{T}c_{t}^{2}}\right).
  \end{align}
\end{theorem}

As before, implicit in this theorem statement is a set of probability measures $\P$ under which $|X_t-X_{t-1}|\leq c_t$ holds $P$-a.s.\ for all $P\in\P$, and under which $\{X_t\}_{t\leq T}$ is a $\P$-supermartingale.
Let us begin with the simplest game-theoretic version of this statement.

\begin{theorem}[Specific game-theoretic Azuma--Hoeffding]
  \label{thm:simple-azuma-gtp}
  Let $(\Y,\hat\Z,T)$ be the simple repeated gamble space with $\Y = [-1,1]$, $\hat\Z = \{y\mapsto \beta y \mid \beta \geq 0\}$, and some $T\in\N$.
  Then for all $\epsilon > 0$ we have
  \begin{align}
    \label{eq:simple-azuma-gtp}
    \Pgu\left[\sum_{t=1}^T Y_t \geq \epsilon\right] \leq \exp \left(- {\epsilon ^{2} \over 2T}\right)~.
  \end{align}
\end{theorem}
\begin{proof}
  We first establish that all the probability measures in $\Delta_0^T(\hat\Z)$ satisfy the conditions of Theorem~\ref{thm:azuma-mtp}.
  Letting $X_t = \sum_{i=1}^t Y_i$, we have $|X_t - X_{t-1}| = |Y_t| \leq 1$.
  Furthermore, the sequence $\{X_t\}_t$ is the capital of the gambling strategy that chooses $\beta=1$ in every round, and is thus a $P$-supermartingale for all $P\in\Delta_0^T(\hat\Z)$  by Propositions~\ref{prop:gambling-strategy-supermartingale} and~\ref{prop:supermartingale-to-measure}.

  It is straightforward to check that this gamble space satisfies Condition~\ref{cond:usa-minimax}: $\Y=[-1,1]$ is Polish, and $\hat\Z$ is finitely generated, full support, bounded below, and Borel measurable.
  Applying Theorem~\ref{thm:finite-sequential-minimax}, we have
  \begin{align*}
    \Pgu[X_T\geq \epsilon]
    &= \Egu \ones\{X_T\geq \epsilon\} \\
    &= \Epu_0^* \ones\{X_T\geq \epsilon\} & \text{Theorem~\ref{thm:finite-sequential-minimax}}\\
    &= \sup_{P\in\Delta_0^T(\hat\Z)} P(X_T\geq \epsilon) \\
    &\leq \exp \left(- {\epsilon ^{2} \over 2\sum _{t=1}^{T}c_{t}^{2}}\right)~, & \text{Theorem~\ref{thm:azuma-mtp}}
  \end{align*}
  as desired.
\end{proof}

As discussed above and in \S~\ref{sec:discussion-constructing-strategies}, one may wonder what \emph{strategy} $\psi$ Gambler can employ to achieve the guarantee~\eqref{eq:simple-azuma-gtp}.
As our results are nonconstructive, the strategy is not necessarily clear.
In this case, \citet[Theorem 3.6, Corollary 3.8]{shafer2019game} give an explicit multiplicative strategy: take $\psi(y_{1..t}) = y \mapsto \left((\alpha+Z^\psi_t(y_{1..t})) \frac{e^{\epsilon} - e^{-\epsilon}}{2e^{-2\epsilon^2}}\right) \, y$, where $\alpha = \exp(-\epsilon^2/2T)$ is the initial capital.
(Thus $\alpha+Z^\psi_t(y_{1..t})$ is the wealth at time $t$ of this strategy.)
Their proof first shows that $X_t(y_{1..t}) = \prod_{i=1}^t \exp(c y_t - 2 c^2)$ for all $c\geq 0$ is a game-theoretic supermartingale in this gamble space which replicates the (scaled) indicator variable of interest.

While Theorem~\ref{thm:simple-azuma-gtp} is elegant in its simplicity, one may wish to prove a more general theorem, for more general gamble spaces.
Moreover, one may wish to add the additional details and flexibility of the original theorem, like the constants $c_t$ and relaxing the boundedness condition to hold only almost surely.
Fortunately, the same approach goes through whenever one works with gamble spaces satisfying minimax duality.

\begin{theorem}[General game-theoretic Azuma--Hoeffding]
  \label{thm:azuma-gtp}
  Let $(\Y,\hat\Z,T)$ be a sequential gamble space satisfying Condition~\ref{cond:usa-minimax}.
  Let $\{X_t:\Y^T\to\extreals\}_{t\leq T}$ be a game-theoretic supermartingale such that $|X_t-X_{t-1}|\leq c_t$ g.t.a.s.\ for all $t\leq T$ for constants $c_t\geq 0$.
  Then for all $\epsilon > 0$ we have
  \begin{align}
    \label{eq:azuma-game-theoretic}
    \Pgu[X_T-X_0\geq \epsilon] \leq \exp \left(- {\epsilon ^{2} \over 2\sum _{t=1}^{T}c_{t}^{2}}\right)~.
  \end{align}
\end{theorem}
\begin{proof}
  From Proposition~\ref{prop:supermartingale-to-measure}, for any $P\in\Delta_0^T(\hat\Z)$, the sequence $\{X_t\}_t$ is a $P$-supermartingale.
  Similarly, Theorem~\ref{thm:chain-of-price-inequalities} gives $P(|X_t - X_{t-1}| \leq c_t) = 1$ for all $P\in\Delta_0(\Z_T)$, and thus for all $P\in\Delta_0^T(\hat\Z)$ from Proposition~\ref{prop:consistent-implies-seq-consistent}.
  Theorem~\ref{thm:azuma-mtp} thus applies to each $P\in\Delta_0^T(\hat\Z)$, giving
  \begin{align*}
    \Pgu[X_T-X_0\geq \epsilon]
    &= \Egu \ones\{X_T-X_0\geq \epsilon\}
    \\
    &= \Epu_0^* \ones\{X_T-X_0\geq \epsilon\}
    & \text{Theorem~\ref{thm:finite-sequential-minimax}}
    \\
    &= \sup_{P\in\Delta_0^T(\hat\Z)} P(X_T-X_0\geq \epsilon)
    \\
    &\leq \exp \left(- {\epsilon ^{2} \over 2\sum _{t=1}^{T}c_{t}^{2}}\right)~.
    & \text{Theorem~\ref{thm:azuma-mtp}}
    \\[-10pt]
  \end{align*}
\end{proof}

\subsection{Central limit theorems}

A standard form of the martingale central limit theorem (CLT) is as follows.

\begin{theorem}[Lindeberg martingale CLT]
  \label{thm:mtp-lindeberg-clt}
  Let $\{X_n, \mathcal{F}_n\}_{n \geq 0}$ be a square-integrable martingale with $X_0 = 0$, and let $\{Y_n = X_n - X_{n-1}, \mathcal{F}_n\}_{n \geq 1}$ be the martingale differences. Let the cumulative variance be given by
  \(
  V_n = \sum_{i=1}^n \mathbb{E}[Y_i^2 | \mathcal{F}_{i-1}]~.
  \)
  Suppose $V_n \to \infty$ a.s., and
  for all $\delta > 0$,
  \begin{equation}
    \label{eq:measure-theoretic-lindeberg}
    \frac{1}{V_n} \sum_{i=1}^n \mathbb{E}\left[Y_i^2 \ones\{|Y_i| > \delta \sqrt{V_n}\} \, \middle| \, \mathcal{F}_{i-1}\right] \xrightarrow{P} 0 \quad \text{as} \ n \to \infty~.    
  \end{equation}
  Then
  \(
  \dfrac{X_n}{\sqrt{V_n}} \xrightarrow{D} N(0, 1)
  \).
\end{theorem}

To translate Theorem~\ref{thm:mtp-lindeberg-clt} to a game-theoretic statement, we would first need suitable notions of convergence in distribution and probability.

\begin{definition}[Convergence in distribution (game-theoretic)]
  Let $(\Omega,\Z)$ be a gamble space and $\{X_n:\Omega\to\extreals\}_n$ a sequence of variables.
  We say the $X_n$ \emph{converge in distribution} to some $P\in\Delta(\reals)$ with CDF $F$, written $X_n \xrightarrow{D} P$, if for all $x\in\reals$ such that $F$ is continuous at $x$ we have $\lim_{n\to\infty} \Pgu[X_n \leq x] = \lim_{n\to\infty} \Pgl[X_n \leq x] = F(x)$.
\end{definition}

\begin{definition}[Convergence in probability (game-theoretic)]
  Let $(\Omega,\Z)$ be a gamble space and $\{X_n:\Omega\to\extreals\}_n$ a sequence of variables.
  We say the $X_n$ \emph{converge in probability} to some variable $X:\Omega\to\reals$, written $X_n \xrightarrow{P} X$, if for all $\epsilon>0$ we have
$\lim_{n\to\infty} \Pgu[|X_n-X|>\epsilon] =0$.
\end{definition}

One might hope to prove the following game-theoretic version of Theorem~\ref{thm:mtp-lindeberg-clt}.

\begin{nontheorem}[Game-theoretic CLT]
  \label{thm:potential-gtp-clt}
  Let $(\Y,\hat\Z,n)$  with $\Y\subseteq\reals$ satisfy Condition~\ref{cond:usa-minimax}.
  Let $V_n = \sum_{t=1}^n \Eg[Y_t^2 | Y_{1..t-1}]$ and $X_n = \sum_{t=1}^n Y_t$.
  Suppose that $\lim_{n\to\infty} V_n = \infty$ g.t.a.s.\ and
  for all $\delta > 0$,
  \begin{equation}
    \label{eq:game-theoretic-lindeberg}
    \frac{1}{V_n} \sum_{t=1}^n \Egu\left[Y_t^2 \ones\bigl\{|Y_t| > \delta \sqrt{V_n}\bigr\} \, \middle| \, Y_{1..t-1}\right] \xrightarrow{P} 0 \quad \text{as} \ n \to \infty~.
  \end{equation}
  Then
  \(
  \dfrac{X_n}{\sqrt{V_n}} \xrightarrow{D} N(0, 1)
  \).
\end{nontheorem}

It turns out that the game-theoretic conditions stated here do imply the measure-theoretic antecedent of Theorem~\ref{thm:mtp-lindeberg-clt} for any sequentially consistent $P$, the first key step of our approach.
Unfortunately, unlike in previous examples above, the measure-theoretic conclusion does not imply the game-theoretic conclusion.
The reason is that game-theoretic statements are \emph{composite}, describing an entire set of sequentially consistent probability measures rather than a single measure (\S~\ref{sec:intro-composite-ville}, \S~\ref{sec:chain-price-inequalities-sequential}), and the rate of convergence in the CLT is not uniform over this set.
In particular, one can take a family of examples
that illustrate the dependence of the convergence rate on the third absolute moment, as in the Berry--Esseen bound.
See \S~\ref{sec:app-example-non-uniform-clt} for a concrete example illustrating the failure of game-theoretic convergence in distribution.

To remedy the situation, we must impose further constraints, e.g.\ on the third absolute moment, so that the convergence rate is uniform.
At that point, we may as well introduce a quantitative version of the theorem.
Let $\Phi(x)$ denote the standard Normal CDF.

\begin{theorem}[\citetsafe{Theorem 3.7}{hall2014martingale}]
  \label{thm:hall-heyde-quantitative-clt}
  Let $X_t = \sum_{j=1}^t Y_j$, with $\mathcal{F}_t$ the $\sigma$-field generated by $Y_1, Y_2, \dots, Y_t$. Let 
  \(
  V_t = \sum_{j=1}^t \mathbb{E}[Y_j^2 \mid \mathcal{F}_{j-1}], 1 \leq t \leq n,
  \)
  and suppose that for constants $M$, $C$, and $D$ we have

  \begin{equation}
    \label{eq:quant-clt-bound-condition}
    \max_{t \leq n} |Y_t| \leq n^{-1/2} M \quad \text{a.s.},
  \end{equation}
  and
  \begin{equation}
    \label{eq:quant-clt-variance-condition}
    \mathbb{P}\left( \left| V_n - 1 \right| > 9M^2 D n^{-1/2} (\log n)^2 \right) \leq C n^{-1/4} / \log n~.
  \end{equation}
  Then for $n \geq 2$,
  \begin{equation}
    \label{eq:quant-clt-conclusion}
    \sup_{-\infty < x < \infty} \left| \mathbb{P}(X_n \leq x) - \Phi(x) \right| \leq (2 + C + 7 M D^{1/2}) n^{-1/4} \log n~.
  \end{equation}
\end{theorem}

\begin{theorem}[Quantitative game-theoretic CLT]
  \label{thm:gtp-quant-clt}
  Let $(\Y,\hat\Z,n)$
  with $\Y\subseteq\reals$ satisfy Condition~\ref{cond:usa-minimax}.
  Let $\overline V_n = \sum_{t=1}^n \Egu[Y_t^2 | Y_{1..t-1}]$ and
  $\underline V_n = \sum_{t=1}^n \Egl[Y_t^2 | Y_{1..t-1}]$ be real-valued, and let $X_n = \sum_{t=1}^n Y_t$.
  Suppose that for constants $M$, $C$, and $D$ we have

  \begin{equation}
    \label{eq:gtp-clt-bound-condition}
    \max_{t\leq n} |Y_t| \leq n^{-1/2} M \quad \text{g.t.a.s.}~,
  \end{equation}
  and
  \begin{equation}
    \label{eq:gtp-clt-variance-condition}
    \Pgu\left[ \max(1- \underline V_n, \overline V_n - 1) > 9M^2 D n^{-1/2} (\log n)^2 \right] \leq C n^{-1/4} / (\log n)~.
  \end{equation}
  Then for $n \geq 2$,
  \begin{equation}
    \label{eq:gtp-clt-conclusion}
    \sup_{-\infty < x < \infty} \max\left( \Pgu(X_n \leq x) - \Phi(x), \Phi(x) - \Pgl(X_n \leq x)\right) \leq (2 + C + 7 M D^{1/2}) n^{-1/4} \log n~.
  \end{equation}
\end{theorem}
\begin{proof}
  As before, we first show that the game-theoretic andecedent implies the measure-theoretic one for each $P\in\Delta_0^n(\hat\Z)$.
  We then show that the measure-thearetic consequent, together with minimax duality, implies the game-theoretic consequent.

  Let $P\in\Delta_0^n(\hat\Z)$.
  Define $V_n = \sum_{t=1}^n \E_P[Y_t^2 | Y_{1..t-1}]$.
  Theorem~\ref{thm:chain-of-price-inequalities} per-round gives $\underline V_n \leq V_n \leq \overline V_n$ for all $n$.
  The g.t.a.s.\ statement in eq.~\eqref{eq:gtp-clt-bound-condition}, translated to $\Pgl[\cdots] = 1$, implies eq.~\eqref{eq:quant-clt-bound-condition} via Corollary~\ref{cor:sequential-price-inequalities}.
  Similarly, letting $\epsilon,\delta$ be the relevant quantities in eq.~\eqref{eq:gtp-clt-variance-condition}, Corollary~\ref{cor:sequential-price-inequalities} gives
  \begin{align*}
    \delta
    &\geq \Pgu\left( \max(1- \underline V_n, \overline V_n - 1) > \epsilon \right)
    \\
    &\geq P\left( \max(1- \underline V_n, \overline V_n - 1) > \epsilon \right)
    \\
    &\geq P\left( |1-V_n| > \epsilon \right)~,
  \end{align*}
  as $|1-V_n| \leq \max(|1- \underline V_n|, |\overline V_n - 1|) = \max(1- \underline V_n, \overline V_n - 1)$.
  We thus have condition~\eqref{eq:quant-clt-variance-condition} as well.

  Applying Theorem~\ref{thm:hall-heyde-quantitative-clt} to $X_n$ and $P$, we thus have
  \begin{align*}
    \sup_{-\infty < x < \infty} \left| P(X_n \leq x) - \Phi(x) \right| \leq \gamma~,
  \end{align*}
  where $\gamma = (2 + C + 7 M D^{1/2}) n^{-1/4} \log n$.
  As this statement holds for all $P\in\Delta_0^n(\hat\Z)$, we have 
  \begin{align*}
    \sup_{-\infty < x < \infty} \left| \Epuseq \ones\{X_n \leq x\} - \Phi(x) \right| \leq \gamma
  \end{align*}
  and the same for $\Epl^*$.
  The variable $\ones\{X_n\leq x\}$ is bounded and measurable for all $x$.
  Theorem~\ref{thm:finite-sequential-minimax} gives $\Egu \ones\{X_n \leq x\} = \Epuseq \ones\{X_n \leq x\}$ and $\Egl \ones\{X_n \leq x\} = \Epl^* \ones\{X_n \leq x\}$, from which the result follows.
\end{proof}

Similar, more condensed game-theoretic CLTs appear in the literature, such as \citet[Proposition 2.10, Theorem 7.9]{shafer2019game}, and \citet[Theorem 7.1]{shafer2001probability}.

\subsection{Matrix concentration}

Let $\mathrm{SA}_d \subseteq \reals^{d\times d}$ be the set of self-adjoint matrices in dimension $d$.

\begin{theorem}[Matrix Azuma--Hoeffding, {\citep[Thm.~7.1]{tropp2012user}}]
  \label{thm:matrix_azuma}
  Consider a finite adapted sequence $\{X_k\} \subset \mathrm{SA}_d$, and a fixed sequence $\{A_k\} \subset \mathrm{SA}_d$, that satisfy
  \begin{align*}
    \E[ X_k \mid \F_{k-1} ] = 0 \quad \text{and} \quad X_k^2 \preceq A_k^2 \quad \text{a.s.}
  \end{align*}
  Let \(\sigma^2 := \left\| \sum_k A_k^2 \right\|\).
  Then, for all $t \geq 0$,
  \begin{align*}
    \mathbb{P}\left[ \lambda_{\max} \left( \sum_k X_k \right) \geq t \right] \leq d \cdot e^{-t^2 / 8\sigma^2}~.
  \end{align*}
\end{theorem}

\begin{theorem}[Game-theoretic matrix Azuma--Hoeffding]
  \label{thm:matrix_azuma-gtp}
  Let $(\Y,\hat\Z,T)$ be a sequential gamble space, $T\in\N$, satisfying Condition~\ref{cond:usa-minimax}.
  Consider sequences $\{X_k:\Y^k\to \mathrm{SA}_d\}$ and $\{A_k\in\mathrm{SA}_d\}$ satisfying
  \begin{align*}
    \Eg[ X_k \mid s ] = 0 \text{ for all } s\in\Y^{k-1}, \quad \text{and} \quad X_k^2 \preceq A_k^2 \quad \text{g.t.a.s.}
  \end{align*}
  Let \(\sigma^2 := \left\| \sum_k A_k^2 \right\|\).
  Then, for all $t \geq 0$,
  \begin{align*}
    \Pgu\left[ \lambda_{\max} \left( \sum_k X_k \right) \geq t \right] \leq d \cdot e^{-t^2 / 8\sigma^2}~.
  \end{align*}
\end{theorem}

As the proof of Theorem~\ref{thm:matrix_azuma} does not obviously follow from the scalar version~\citep[\S~7.3]{tropp2012user}, it is not immediately clear what the corresponding strategy would be.
Other matrix-valued martingale results can be extended similarly, such as Freedman's inequality~\citep{tropp2011freedman}.

\section{Discussion and Future Work}
\label{sec:discussion}

We have presented a new framing of game-theoretic probability based on gamble spaces, with several fundamental results, minimax theorems, connections to finite additivity, and a way to convert measure-theoretic results to game-theoretic ones.
Here we address the nonconstructive nature of many of our results, and some exciting directions for futurue work.

\subsection{Constructing strategies for game-theoretic results}
\label{sec:discussion-constructing-strategies}

We have seen that many measure-theoretic statements imply their game-theoretic counterparts.
By definition of a game-theoretic upper expectation, that in turn implies the existence of a sequence of increasingly efficient replication strategies.
In the case of almost sure events, these strategies risk arbitrarily low initial capital and become arbitrarily rich when the event does not occur.
Yet our results are largely nonconstructive: they establish the existence of such a sequence of strategies, but do not construct them.
The existing explicit constructions in the literature on game-theoretic probability are there therefore still of great interest.

In some cases, the measure-theoretic probability literature gives explicit constructions of supermartingales, e.g.\ that diverge to infinity when an event does not occur.
A prime example is Doob's martingale convergence theorem.
When combined with a minimax theorem stating that the game-theoretic version of such a result also holds, it may be tempting to say that the nonnegative supermartingales constructed in measure-theoretic proofs are therefore gambling strategies in the game-theoretic sense.

While this statement is morally true, the distinction is that measure-theoretic supermartingales are only required to satisfy the supermartingale condition almost surely, whereas game-theoretic supermartingales must satisfy it always.
In the case of a fixed reference measure, one can simply hedge the the event that the supermartingale condition is violated; see the proof of \citet[Theorem 9.3]{shafer2019game}.
When the set of sequentially consistent measures is larger, however, it is no longer clear when one can perform such a hedge, as these null events where the condition is violated can change for each measure; see open direction \#3 below.

Nonetheless, modulo this distinction between almost sure and always for supermartingales, our results highlight the distinction between a game-theoretic result and a game-theoretic \emph{proof}.
The general minimax theorem presented here implies that many measure-theoretic statements are true in the game-theoretic sense.
But it is still of interest to construct a gambling strategy explicitly, whether in the measure-theoretic or game-theoretic world.

\subsection{Open directions}
\label{sec:future-work}

Several fundamental questions remain.

\begin{enumerate}
\item More general minimax theorems.

  Conspicuously absent from our minimax results are the limit theorems, even the bounded law of large numbers that motivated our study in \S~\ref{sec:introduction}.
  These are infinite-time results, ruling out Theorem~\ref{thm:finite-sequential-minimax}, and variables in these settings are highly discontinuous indicators of sets, such as $X = \ones_{(\ALLN)^c}$ for the bounded strong law, ruling out Theorem~\ref{thm:minimax-continuous-compact-Omega}.
  While there are infinite-time minimax theorems in the literature, namely \citep[Theorem 9.7]{shafer2019game}, those results require $\Y$ to be a finite set, unlike the more natural $\Y = [-1,1]$ in \S~\ref{sec:motivating-example} or $\Y=\reals$.
  It seems plausible that one could extend Theorem~\ref{thm:finite-sequential-minimax} to $T = \infty$, at least for indicators of tail events.
  The fact that we have game-theoretic versions of many limit theorems suggests that a general result of this form may be possible.
  Finally, it would be interesting to develop minimax theorems that also apply to unbounded measurable $X$.

\item Game-theoretic e-processes.

  \citet{ruf2023composite} show a composite measure-theoretic version of Ville's Theorem which is characterized by e-processes, a generalization of nonnegative $\P$-supermartingales.
  As we saw, the type of minimax theorems we have developed imply measure-theoretic composite versions of Ville's Theorem as well (Proposition~\ref{prop:minimax-implies-composite-ville}).
  In the terminology of \citet{ruf2023composite}, the composite versions we recover are for the \emph{maximum likelihood} measure $\mu(A) = \sup_{P\in\Delta_0(\Z)} P(A)$.
  They go onto to give several examples where this $\mu$ does not seem to capture the right notion of ``testable'' replication.
  It is therefore important to understand why $\mu$ suffices for our setting---surely a consequence of the fact that we restrict to $\P$ such that gambling strategies give rise to $\P$-supermartingales---and how to develop a game-theoretic theory of e-processes that matches their \emph{inverse capital measure}~$\nu$.
  A related question is the connection between sequentially consistent $\P$ and the condition of fork-convexity~\citet{ramdas2022testing}.

\item Converting composite measure-theoretic supermartingales to game-theoretic.

  A key step in the argument of \citet[Theorem 9.3]{shafer2019game}, a game-theoretic version of Ville's Theorem, is a construction to convert a $P$-supermartingale $\{X_t\}_t$ to a game-theoretic supermartingale $\{X_t'\}_t$, essentially by additionally hedging the $P$-null event that $\E_P[X_{t+1} \mid X_{1..t}] > X_t$.
  An important open direction is to develop a similar construction for the composite case.
  When $\P$ is sequentially consistent, for example, can we show that a $\P$-supermartingale $\{X_t\}_t$ can be strengthened to a game-theoretic supermartingale?
  Doing so would mathematically justify the word ``game'' in game-theoretic statistics: it would allow us to interpret game-theoretic tests as bonafide strategies in a game.
  The key barrier is that the event $\E_P[X_{t+1} \mid X_{1..t}] > X_t$ could be different for each $P$, giving potentially uncountably many null sets to hedge.
  One would hope that sequentially consistent $P$ are rich enough to pool these nulls sets into a single null set (or even eliminate them altogether), allowing essentially the same techniques as in the singleton $P$ case.

\item Connections to and implications for online machine learning.

  The connections between adversarial online learning and game-theoretic probability are perhaps well understood at an arm's length, but only a handful of works go closer.
  As discussed in \S~\ref{sec:online-learning}, \citet{orabona2016coin} show how to use game-theoretic betting strategies to develop online convex optimization algorithms.
  We suspect there is much more to say.
  Another line of work \citet{rakhlin2014statistical,rakhlin2017equivalence,foster2018online,cover1965behavior}
  shows the agreement, in some particular online learning problems, between the adversarial regret bounds and the stochastic versions.
  These results echo the minimax duality and price equality we study here, though interestingly their settings typically are not scalable, suggesting that the conditions for price equality can be relaxed.
  
\end{enumerate}

\appendix

\section{Connection to financial risk measures}
\label{sec:financial-risk-measures}

\begin{definition}
  Let $\X \subseteq \reals^\Omega$ be a linear space of bounded functions containing the constant functions.
  A \emph{(financial) risk measure} on $\Omega$ is a function $\rho:\X\to\reals$ satisfying
  \begin{enumerate}
  \item $\rho(X+c) = \rho(X) - c$ for all $X\in\X, c\in\reals$ (translation'),
  \item $X\leq Y \implies \rho(X) \leq \rho(X)$ for all $X,Y\in\X$ (monotonicity').
  \end{enumerate}
  The \emph{acceptance set} of $\rho$ is the set $\A_\rho := \{ X\in\X \mid \rho(X) \leq 0 \}$.
\end{definition}

In light of Proposition~\ref{prop:basic-facts-no-assumptions}, the reader may immediately see the connection to game-theoretic upper expectations, which also satisfy translation and monotonicity.
Specifically, if $\Egu$ is real-valued on $\X$, then $\rho(X) = \Egu(-X)$ is a risk measure.
Its acceptance set is given by
\begin{equation}
  \label{eq:egu-acceptance-set}
  \A_\rho = -\{X \in \X \mid \Egu X \leq 0\}~,
\end{equation}
which is closely related to $\dcl(\Z)$.
Specifically, for any $Z\in\dcl(\Z)$, we have $Z-Z = 0$ as $Z$ is real-valued, giving $\Egu Z \leq 0$ when $0\in\Z$.
In this case, $\dcl(\Z) \cap \X \subseteq -\A_\rho$.

\section{Online learning algorithms as game-theoretic supermartingales}
\label{sec:online-learning-supermartingales}

To further illustrate how online learning algorithms can be expressed in terms of gamble spaces, and derived using the tools of game-theoretic probability, let us recall the ``relax and randomize'' framework of
\citet{rakhlin2012relax}.
We will first introduce their setting in their notation (apart from changing $x$ to $y$ and suppressing $T$ in $\Rel_T$), and then show how to recast their framework in terms of game-theoretic supermartingales.

\subsection{Original setting and sample results}

Let $\F$ be the learner's action set and $\Y$ World's action set.
At each round $t=1,\ldots,T$, the learner picks $f_t\in\F$, World picks $y_t\in\Y$, and the loss $\ell(f_t,y_t)\in\reals$ is incurred.
The regret is
\[
\Reg_T(\Alg) \;\;:=\;\; \sum_{t=1}^T \ell(f_t,y_t)\;-\;\inf_{f\in\F}\sum_{t=1}^T \ell(f,y_t).
\]
We write $\Delta(\F)$ and $\Delta(\Y)$ for distributions over $\F$ and $\Y$, respectively.

The (distributional) minimax value of the game is
\begin{equation}
\label{eq:value}
V_T(\F)
=\inf_{q_1\in\Delta(\F)}\;\sup_{y_1\in\Y}\;\Eop_{f_1\sim q_1}\cdots
\inf_{q_T\in\Delta(\F)}\;\sup_{y_T\in\Y}\;\Eop_{f_T\sim q_T}
\Bigg[\sum_{t=1}^T \ell(f_t,y_t)\;-\;\inf_{f\in\F}\sum_{t=1}^T \ell(f,y_t)\Bigg].
\end{equation}
Define the conditional value (for a prefix $y_{1..t}$):
\[
V_T(\F\,|\,y_{1..t})
:=\inf_{q\in\Delta(\F)}\;\sup_{y\in\Y}\left\{\Eop_{f\sim q}[\ell(f,y)]\;+\;V_T(\F\,|\,y_{1..t},y)\right\},
\]
with base case $V_T(\F\,|\,y_{1..T})=-\inf_{f\in\F}\sum_{t=1}^T \ell(f,y_t)$.
Note that $V_T(\F)=V_T(\F\,|\,\emptyset)$.
The minimax-optimal strategy at round $t$ is therefore to choose
\begin{equation}
\label{eq:bellman}
q_t \in \argmin_{q\in\Delta(\F)}\;\sup_{y\in\Y}\left\{\Eop_{f\sim q}[\ell(f,y)]\;+\;V_T(\F\,|\,y_{1..t-1},y)\right\}.
\end{equation}

A \emph{relaxation} is a sequence of functions $\Rel(\F\,|\,y_{1..t})$ for $t=0,\ldots,T$.
It is \emph{admissible} if for every $t\le T-1$,
\begin{align}
\label{eq:admissibility-supermartingale}
\Rel(\F\,|\,y_{1..t})
  &\;\ge\; \inf_{q\in\Delta(\F)}\;\sup_{y\in\Y}\left\{\Eop_{f\sim q}[\ell(f,y)] + \Rel(\F\,|\,y_{1..t} \oplus y)\right\},
\\
\label{eq:admissibility-replicate}
  \Rel(\F\,|\,y_{1..T})
  &\;\ge\;-\inf_{f\in\F}\sum_{t=1}^T \ell(f,y_t)~.
\end{align}
Given an admissible relaxation $\Rel$, an \emph{admissible algorithm} with respect to $\Rel$ is one choosing $q_t$ such that eq.~\eqref{eq:admissibility-supermartingale} holds for $q=q_t$, i.e., such that
\begin{align}
\label{eq:admissible-algorithm}
\Rel(\F\,|\,y_{1..t})
  &\;\ge\; \sup_{y\in\Y}\left\{\Eop_{f\sim q_t}[\ell(f,y)] + \Rel(\F\,|\,y_{1..t} \oplus y)\right\}~.
\end{align}
In particular, the \emph{meta algorithm}, which chooses the optimal action with respect to $\Rel$,
\begin{equation}
\label{eq:meta}
q_t \in \argmin_{q\in\Delta(\F)}\;\sup_{y\in\Y}\left\{\Eop_{f\sim q}[\ell(f,y)] + \Rel(\F\,|\,y_{1..t-1} \oplus y)\right\}~,
\end{equation}
is admissible with respect to $\Rel$.

\begin{proposition}[{{\citet[Prop.~1]{rakhlin2012relax}}}]
\label{prop:master}
For any admissible relaxation $\Rel$ and admissible algorithm with respect to $\Rel$, we have
\begin{equation}
\sum_{t=1}^T \Eop_{f_t\sim q_t}\,\ell(f_t,y_t)\;-\;\inf_{f\in\F}\sum_{t=1}^T \ell(f,y_t)
\;\;\le\;\; \Rel(\F)~,
\end{equation}
so that $\Eop[\Reg_T]\le \Rel(\F)$.
\end{proposition}

\citet{rakhlin2012relax} show that several standard algorithms are instances of this meta algorithm, where the relaxation $\Rel$ is a convenient upper bound derived using the \emph{sequential Rademacher complexity} $\mathfrak{R}_T$.
They give several examples, including the following.

\begin{example}[Exponential weights]
  For a finite class $\F$ and bounded losses $|\ell(f,y)|\le 1$, a convenient relaxation is
  \begin{equation}
    \label{eq:ew-relax}
    \Rel(\F\,|\,y_{1..t})
    \;=\;
    \inf_{\lambda>0}\Bigg\{\frac{1}{\lambda}\log
    \Bigg(\sum_{f\in\F} \exp\!\Bigg(-\lambda\sum_{s=1}^t \ell(f,y_s)\Bigg)\Bigg)
    \;+\;2\lambda\,(T-t)\Bigg\},
  \end{equation}
  which yields a parameter-free exponential-weights update via \eqref{eq:meta}, and the regret bound
  $\E[\Reg_T]\lesssim 2\sqrt{2T\log|\F|}$ when $\lambda$ is tuned online.
\end{example}

Mirror descent is another example.

\subsection{Algorithms and relaxations as game-theoretic supermartingales}

We can write the setting above as a simple repeated gamble space
$(\Y,\hat\Z,T)$, where
\begin{align*}
  \Omega &= \Y^T
  \\
  \hat\Z &= \left\{ Z_q:y\mapsto -\Eop_{f\sim q} \ell(f,y) \;\middle|\; q \in \Delta(\F) \right\}
  \\
  X(y) &=  -\inf_{f\in\F}\sum_{t=1}^T \ell(f,y_t)~.
\end{align*}
As discussed in \S~\ref{sec:online-learning}, the algorithm $\Alg$, which chooses $q_t$ as a function of $y_{1..t-1}$, is in bijection with the strategy $\psi$ in this gamble space.

Suppose $\Rel$ is an admissible relaxation.
Consider the sequence $\{X^\Rel_t\}_t$ given by $X^\Rel_t(y_{1..t}) = \Rel(\F \mid y_{1..t})$.
Then from the definitions above, this sequence is a game-theoretic supermartingale that replicates $X$:
\begin{align*}
  \Egu[X^\Rel_{t+1} \mid y_{1..t}]
  & := \inf_{Z\in\hat\Z} \; \sup_{y\in\Y} \; \left\{ X^\Rel_{t+1}(y_{1..t} \oplus y) - Z(y) \right\}
  \\
  & = \inf_{q\in\Delta(\F)} \; \sup_{y\in\Y} \; \left\{ \Rel(\F \mid y_{1..t} \oplus y) - \left(-\Eop_{f\sim q} \ell(f,y)\right) \right\}
  \\
  & = \inf_{q\in\Delta(\F)} \; \sup_{y\in\Y} \; \left\{ \Eop_{f\sim q} \ell(f,y) + \Rel(\F \mid y_{1..t} \oplus y) \right\}
  \\
  & \leq \Rel(\F \mid y_{1..t})
  \\
  & = X^\Rel_t(y_{1..t})~,
\end{align*}
where the inequality is 
eq.~\eqref{eq:admissibility-supermartingale}.
The inequality $X^\Rel_T \geq X$ is eq.~\eqref{eq:admissibility-replicate}.

Now suppose $\Alg$ is an admissible algorithm with respect to $\Rel$.
Define $\{X^\Alg_t\}_t$ by
\begin{align}
  X^\Alg_t(y_{1..t}) = - \sum_{i=1}^t \Eop_{f\sim q_i}[\ell(f,y_i)]~.\footnote{Comparing to eq.~\eqref{eq:online-learning-supermartingale} we have dropped the regret term.}
\end{align}
Admissibility implies that $\{X^\Alg_t\}_t$ is at least as efficient as $\{X^\Rel_t\}_t$, in the sense that $X^\Alg_{t+1} - X^\Alg_t \geq X^\Rel_{t+1} - X^\Rel_t$: the gamble $\Alg$ chooses on round $t$ is weakly better than demanded by $\Rel$.
To see this statement, we simply translate eq.~\eqref{eq:admissible-algorithm}:
\begin{align*}
\Rel(\F\,|\,y_{1..t})
  &\;\ge\; \sup_{y\in\Y}\left\{\Eop_{f\sim q_t}[\ell(f,y)] + \Rel(\F\,|\,y_{1..t} \oplus y)\right\}~.
  \\
  X^\Rel_t(y_{1..t}) & \;\geq\; \sup_{y\in\Y} \; \left\{ X^\Rel_{t+1}(y_{1..t} \oplus y) - Z(y) \right\}
  \\
  X^\Rel_t(y_{1..t}) & \;\geq\; \sup_{y\in\Y} \; \left\{ X^\Rel_{t+1}(y_{1..t} \oplus y) - \left(X^\Alg_{t+1}(y_{1..t} \oplus y) - X^\Alg_t(y_{1..t})\right) \right\}
  \\
  0 & \;\geq\; \sup_{y\in\Y} \; \left\{ X^\Rel_{t+1}(y_{1..t} \oplus y) - X^\Rel_t(y_{1..t}) - \left(X^\Alg_{t+1}(y_{1..t} \oplus y) - X^\Alg_t(y_{1..t})\right) \right\}
  \\
  0 & \;\geq\; X^\Rel_{t+1} - X^\Rel_t - \left(X^\Alg_{t+1} - X^\Alg_t\right)
  \\
  X^\Alg_{t+1} - X^\Alg_t & \;\geq\; X^\Rel_{t+1} - X^\Rel_t~.
\end{align*}
We therefore have $X^\Alg_T = X^\Alg_T - X^\Alg_0 \geq X^\Rel_T - X^\Rel_0 \geq X - X^\Rel_0$.
The regret bound $X - X^\Alg_T \leq X^\Rel_0$ of Proposition~\ref{prop:master} now follows.

In summary, the potential functions $\Rel$ discussed in \citet{rakhlin2012relax}, which are the building blocks for the design and analysis of many online learning algorithms, can be viewed as game-theoretic supermartingales which replicate the benchmark $X$.
Any algorithm at least as efficient as the potential function will enjoy the regret bound $X^\Rel_0$, the ``initial capital'' the potential sets aside to replicate $X$.
See also \citet{foster2018online}.

\section{Conditions for Lower $\leq$ Upper Expectations}
\label{sec:conditions-upper-lower}

As discussed in Remark~\ref{rem:upper-lower-inequality}, it is usually though not always the case that upper game-theoretic expectations are higher than lower ones.
In this section, we explore conditions for $\Egl \leq \Egu$, as well as the corresponding inequalities for the other prices.

We may begin with the most trivial: $\Epl^0 X \leq \Epucons X$ exactly when there exists a consistent measure, i.e., $\Delta_0(\Z) \neq \emptyset$.
In fact, in light of Theorem~\ref{thm:chain-of-price-inequalities}, the existence of a consistent measure is also sufficient for $\Egl X \leq \Egu X$ and $\Epl X \leq \Epu X$.
This condition is not necessary, however, as we explore below.

Before discussing these inequalities further, let us first see a series of examples illustrating when they fail.
One simple example is the one from Remark~\ref{rem:upper-expectation-philosophy}: $\Z = \{\omega \mapsto 1\}$, where $\Epu 0 = \Egu 0 = -1 < 1 = \Egl 0 = \Epl 0$.
The gamble in this example is clearly not arbitrage-free, however, leading to the question of whether that could be a sufficient condition.
In fact, it is not, as the next examples show.

In these next examples we will have $\Omega = \{1,2\}$, and for brevity will represent variables and gambles as vectors, so that $Z \in \extreals^2$, where $Z_i := Z(i)$ for $i\in\{1,2\}$.

\begin{figure}[t]
  \centering
  \includegraphics[width=0.4\textwidth]{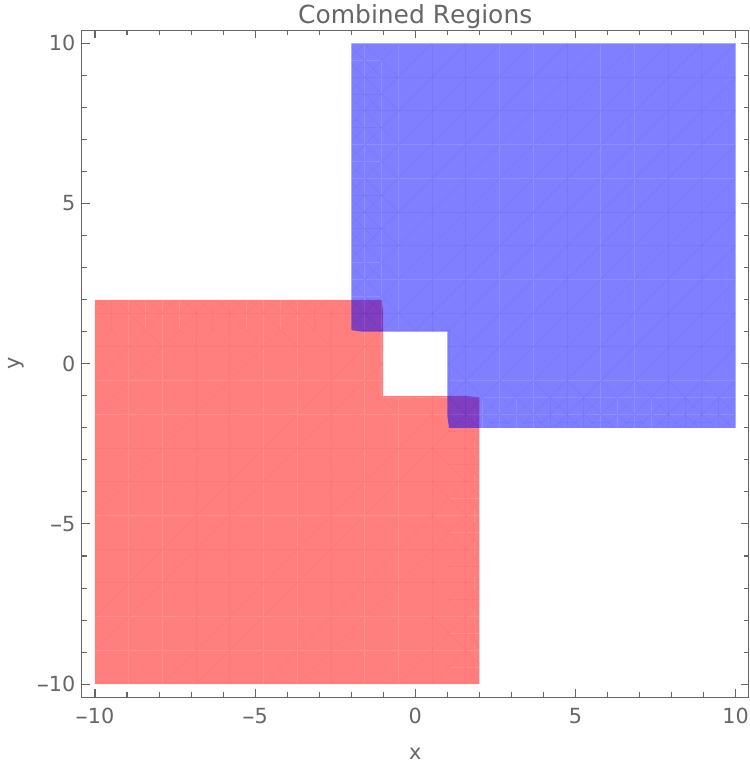}\hspace{0.05\textwidth}
  \includegraphics[width=0.4\textwidth]{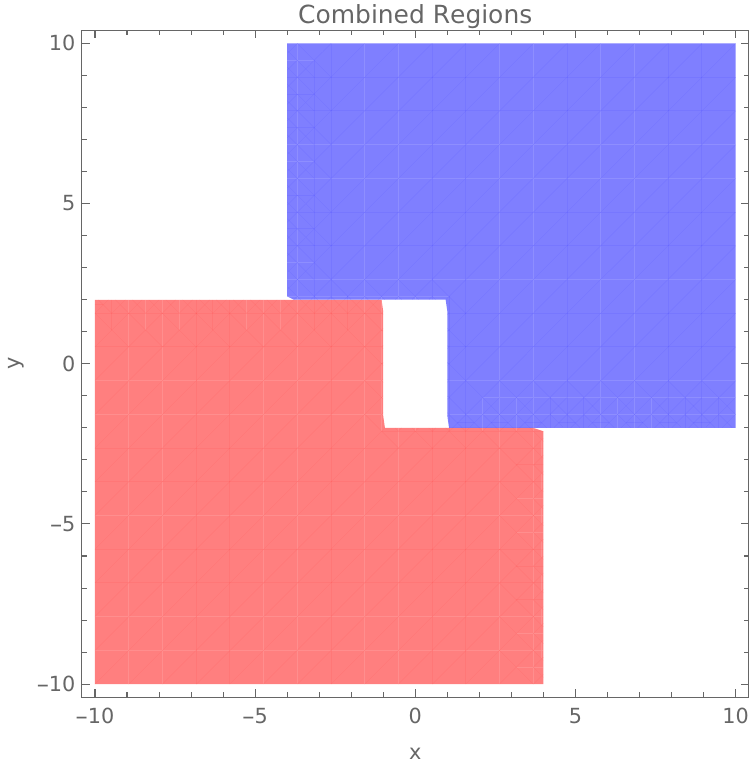}
  \caption{Visualizations of $\dcl(\Z)$ in red and $-\dcl(-\Z)$ in blue for the two gambles spaces in Examples~\ref{ex:egl-greater-egu} and~\ref{ex:egl-leq-egu-but-no-consistent}.  Overlap between these regions gives a variable $X$ for which $\Egl X > \Egu X$.  Consistent measures could be visualized as the normal cone to $\dcl(\Z)$, which is empty in both examples.}
  \label{fig:egl-egu-examples}
\end{figure}
\begin{example}[$\Egl > \Egu$]\label{ex:egl-greater-egu}
  Consider the gamble space $(\{1,2\},\Z)$ where $\Z = \{(2,-1),(-1,2)\}$.
  To clarify the notation once more, there are two gambles, $Z^{(1)}$ which awards $2$ upon outcome $1$ and $-1$ upon outcome $2$, and $Z^{(2)}$ with those values reversed.
  Clearly $\Z$ is arbitrage-free.
  Now $X = Z^{(1)} = (2,-1)$.
  We have $\Egu X = 0$, taking the first gamble $Z^{(1)}$.
  We also have $\Egu(-X) = -1$, taking the second gamble, as $(-X) - Z^{(2)} = (-2,1) - (-1,2) = (-1,-1)$.
  Thus $\Egl X = -\Egu(-X) = 1 > 0 = \Egu X$.
  See Fig.~\ref{fig:egl-egu-examples} for a visualization.
\end{example}

\begin{example}[$\Egl \leq \Egu$ but $\Delta_0(\Z) = \emptyset$]
  \label{ex:egl-leq-egu-but-no-consistent}
  Now consider the choice $\Z = \{(4,-2),(-1,2)\}$.
  One can check that here $\Egl X \leq \Egu X$ for all variables $X$.
  For example, $\Egl 0 = -1 \leq 1 = \Egu 0$.
  Observe that (like the previous example) there is no consistent measure, as $\E_P Z^{(1)} \leq 0 \implies P(1) \leq 1/3$ and $\E_P Z^{(2)} \leq 0 \implies P(2) \leq 1/3$.
  Now consider $\Epu 0 = \sup_P \min_{i\in\{1,2\}} \E_P[0 - Z^{(i)}]$.
  The optimal choice for World is $P = (\tfrac 4 9,\tfrac 5 9) \in \Delta(\Omega)$, as $\E_P Z^{(1)} = (16 - 10)/9 = \tfrac 2 3$ and $\E_P Z^{(2)} = (-4 + 10)/9 = \tfrac 2 3$.
  Thus $\Epu 0 = - \tfrac 2 3 < \tfrac 2 3 = \Epl 0$.
  In summary, we have $\Egl 0 < \Epu 0 < \Epl 0 < \Egu 0$.
\end{example}

Let us now turn to a characterization of when $\Egl \leq \Egu$.
The restriction that gambles not take on $-\infty$ is without loss of generality by Proposition~\ref{prop:eliminate-X-neg-infinity}.

\begin{proposition}\label{prop:egu-egl-inequality-char}
  Let $(\Omega,\Z)$ be a gamble space with $\Z \subseteq (\Omega\to\infreals)$.
  Then $\Egl X \leq \Egu X$ for all $X:\Omega\to\extreals$ if and only if $\Z + \Z$ is arbitrage-free, i.e., if $\inf (Z + Z') \leq 0$ for all $Z,Z' \in \Z$.
\end{proposition}
\begin{proof}
  First, let us argue that $\Z+\Z$ is arbitrage-free if and only if $\dcl(\Z) + \dcl(\Z)$ is arbitrage-free.
  Clearly if $\inf (\hat Z + \hat Z') > 0$ for $\hat Z,\hat Z'\in\dcl(\Z)$, then we have $Z,Z'\in\Z$ with $\hat Z \leq Z$, $\hat Z' \leq Z'$ and thus $\inf (Z + Z') \geq \inf (\hat Z + \hat Z') > 0$.
  Conversely, let $Z,Z'\in\Z$ with $\inf (Z + Z') > 0$.
  If both are finite-valued, we are done.
  Otherwise, define $\hat Z(\omega) =
  \begin{cases}
    1 - \min(0,Z'(\omega)) & Z(\omega) = \infty
    \\
    Z(\omega) & \text{otherwise}
  \end{cases}$, and similarly for $\hat Z'$.
  Then $\hat Z,\hat Z'\in\dcl(\Z)$.
  Furthermore, by construction, if either $Z$ or $Z'$ is infinite at $\omega$, the sum of $\hat Z$ and $\hat Z'$ at $\omega$ is at least one.
  Letting $A = \{\omega \in \Omega \mid Z(\omega),Z'(\omega) \in \reals\}$, we thus have $\hat Z + \hat Z' \geq (Z+Z')\ones_A + \ones_{A^c}$.
  Thus $\inf (\hat Z + \hat Z') \geq \min(\inf (Z + Z'), 1) > 0$.

  Suppose we have $\Egl X > \Egu X$ for some $X$.
  Then we must have $\Egu X < \infty$, which in turn implies $X < \infty$ by Proposition~\ref{prop:egu-x-infinite}.
  Similarly,  $\Egl X > -\infty$ implies $\Egu(-X) < \infty$ and thus $X > -\infty$.  We conclude $X:\Omega\to\reals$.
  By translation we have $\Egl X > 0 > \Egu X$ without loss of generality; if not, letting $c = \tfrac 1 2 (\Egl X + \Egu X)$, replace $X$ by $X - c$.
  By Proposition~\ref{prop:dcl-wlog}, we have $\Egl_{\dcl(\Z)} X, -\Egu_{\dcl(\Z)} X > 0$.
  Thus $\sup_{Z\in\dcl(\Z)} \inf (X + Z) = \Egl X > 0$ and $\sup_{Z\in\dcl(\Z)} \inf (-X +Z) = -\Egu X > 0$.
  By definition of supremum, there exist $Z,Z'\in\dcl(\Z)$ such that $0 < \inf (X + Z)$ and $0 < \inf (-X + Z')$.
  We conclude $0 < \inf (X + Z) + \inf (-X + Z') \leq \inf (Z + Z')$.

  For the converse, suppose $\inf (Z + Z') > 0$ for some $Z,Z'\in\dcl(\Z)$.
  Note that $\sup(-Z-Z') = -\inf (Z + Z') < 0$.
  Consider the choice $X = Z$.
  We have $\Egu Z \leq \sup Z - Z = 0$ as $Z:\Omega\to\reals$, and $\Egu(-Z) \leq \sup(-Z-Z') < 0$.
  Thus $\Egl Z = -\Egu(-Z) > 0 \geq \Egu Z$.
\end{proof}

As $Z,Z'\in\Z \implies Z+Z'\in\Z$ for positive linear gamble spaces, we have the following useful implication.
\begin{corollary}
  Let $(\Omega,\Z)$ be a positive-linear, arbitrage-free gamble space.
  Then $\Egl X \leq \Egu X$ for all $X:\Omega\to\extreals$.
\end{corollary}

Intuitively, we can leverage this characterization to understand when $\Epl \leq \Epu$ as well, since as discussed in Remark~\ref{remark:epu-convexify-intuition}, the operator $\Epu_\Z$ essentially convexifies $\Z$.
One may expect then that $\Epl \leq \Epu$ when $\Z' + \Z'$ is arbitrage-free where $\Z' = \overline \conv \, \Z$ is the closed convex hull of $\Z$ in a suitable topology.
(See \S~\ref{sec:structure-epu-finite-omega}.)

It is possible that a more direct condition could be obtained.

We can however verify that $\Delta_0(\Z) \neq \emptyset$ is not necessary for $\Epl \leq \Epu$.
Intuitively, it is \emph{almost} necessary, as seen by considering $X=0$.
Here $\Epu 0 = \sup_{P\in\Delta(\Omega)} \inf_{Z\in\Z} \E_P[0 - Z] = -\inf_{P\in\Delta(\Omega)} \sup_{Z\in\Z} \E_P Z$.
If $\Delta_0(\Z) = \emptyset$, then every $P$ yields some $Z$ with $\E_P Z > 0$.
Thus the only way for $\Epl 0 \leq \Epu 0$ is for every probability measure to give Gambler some positive profit but one that World can make arbitrarily (and uniformly) small.
The following example does precisely that.

\begin{example}[$\Epl \leq \Epu$ yet $\Delta_0(\Z) = \emptyset$]
  Let $\Omega = \N = \{1,2,3,\ldots\}$ and define $\Z = \{Z:n \mapsto 1/n\}$.
  Then clearly $\Delta_0(\Z) = \emptyset$, as we have $Z>0$ and thus $\E_P Z > 0$ for all $P\in\Delta(\Omega)$.
  Now suppose for a contradiction that we had some $X:\Omega\to\extreals$ for which $\Epl X > \Epu X$.
  As above, without loss of generality we may assume $\Epl X > c$, $-c > \Epu X$ for some $c>0$.
  Thus, as there is only one gamble, we have
  \begin{align*}
    \Epu X
    &=
      \sup_{P\in\Delta(\Omega)} \E_P[X - Z] < -c
      \implies
      \forall n\in\N,\; X(n) - 1/n < -c~,
    \\
    \Epl X
    &=
      \inf_{P\in\Delta(\Omega)} \E_P [X + Z] > c
      \implies
      \forall n\in\N,\; X(n) + 1/n > c~.
  \end{align*}
  A contradiction arises for $n > 1/c$.
  We conclude $\Epl X \leq \Epu X$ for all $X$.
  Note that the proposed condition above, that $\overline\conv\Z+\overline\conv\Z = \{2Z\}$ be arbitrage-free, is satisfied as $\inf Z = 0$.
\end{example}

\section{More General Sequential Gamble Spaces}
\label{sec:general-seq-gambles}

One may define a more general sequential gamble space as before but now with per-round outcomes being restricted to some $\Y^{(s)} \subseteq \Y$, as follows.
As before, let $\Y$ be a set, subsets of which will form the per-round outcomes.
Let time horizon $T\in\N\cup\{\infty\}$ be given.
Let $\Omega \subseteq \Y^T$ be the set of \emph{outcomes}.
Given $y\in\Y^t$, we write $y_{1..i} := (y_1,\ldots,y_i)$ to be the first $i$ elements of $y$, where $y_{1..0} := \emptystring$ is the empty sequence.
Letting $\Sc_t := \{y_{1..t} \mid y\in\Omega\}$ for all $t \leq T$, define the set of \emph{situations}
$\Sc = \bigcup_{t < T} \Sc_t$.
The outcomes $\Omega = \Sc_T$ are not elements of $\Sc$, though they can be thought of as ``terminal'' situations.
Given $s\in\Sc$, we define the \emph{per-round outcomes} $\Y^{(s)} = \{y\in\Y \mid s\oplus y\in\Sc_{|s|+1}\}$, where $\oplus$ denotes sequence concatenation and $t=|s|$ is the length of $s$.

For each situation $s\in\Sc$ we are given a set of \emph{per-round gambles} $\hat \Z^{(s)} \subseteq (\Y^{(s)}\to\extreals)$ available in that situation.
We let $\Psi$ be the set of all gambling strategies $\psi$ which map a $s\in\Sc$ to a choice of per-round gamble $\hat Z \in \hat \Z^{(s)}$; formally $\Psi = \{ \psi: \Sc \to \bigcup_{s\in\Sc} \hat \Z^{(s)} \mid \psi(s) \in \hat\Z^{(s)}\, \forall s\in\Sc\}$.
We can equivalently represent the available per-round gambles via a single set $\hat\Z \subseteq (\Sc\to\extreals)$, where $\hat Z\in\hat\Z$ is given by $\hat Z(y_{1..t}) = \psi(y_{1..t-1})(y_t)$.
Then a gambling strategy is given by an element $\hat Z\in\hat\Z$.
Defining $\hat Z_t := \hat Z|_{\Sc_t}$, the strategy is equivalently represented by the sequence $\{\hat Z_t\}_t$.
For each situation $s\in\Sc$, $t = |s|-1$, we recover $\hat \Z^{(s)} := \{\hat Z_t(s\oplus \cdot):\Y^{(s)}\to\extreals \mid \hat Z\in\hat\Z\}$.
As before, we will often write $\hat\Z \subseteq (\Sc\to\extreals)$ as shorthand for the indexed set $\{\hat \Z^{(s)}\}_s$.

The cumulative gamble $Z^\psi_t$ is given as before, in eq.~\eqref{eq:cumulative-gamble-finite} and eq.~\eqref{eq:cumulative-gamble-infinite}.

\begin{definition}[Generalized sequential gamble space]
  \label{def:gen-sequential-gamble-space}
  Let per-round outcomes $\Y$, time horizon $T\in\N\cup\{\infty\}$, and outcomes $\Omega \subseteq \Y^T$ be given, and define the set of situations $\Sc$ as above.
  Let per-round gambles $\hat\Z \subseteq (\Sc\to\extreals)$ be given, from which we can define $\{\hat\Z^{(s)}\}_{s\in\Sc}$ as above.
  Then we define the \emph{generalized sequential gamble space} $(\Omega,\hat\Z,T)$ to be the gamble space $(\Omega,\Z_T)$, as defined following eq.~\eqref{eq:cumulative-gamble-infinite}.
\end{definition}

Similar to before, we have global upper expectations on the gamble space $(\Omega,\Z_T)$, and the per-round upper expectation on $(\Y^{(s)},\hat\Z^{(s)})$, which gives rise to the conditional upper expectations.

\begin{definition}[Conditional game-theoretic upper expectation]
  \label{def:gen-conditional-egu}
  Let $(\Omega,\hat\Z,T=2)$ be a generalized sequential gamble space.
  For any $X:\Omega\to\extreals$ and $y_1\in\Y^{(\emptystring)}$, we define $\Egu[X\mid y_1] := \Egu[X(y_1,\cdot)]$ with respect to $(\Y^{(y_1)},\hat\Z^{(y_1)})$.
\end{definition}

While defined only for two-round gamble spaces, Definition~\ref{def:conditional-egu} applies much more broadly, since $\hat\Z^{(\emptystring)}$ and $\hat\Z^{(y_1)}$ can themselves be generalized sequential gamble spaces.
In that case, $y_1$ represents a sequence of outcomes---those seen thus far---and $y_2$ the sequence still to come.

Formally, given a generalized sequential gamble space $(\Omega,\hat\Z,T)$, and any $t < T$,
we can define the two-round generalized sequential gamble space $(\Omega',\hat\Z',T'=2)$ by combining the first $t$ rounds as the new round 1, and the remaining rounds as the new round 2.
Formally, for $s\in\Sc_t$,
define the conditional situations $\Sc|_s := \{s'_{t+1..|s'|} \mid s'\in\Sc, s'_{1..t}=s\}$ and conditional outcomes $\Omega|_s := \{s'_{t+1..T} \mid s'\in\Omega, s'_{1..t}=s\}$ to be the possible completions of $s$ in future rounds.
Then we let $\Omega' = \{(s,s') \mid s\in\Sc_t,s'\in\Omega|_s\}$,
$(\Y^{\prime(\emptystring)},\hat\Z^{\prime(\emptystring)}) = (\Sc_t,\Z_{t+1})$, and
$(\Y^{\prime(y_1')},\hat\Z^{\prime(y_1')})$ is the generalized sequential gamble space $(\Omega|_s,\{\hat\Z^{(s\oplus \hat s)}\}_{\hat s\in\Sc|_s},T-t)$.
Again, $\Egu[ X \mid s ]$ is simply $\Egu[X(s\oplus \cdot)]$ with respect to this latter gamble space for the last $T-t$ rounds.

The conditional upper expectation on a single following round reduces as before to the upper expectation on gamble space $(\Y^{(s)},\hat\Z^{(s)})$.
In this case, if $X:\Sc_t\to\extreals$ and $s\in\Sc_{t-1}$ we have
\begin{align}
  \label{eq:gen-conditional-expectation}
  \Egu[ X \mid s] &= \Egu_{\hat\Z^{(s)}} X(s \oplus \, \cdot) = \inf_{Z \in \hat\Z^{(s)}} \sup_{y\in\Y^{(s)}} X(s\oplus y) - Z(y)~.
\end{align}

\begin{definition}[Game-theoretic supermartingale]
  \label{def:gen-supermartingale}
  Let $(\Omega,\hat\Z,T)$ be a generalized sequential gamble space.
  A sequence $\{X_t:\Sc_t\to\extreals\}_{t\leq T}$ is a \emph{game-theoretic supermartingale} if for all $t< T$ and situations $s\in\Sc_t$, we have 
  \begin{align}
    \label{eq:gen-supermartingale-def}
    \Egu[ X_{t+1} \mid s] \leq X_t(s)~,
  \end{align}
  and a \emph{game-theoretic martingale} if $\Eg[ X_{t+1} \mid s] = X_t(s)$.
\end{definition}

Most results in \S~\ref{sec:definitions} for sequential gamble spaces readily extend to the generalized setting.
To illustrate, let us prove the generalization of Proposition~\ref{prop:gambling-strategy-supermartingale}.
\begin{proposition}\label{prop:gen-gambling-strategy-supermartingale}
  Let $(\Omega,\hat\Z,T)$ be a generalized sequential gamble space and $Z^\psi\in\Z_T$.
  If $\psi(s)(y)\in\reals$ for all $s\in\Sc,y\in\Y^{(s)}$, then $\{Z^\psi_t\}_t$ is a game-theoretic supermartingale.
\end{proposition}
\begin{proof}
  Let $t< T$ and $s \in \Sc_t$ be given.
  \begin{align*}
    \Egu[ Z^\psi_{t+1} \mid s]
    &= \inf_{\hat Z \in \hat\Z^{(s)}} \sup_{y\in\Y^{(s)}} Z^\psi_{t+1}(s\oplus y) - \hat Z(y)
    \\
    &= \inf_{\hat Z \in \hat\Z^{(s)}} \sup_{y\in\Y^{(s)}} \left(Z^\psi_t(s) + \psi(s)(y)\right) - \hat Z(y)
    \\
    &\leq \sup_{y\in\Y^{(s)}} \left(Z^\psi_t(s) + \psi(s)(y)\right) - \psi(s)(y)
    \\
    &= Z^\psi_t(s)~.
  \end{align*}
\end{proof}

\section{Structure of Prices and Minimax for Finite $\Omega$}
\label{sec:structure-epu-finite-omega}

A fruitful way to understand the various operators $\Egu, \Epu, \Epucons$ is how they operate on the set of gambles.
That is, given $\Z$, we could ask what ``effective'' gambles $\Z',\Z''$ would satisfy $\Epu_{\Z} = \Egu_{\Z'}$ and $\Epucons_{\Z} = \Egu_{\Z''}$.
It may not be clear that such $\Z',\Z''$ exist, but if they do, by the price inequalities of Theorem~\ref{thm:chain-of-price-inequalities}, we could presumably take $\Z \subseteq \Z' \subseteq \Z''$.

In this section we will show how to interpret these prices as certain closure operations on $\Z$, in the case when $\Omega$ to be a finite set.
In fact, even $\Egu$ itself can be thought of as a closure operation (Proposition~\ref{prop:finite-omega-egu}).
When speaking of closures of subsets of finite-valued variables and gambles, we refer to the standard topology on $\reals^\Omega$.

\begin{proposition}\label{prop:finite-omega-egu}
  Let $(\Omega,\Z)$ be a gamble space for a finite set $\Omega$.
  Then $\Egu_\Z = \Egu_{\cdcl{\Z}}$.
  Furthermore, $\cdcl{\Z} = \{Z\in\reals^\Omega \mid \Egu Z \leq 0\}$.
\end{proposition}
\begin{proof}
  We already have $\Egu_\Z = \Egu_{\dcl(\Z)} = \Egu_{\tilde\Z}$ from Proposition~\ref{prop:dcl-wlog}.
  It thus remains only to show $\tilde\Z = \cdcl{\Z}$.
  This statement follows from standard results in financial risk measures, that $\tilde\Z$ is the closure of $\dcl(\Z)$ with respect to the supremum norm; see \S~\ref{sec:financial-risk-measures} and \citet[Proposition 4.7(d)]{follmer2016stochastic}.
\end{proof}

\begin{proposition}\label{prop:finite-omega-epu}
  Let $(\Omega,\Z)$ be a gamble space for a finite set $\Omega$.
  Then $\Epu_\Z = \Egu_{\overline\conv\, \dcl(\Z)}$.
\end{proposition}
\begin{proof}
  Define a convex function $G:\Delta(\Omega)\to\extreals$ by $G(P) = \sup_{Z\in\mdcl(\Z)} \E_P Z$.
  We have
  \begin{align*}
    \Epu_\Z X
    &= \sup_{P\in\Delta(\Omega)} \inf_{Z\in\mdcl(\Z)} \E_P[X - Z]
    \\
    &= \sup_{P\in\Delta(\Omega)} \E_P X - G(P)
    \\
    &= G^*(X)~.
  \end{align*}
  As a dual of a convex function with domain $\Delta(\Omega)$, $G^*$ satisfies monotonicity and translation.
  Let $\Z' = \{Z\in\reals^\Omega \mid G^*(Z) \leq 0\}$.
  Then $\Egu_{\Z'} = G^*$.

  As $\Omega$ is a finite set, we have $\mdcl(\Z) = \dcl(\Z)$.
\end{proof}

\begin{corollary}
  Let $(\Omega,\Z)$ be a gamble space for a finite set $\Omega$.
  Then $\Epu X = \Egu X$ for all $X:\Omega\to\extreals$ if and only if $\cdcl{\Z}$ is convex.
\end{corollary}

Given a set $S \subseteq \reals^n$, we define its \emph{polar cone} to be the set $S^\circ := \left\{x\in \reals^n \mid y \cdot x \leq 0\;\; \forall y\in S\right\}$.

\begin{proposition}\label{prop:finite-omega-epu0}
  Let $(\Omega,\Z)$ be a gamble space for a finite set $\Omega$.
  Then $\Epucons = \Egu_{\dcl(\Z)^{\circ\circ}}$.
\end{proposition}

\begin{figure}[t]
  \centering
  \includegraphics[width=0.4\textwidth]{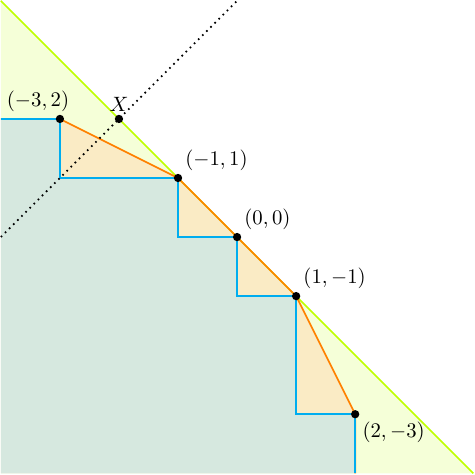}
  \caption{The effective gambles for $\Epucons,\Epu,\Egu$ for $\Z = \{(-3,2), (-1,1), (0,0), (1,-1), (2,-3)\}$, where $|\Omega|=2$ and we represent a gamble $Z:\Omega\to\extreals$ as an element of $\extreals^2$.  The depicted $X = (-2,2)$ has strictly increasing prices, as seen by the distance in the $(1,1)$ direction needed to travel to reach the effective gamble sets.
  }
  \label{fig:strict-price-inequalities}
\end{figure}

\section{Proofs for Undefined Expectations}
\label{sec:proofs-undef-expect}

Throughout, let $X^+ := \max(0,X)$ and $X^- := \max(-X,0)$ for any random variable $X$.
We define $\E X = \E X^+ - \E X^-$ when at most one is infinite, and say $\E X$ is undefined if $\E X^+ = \E X^- = \infty$.

To begin, let us consider a standard example of when expectations are undefined: the Cauchy distribution.
Here we will see that our definitions give $\Epucons X = \Egu X$ for a Cauchy random varible, validating our choice of defining undefined expectations to be infinite.
\begin{example}[Cauchy gambles]
  Let $\Omega = \reals$ and $X:\omega\mapsto\omega$.
  Let $Q$ be the Cauchy distribution, or indeed any probability measure where $\E_Q X$ is undefined.
  Define $\Z = \Z_0(\{Q\}) = \{Z \in \X \mid \E_Q Z \leq 0\}$ to be all gambles consistent with $Q$.
  Clearly $\Delta_0(\Z) = \{Q\}$, and thus $\Epucons X = \E_Q X = \infty$, as $\E_Q X$ is undefined and thus defined to be $\infty$ in the context of $\Epucons$.
  Let us verify that $\Egu X = \infty$ as well.
  Suppose for a contradiction that $\Egu X < c$ for some $c \geq 0$.
  By definition, for some $Z \in \Z$ we have $Z + c \geq X$.
  Since $\E_Q X^+ = \infty$ and $Z^+ \geq (X-c)^+ \geq X^+ - c$, we have $\E_Q Z^+ \geq \infty - c = \infty$ as well.
  We conclude that $\E_Q Z$ is either undefined or infinite, contradicting $\E_Q Z \leq 0$.
  Thus, $\Epu X = \infty = \Epucons X$.
\end{example}

We now turn to specific statements, and fill in omitted details involving undefined expectations.

In Theorem~\ref{thm:chain-of-price-inequalities}, we must revisit two statements.
First, the statement $\E_P[X - Z] \leq \sup_{\omega\in\Omega} X(\omega) - Z(\omega)$ for any $P\in\Delta(\Omega)$ continues to hold even when $\E_P[X-Z]$ is undefined, since in that case we must have $\E_P (X-Z)^+ = \infty$, which implies $\sup X-Z = \sup (X-Z)^+ = \infty$.
Second, we need to show
\begin{align*}
  \sup_{P\in\Delta_0(\Z)} \inf_{Z\in\mdcl(\Z)} \E_P[X - Z]
  &\geq \sup_{P\in\Delta_0(\Z)} \E_P X~,
\end{align*}
even when $\E_P[X-Z]$ or $\E_P X$ may be undefined (and thus defined to be $\infty$).
It suffices to show that if $\E_P X^+ = \infty$ for any $P\in\Delta_0(\Z)$, then $\E_P (X-Z)^+ = \infty$ for all $Z\in\Z$.
We have $(X-Z)^+ \geq (X-Z)^+\ones_{X\geq 0} \geq X^+ - Z\ones_{X\geq 0} \geq X^+ - Z^+\ones_{X\geq 0} \geq X^+ - Z^+$.
As $P\in\Delta_0(\Z)$, we have $0 \geq \E_P Z = \E_P Z^+ - \E_P Z^-$ and thus $\E_P Z^+ =: c < \infty$.
We conclude $\E_P(X-Z)^+ \geq \E_P[X^+ - Z^+] = \E_P X^+ - \E_P Z^+ = \infty - c = \infty$.

\section{Finitely additive theory}
\label{sec:finit-addit-theory}

This section contains proofs from \S~\ref{sec:axioms-fa} and further discussion from \S~\ref{sec:translating-mtp-gtp}.

\subsection{Finitely additive prices}
\label{sec:fa-prices}

Given a measurable space $(\Omega,\Sigma)$, let $\Delta_f(\Omega)$ be the set of finitely-additive probability measures.
Given $\Z \subseteq \extreals^\Omega$, define $\Delta_{0,f}(\Z) := \{Q\in\Delta_f(\Omega) \mid \E_Q Z \leq 0 \;\forall Z\in\mdcl(\Z)\}$.

\begin{theorem}\label{thm:fin-add-price-chain}
  Let $(\Omega,\Z)$ be a gamble space.
  For all measurable $X:\Omega\to\extreals$, we have
  \begin{equation}
    \label{eq:fin-add-minimax-price-chain}
    \sup_{Q\in\Delta_{0,f}(\Z)} \E_Q X
    \leq
    \sup_{Q\in\Delta_f(\Omega)} \inf_{Z\in\mdcl(\Z)} \E_Q[ X - Z]
    \leq
    \inf_{Z\in\mdcl(\Z)} \sup_{Q\in\Delta_f(\Omega)} \E_Q[ X - Z]
    =
    \Egu X~.
  \end{equation}
  If $X$ is additionally bounded, the first inequality is an equality when  $\Z$ is upward-scalable and contains zero, and the second when $\dcl(\Z)$ is convex.
\end{theorem}
\begin{proof}
  The chain of inequalities follows by the same argument as Theorem~\ref{thm:chain-of-price-inequalities}.
  Similarly, the first equality follows from the same argument as Theorem~\ref{thm:ep0-equals-ep}.
  For the second equality, we appeal to \citet[Theorem 4.16]{follmer2016stochastic}; see \S~\ref{sec:financial-risk-measures}.
  
\end{proof}

\begin{corollary}\label{cor:fin-add-coherent-risk}
  Let $(\Omega,\Z)$ be an upward-scalable gamble space with $0\in\Z$ and $\dcl(\Z)$ convex.
  Then
  \begin{equation}
    \label{eq:fin-add-equality-2}
    \Egu X
    =
    \sup_{Q\in\Delta_{0,f}(\Z)} \E_Q X~,
  \end{equation}
  for all $X\in\X_b$.
\end{corollary}

\subsection{Proving finitely additive results}
\label{sec:proving-fa-results}

Every result of the form $\Egu X \leq c$ gives a result for all finitely-additive probability measures $\Delta_f(\Omega)$, namely:
\begin{align}
  \label{eq:finitely-additive-generic-result}
  \Egu X \leq c
  & \quad\implies\quad
    \E_Q X \leq c \text{ for all } Q\in\Delta_{0,f}(\Z)~.
\end{align}
In other words, $\Egu X \leq c$ implies the statement: If $Q\in\Delta_f(\Omega)$ satisfies $\E_Q Z \leq 0$ for all $Z\in\mdcl(\Z)$, then $\E_Q X \leq c$.
In fact, when $\Z$ is upward-scalable, contains zero, and $\dcl(\Z)$ is convex, this statement is \emph{equivalent} to $\Egu X \leq c$.

\begin{remark}\label{rem:fa-seq-consistent-fails}
  In many settings of interest, one would like a sequential version of eq.~\eqref{eq:finitely-additive-generic-result}, such as
  \begin{align*}
    \Egu X \leq c
    & \quad\centernot\implies\quad
      \E_Q X \leq c \text{ for all } Q\in\Delta_{0,f}^T(\hat\Z)~,
  \end{align*}
  for some suitable notion $\Delta_{0,f}^T$ of sequentially consistent finitely additive measures.
  Unfortunately, this conversion is not always possible, and indeed a suitable definition of $\Delta_{0,f}^T$ not always available.
  One way to see the challenge is that one would need a proof that $\Delta_{0,f}^T(\hat\Z) \subseteq \Delta_0(\Z')$, which appears to rely on Fatou's lemma, which is not true in general for finitely-additive measures.

  As a simple example of what can go wrong, consider the simple repeated gamble space $(\Y,\hat\Z,\infty)$ with $\Y=\{-1,1\}$, $\hat\Z = \{y\mapsto \beta y \mid \beta\in\reals\}$ the usual bets on $\Y$.
  Then there is a $Q\in\Delta_{0,f}^T(\hat\Z)$ which violates the LLN.
  So letting $A = \{y \mid \lim_{t\to\infty} \frac 1 t \sum_{i=1}^t y_i = 0\}$, we have $\Egu \ones_{A^c} = 0$, yet $\E_Q \ones_{A^c} > 0$.
\end{remark}

\begin{example}[Non-trivial finitely additive result]
  Remark~\ref{rem:fa-seq-consistent-fails} notwithstanding, there are game-theoretic results which imply nontrivial statemenst about finitely additive measures.
  For example, \citet[Proposition 1.2]{shafer2019game} with $m_n = -1$ essentially shows a game-theoretic LLN in the sequential version of Example~\ref{ex:minimax-fail-omega-01}.
  Despite the fact that minimax fails in general, e.g.\ for $X(y) = \ones\{y_1 > 0\}$, the counterexample from Example~\ref{ex:minimax-fail-omega-01}, minimax duality does hold for $X = \ones_{(\ALLN)^c}$.
  Hence eq.~\eqref{eq:finitely-additive-generic-result} implies a LLN result for e.g.\ the i.i.d.\ $Q\in\Delta_f([0,1])$ which has $Q([0,a]) = 1$ for all $a > 0$ but $Q(\{0\}) = 0$.  
\end{example}

\section{Example illustrating non-uniform convergence in the Lindeberg CLT}
\label{sec:app-example-non-uniform-clt}

For all $k \in \N$, let $P_k\in\Delta(\Y^\infty)$ be such that the $Y_t$ are independent, with distributions for all $t$ given by $P_k(Y_t = -1) = 1/(k+1)$, $P_k(Y_t = k) = 1/(k(k+1))$, and $P_k(Y_t = 0)$ having the remaining mass.
One can check that $\E_{P_k} Y_t = 0$, $\E_{P_k} Y_t^2 = 1$ (giving $V_n = n$), and the Lindeberg condition is trivially satisfied as $P_k[ |Y_t| > \delta \sqrt{n} ] = 0$ for $n > (k/\delta)^2$.
Let us consider the resulting CDFs at $x=0$.
For $P_k$, we have $P_k[ \tfrac 1 {\sqrt{n}} X_n \leq 0 ] \geq (1 - P_k(Y_1 = k))^n = (1 - 1/(k(k+1)))^n$.
As this expression limits to $1$ as $k\to\infty$, for all $n \in \N$, we can find $k$ such that $P_k[ \tfrac 1 {\sqrt{n}} X_n \leq 0 ] > 0.9$.
Yet $\Phi(0) = 1/2$.

By construction, for all $k$ we have $P_k \in \Delta_0^\infty(\{\hat\Z_t\})$ for the sequential gamble space in Potential Theorem~\ref{thm:potential-gtp-clt}.
From Corollary~\ref{cor:sequential-price-inequalities}, we thus have $\Pgu[\tfrac 1 {\sqrt{n}} X_n \leq 0] \geq \Epuseq \ones\{\tfrac 1 {\sqrt{n}} X_n \leq 0\} = \sup_{P\in\Delta_0^\infty(\{\hat\Z_t\})} P(\tfrac 1 {\sqrt{n}} X_n \leq 0) \geq 0.9$.
As the above holds for all $n$, game-theoretic convergence in distribution fails.

\bibliographystyle{plainnat}
\bibliography{refs}

\end{document}